\documentclass[12pt]{article}

\usepackage[utf8]{inputenc} 
\usepackage[T1]{fontenc}
\usepackage[top=2cm, bottom=2cm, left=2cm, right=2cm]{geometry}

\usepackage{amsthm}
\usepackage{amsmath}
\usepackage{amssymb}
\usepackage{mathrsfs}
\usepackage{graphicx}
\usepackage{wrapfig}
\usepackage{makeidx}
\usepackage{dsfont}
\usepackage{stmaryrd}
\usepackage{accents}

\usepackage{tikz}
\usetikzlibrary{trees,positioning,decorations.pathreplacing}

\theoremstyle{plain}
\newtheorem{defi}{Definition}
\newtheorem{prop}{Proposition}[section]
\newtheorem{thm}{Theorem}[section]
\newtheorem{lem}{Lemma}[section]

\theoremstyle{definition}
\newtheorem{rem}{Remark}[section]

\renewcommand{\l}{\left}
\renewcommand{\r}{\right}
\newcommand{\llb}{\llbracket}
\newcommand{\rrb}{\rrbracket}
\newcommand{\lf}{\lfloor}
\newcommand{\rf}{\rfloor}
\renewcommand{\o}[1]{\overline{#1}}
\renewcommand{\epsilon}{\varepsilon}
\newcommand{\n}{{(n)}}
\newcommand{\m}{{(m)}}
\newcommand{\diag}{\mathrm{diag}}

\makeatletter
\DeclareRobustCommand{\p}[1]{%
  \mathpalette\do@cev{#1}%
}
\newcommand{\do@cev}[2]{%
  \fix@cev{#1}{+}%
  \reflectbox{$\m@th#1\vec{\reflectbox{$\fix@cev{#1}{-}\m@th#1#2\fix@cev{#1}{+}$}}$}%
  \fix@cev{#1}{-}%
}
\newcommand{\fix@cev}[2]{%
  \ifx#1\displaystyle
    \mkern#23mu
  \else
    \ifx#1\textstyle
      \mkern#23mu
    \else
      \ifx#1\scriptstyle
        \mkern#22mu
      \else
        \mkern#22mu
      \fi
    \fi
  \fi
}

\makeatother

\newcommand{\R}{\mathbb{R}}

\newcommand{\Z}{\mathbb{Z}}
\newcommand{\N}{\mathbb{N}}

\newcommand{\F}{\mathcal{F}}
\newcommand{\M}{\mathcal{M}}
\newcommand{\G}{\mathcal{G}}
\renewcommand{\H}{\mathcal{H}}
\newcommand{\calP}{\mathcal{P}}

\newcommand{\calT}{\mathcal{T}}
\newcommand{\calD}{\mathcal{D}}
\newcommand{\calR}{\mathcal{R}}
\newcommand{\calJ}{\mathcal{J}}

\newcommand{\E}{\mathbb{E}}
\renewcommand{\P}{\mathbb{P}}
\newcommand{\ind}{\mathds{1}}

\newcommand{\Tt}{\tilde{T}}
\newcommand{\Gh}{\hat{G}}
\newcommand{\Gc}{\check{G}}
\newcommand{\Gt}{\tilde{G}}
\newcommand{\gc}{\check{g}}

\newcommand{\Hc}{\check{H}}
\newcommand{\Wt}{\tilde{W}}
\newcommand{\Wh}{\hat{W}}
\newcommand{\Wc}{\check{W}}
\newcommand{\Fh}{\hat{F}}
\newcommand{\Ft}{\tilde{F}}
\newcommand{\Cc}{\check{C}}

\title{Representations of the Vertex Reinforced Jump Process as a mixture of Markov processes on $\Z^d$ and infinite trees}

\date{March 23, 2019}
\author{Thomas \sc{Gerard}}

\begin{document}
\maketitle

\begin{abstract}
This paper concerns the Vertex Reinforced Jump Process (VRJP) and its representations as a Markov process in random environment. In \cite{SabTar}, it was shown that the VRJP on finite graphs, under a certain time rescaling, has the distribution of a mixture of Markov jump processes. This representation was extended to infinite graphs in \cite{SabZen}, by introducing a random potential $\beta$. In this paper, we show that all possible representations of the VRJP as a mixture of Markov processes can be expressed in a similar form as in \cite{SabZen}, using the random field $\beta$ and  harmonic functions for an associated operator $H_\beta$. This allows to show that the VRJP on $\Z^d$ (with certain initial conditions) has a unique representation, by proving that an associated Martin boundary is trivial. Moreover, on infinite trees, we construct a family of representations, that are all different when the VRJP is transient and the tree is $d$-regular (with $d\geq 3$).
\end{abstract}

\section{Introduction}

This paper concerns the Vertex Reinforced Jump Process (or VRJP) on infinite graphs and its representations as a Markov process in a random environment. In particular, we are interested in knowing if the VRJP admits several different representations, and what form they can take.

Let $\G=(V,E)$ be a non-directed locally finite graph, \textit{i.e.} each vertex $i\in V$ has finite degree. For $i,j\in V$, we write $i\sim j$ if $i$ and $j$ are neighbours, \textit{i.e.} if $\{i,j\}\in E$. We endow $\G$ with positive conductances $(W_e)_{e\in E}$, and denote $W_{i,j}=\ind_{\{i,j\}\in E}W_{\{i,j\}}$. The VRJP on $\G$, with respect to $W$, is the self-interacting random process  $(Y_s)_{s\in\R_+}$ on $V$ defined as follows: the process starts at some vertex $i_0\in V$ at time $0$, and conditionally on the past at time $s$, jumps to a neighbour $j$ of $i=Y_s$ at rate 
\[W_{i,j}L_j(s) \mbox{, where } L_j(s)=1+\int_0^s\ind_{\{Y_u=j\}}du.\]
In other words, as the local time $\int_0^s\ind_{\{Y_u=i\}}du$ spent by the process at $i$ increases, the vertex $i$ becomes more attractive. This process was introduced in \cite{DavVol}.

In \cite{SabTar}, Sabot and Tarrès introduced a time change for the VRJP, by defining the increasing function $D(s)=\sum_{i\in V}(L_i(s)^2-1)$, and taking $(Z_t)_{t\geq 0}=(Y_{D^{-1}(t)})_{t\geq 0}$. On finite graphs, this time-changed VRJP $Z$ started at a vertex $i_0$ is then a mixture of Markov processes, in the following sense: there exists a random potential $(u_{i_0}(i))_{i\in V}$, whose distribution is explicit, such that the law of $Z$ is the same as that of a Markov process in a random environment given by jump rates 
\[\frac{1}{2}W_{i,j}e^{u_{i_0}(j)-u_{i_0}(i)}\]
from $i$ to $j$. The idea behind this time change is that the VRJP $(Y_s)_{s\geq 0}$ jumps faster and faster as the vertices become more attractive, and that the time change $D$ is such that $(Z_t)_{t\geq 0}$ has more stationary jumping times, which is necessary for it to be a mixture of Markov processes.

In \cite{SabTar}, Sabot and Tarrès also showed that the VRJP was related to another self-interacting process, the Edge Reinforced Random Walk (or ERRW), introduced in \cite{CopDia} by Coppersmith and Diaconis. On finite graphs, thanks to a de Finetti type theorem for Markov chains (see \cite{DiaFre}), it can be seen as a mixture of Markov chains. This interpretation of the ERRW as a mixture of random walks was studied in \cite{Pem}, \cite{Pem2}, \cite{MerRol}, \cite{MerRol2}, \cite{AngCraKoz}. The link between VRJP and ERRW proven in \cite{SabTar} gives an explicit representation of the ERRW as a mixture of random walks on finite graphs, and has consequences for the study of ERRW on infinite graphs.

In \cite{SabTarZen}, Sabot, Tarrès and Zeng showed that the distributions of potentials $u_{i_0}$ can be coupled for $i_0\in V$, using a potential $\beta=(\beta_i)_{i\in V}$ on $V$, and a random Schrödinger operator associated with $\beta$. Let us denote by $H_\beta=2\beta-W$ the random Schrödinger operator, \textit{i.e.} the $|V|\times|V|$ symmetrical matrix such that $(H_\beta)_{i,j}=2\beta_i\ind_{i=j} -W_{i,j}$ for $i,j\in V$. Moreover, we define by $G=(H_\beta)^{-1}$ the associated Green function. Then $u_{i_0}$ can be defined by 
\[e^{u_{i_0}(i)}=\frac{G(i_0,i)}{G(i_0,i_0)}\]
for $i_0,i\in V$.

This representation using the $\beta$ field allows a generalisation to infinite graphs: in \cite{SabZen}, Sabot and Zeng used a similar potential $\beta$ on infinite graphs to show that the VRJP is still a mixture of Markov processes. If we still denote by $H_\beta=2\beta-W$ the operator associated with $\beta$, we can define the Green function $\Gh=(H_\beta)^{-1}$ in a certain sense. Moreover, there exists $\psi$, a $H_\beta$-harmonic function on $V$ (\textit{i.e.} $H_\beta\psi=0$), obtained as the limit of a martingale. Then if we define $G(i,j)=\Gh(i,j)+\frac{1}{2\gamma}\psi(i)\psi(j)$, where $\gamma$ is a random Gamma variable independent from $\beta$, the time-changed VRJP $(Z_t)$ is still a mixture of Markov processes, with jump rates from $i$ to $j$ given by
\[\frac{1}{2}W_{i,j}\frac{G(i_0,j)}{G(i_0,i)}\]
The term $\frac{1}{2\gamma}$ corresponds to a boundary term. Indeed, to show the result for infinite graphs, the VRJP is first studied on finite subgraphs, endowed with a wired boundary condition. This representation also gave results for the ERRW on infinite graphs.

In \cite{CheZen}, Chen and Zeng showed that in the case of infinite trees, there is another representation of $(Z_t)$ as a mixture of Markov processes. This representation is obtained by using free boundary conditions on restrictions of the tree, since the representation of the VRJP on finite trees has a simpler expression. The particular structure of the tree already gave results for the ERRW (see \cite{Pem}) and the VRJP (see \cite{DavVol2}, \cite{BasSin}). We show that in some cases, the representation of the VRJP obtained this way on the tree differs from the one defined in \cite{SabZen}. This raises the question of the classification of all possible representations of the VRJP as a mixture of Markov processes. In this paper, we give several partial answers to this question.

We first show that any such representation can be expressed in the same form as before, using a $\beta$ field, \textit{i.e.} the random jump rates are given by
\[\frac{1}{2}W_{i,j}\frac{G(i_0,j)}{G(i_0,i)},\]
where $G(i_0,i)=\Gh(i_0,i)+h(i)$, with $h$ a random $H_\beta$-harmonic function.

In the case where the graph is the lattice $\Z^d$, this allows us to show that for certain initial conductances $W$, there is only one representation of the VRJP as a mixture of Markov processes. This is true for strong reinforcement (\textit{i.e.} small $W$), since the VRJP is recurrent, but also for weak reinforcement (\textit{i.e.} large $W$). In this last case, we use a local limit theorem for random walks in random environment to show that the only $H_\beta$-harmonic functions are constants, by proving that the associated Martin boundary is trivial.

In the case where the graph is an infinite tree, we already know of two different representations of the VRJP. Using new boundary conditions, we construct a family of representations, that are all different if the tree is regular enough.

\section{Statement of the results}

\subsection{Previous results}\label{2.1}
Let $\G=(V,E)$ be a finite connected nondirected graph, endowed with conductances $(W_e)_{e\in E}$. We describe $(W_e)_{e\in E}$ with a matrix $(W_{i,j})_{i,j\in V}$, where
\[W_{i,j}=\begin{cases} W_{\{i,j\}} &\mbox{if } \{i,j\}\in E,\\0 &\mbox{otherwise.}\end{cases}\]
In \cite{SabTar}, Sabot and Tarrès showed that that the time-changed VRJP on $\G$ with respect to $W$ could be represented as a mixture of Markov processes, \textit{i.e.} as a random walk in random environment. In \cite{SabTarZen}, Sabot, Tarrès and Zeng showed that this environment could be related to a random Schrödinger operator $H_\beta$, constructed from a random potential $\beta=(\beta_i)_{i\in V}$, in the following way.

For $\beta\in\R^V$, we will denote by $H_\beta=2\beta-W$ the $|V|\times|V|$ symmetrical matrix such that $(H_\beta)_{i,j}=2\beta_i\ind_{i=j} -W_{i,j}$ for $i,j\in V$. Let us define the set $\calD_V^W=\{\beta\in\R^V,H_\beta>0\}$, where $H_\beta>0$ means that the matrix $H_\beta$ is positive definite. Note that if $\beta\in\calD_V^W$, then $\beta_i>0$ for all $i\in V$. The following proposition describes the probability distribution of the random potential that will be used to represent the VRJP.

\begin{prop}\label{prop:nu.dist}[Theorem 1 in \cite{SabTarZen}, Lemma 4 in \cite{SabZen}]
\begin{itemize}
\item[(i)] Let $\G=(V,E)$ be a finite connected graph, endowed with conductances $W$, and let $\eta\in\R_+^V$. We define by $\nu_V^{W,\eta}$ the measure on $(\calD_V^W,\mathcal{B}(\calD_V^W))$ such that
\[\nu_V^{W,\eta}(d\beta)=\l(\frac{2}{\pi}\r)^{\frac{|V|}{2}}e^{-\frac{1}{2}(\langle \ind,H_\beta\ind\rangle+\langle\eta,(H_\beta)^{-1}\eta\rangle)}e^{\langle\eta,\ind\rangle}\frac{\prod_{i\in V}d\beta_i}{\sqrt{\det(H_\beta)}}.\]
Then $\nu_V^{W,\eta}$ is a probability distribution. Its Laplace transform is
\[\int e^{-\langle\lambda,\beta\rangle}\nu_V^{W,\eta}(d\beta)=e^{-\sum_{i\in V}\eta_i(\sqrt{1+\lambda_i}-1)-\sum_{i\sim j}W_{i,j}(\sqrt{(1+\lambda_i)(1+\lambda_j)}-1)}\prod_{i\in V}\frac{1}{\sqrt{1+\lambda_i}},\]
for $\lambda\in\R_+^V$. When $\eta=0$, we will write $\nu_V^W=\nu_V^{W,0}$. 
\item[(ii)] Let us denote by $d_\G$ the graph distance in $\G$. Under $\nu_V^{W,\eta}(d\beta)$, if $V_1,V_2\subset V$ are such that $d_\G(V_1,V_2)\geq 2$,  then $(\beta_i)_{i\in V_1}$ and $(\beta_j)_{j\in V_2}$ are independent. We will say that the potential with distribution $\nu_V^W$ is 1-dependent.
\end{itemize}
\end{prop}

Let $D([0,\infty),V)$ be the space of càdlàg functions from $[0,\infty)$ to $V$. This will be the space of trajectories of the random processes we will study in this paper. These processes will be described by probability distributions on $D([0,\infty),V)$, and we will denote by $(Z_t)$ the canonical process, where $Z_t(\omega)=\omega(t)$ for $\omega\in D([0,\infty),V)$. 

For $i_0\in V$, let $P^{VRJP(i_0)}$ denote the distribution of the time-changed VRJP, using the time change $D$ described in the introduction. Note that $P^{VRJP(i_0)}$ is a probability distribution on $D([0,\infty),V)$. The following theorem describes how to represent $P^{VRJP(i_0)}$ as a mixture of Markov processes, using a random environment that can be constructed from the $\beta$ field under $\nu_V^W(d\beta)$.

\begin{thm}\label{thm:beta.mix}[Theorem 2 in \cite{SabTar}, Theorem 3 in \cite{SabTarZen}]
Let $\G=(V,E)$ be a finite graph, endowed with conductances $W$. We fix a vertex $i_0\in V$. For $\beta\in\calD_V^W$, we denote by $G=(H_\beta)^{-1}$ the Green function associated with $\beta$, and by $P^{\beta,i_0}_x$ the distribution of the Markov jump process started at $x\in V$, with jump rate from $i$ to $j$ given by $\frac{1}{2}W_{i,j}\frac{G(i_0,j)}{G(i_0,i)}$.

Then for all $i_0\in V$, the law $P^{VRJP(i_0)}$ of the time-changed VRJP on $V$, with respect to $W$ and started at $i_0$, is a mixture of these Markov jump processes under the distribution $\nu_V^W(d\beta)$. In other words,
\[P^{VRJP(W,i_0)}[\cdot]=\int P^{\beta,i_0}_{i_0}[\cdot] \nu_V^W(d\beta).\]
\end{thm}

An interesting property of the distribution $\nu_V^W$ is its behaviour with respect to restriction. For $\beta\in\R^V$ and $V_1,V_2\subset V$, let us denote $\beta_{V_1}=(\beta_i)_{i\in V_1}$, and $W_{V_1,V_2}=(W_{i,j})_{i\in V_1,j\in V_2}$. 

\begin{prop}\label{prop:beta.restr.sotr}[Lemma 4 in \cite{SabZen}]
Let us fix $U\subset V$, and set $\hat\eta_i=\sum_{j\in U^c}W_{i,j}$ for $i\in U$, \textit{i.e.} $\hat\eta=W_{U,U^c}\ind_{U^c}$. Then under $\nu_V^W(d\beta)$, $\beta_U$ is distributed according to $\nu_U^{W_{U,U},\hat\eta}$.
\end{prop}

Hence under $\nu_V^W(d\beta)$, the distribution of $\beta_U$ depends only on the weights of edges inside $U$, and coming out of $U$. This is useful to define the $\beta$ field on infinite graphs.

Let now $\G=(V,E)$ be an infinite connected nondirected graph, that is locally finite, \textit{i.e.} each vertex $v\in V$ has finite degree. We endow $\G$ with conductances $W$. To study the associated VRJP, we want to define an analogue of the $\beta$ field on $\G$. In \cite{SabZen}, Sabot and Zeng did this by using a wired boundary condition, defined as follows.

Let $(V_n)_{n\in\N}$ be an increasing sequence of finite connected subsets of $V$, such that
\[\bigcup_{n\in\N}V_n=V.\]
For $n\in\N$, we introduce a new vertex $\delta_n$, and define a new graph $\G^\n=(\tilde{V}^\n,\tilde{E}^\n)$, where
\begin{align*}
\tilde{V}^\n &= V_n \cup \{\delta_n\} \\
\mbox{and } \tilde{E}^\n &= \l\{\{i,j\}\in E, i,j\in V_n\r\} \cup \l\{\{i,\delta\}, i\in V_n \mbox{ and } \exists j\notin V_n, i\sim j\r\}.
\end{align*}
The graph $\G^\n$ is called the restriction of $\G$ to $V_n$ with wired boundary condition. We endow this graph with the conductances $\Wt^\n$ defined by $\Wt^\n_{i,j}=W_{i,j}$ if $i,j\in V_n$, and $\Wt^\n_{i,\delta_n}=\sum_{j\sim i, j\notin V_n} W_{i,j}$.

For all $n\in\N$, let $(\beta^\n_i)_{i\in\tilde{V}^\n}$ be a random potential on the graph $\G^\n$ distributed according to $\nu_{\tilde{V}^\n}^{\Wt^\n}$. Then from Proposition \ref{prop:beta.restr.sotr}, we know that the restriction $\beta^\n_{V_n}$ is distributed according to $\nu_{V_n}^{W^\n,\eta^\n}$, where $W^\n=W_{V_n,V_n}$ and $\eta^\n:=\Wt^\n_{V_n,\{\delta_n\}}=W_{V_n,V_n^c}\ind_{V_n^c}$. In fact, for a fixed $n\in\N$ and any $n'\geq n$, the restrictions $\beta^{(n')}_{V_n}$ have the same distribution $\nu_{V_n}^{W^\n,\eta^{(n)}}$. By Kolmogorov extension theorem, this allows the construction of a distribution $\nu_V^W$ for infinite $V$.

For $\beta\in\R^V$, let us still denote by $H_\beta=2\beta-W$ the Schrödinger operator associated with $(\beta_i)_{i\in V}$, \textit{i.e.} for all $f\in\R^V$ and $i\in V$, $(H_\beta f)_i=2\beta_i f_i -\sum_{j\sim i}W_{i,j}f_j$. We also define $\calD_V^W=\{\beta\in\R^V,(H_\beta)_{U,U}>0 \mbox{ for all finite subset }U\mbox{ of }V\}$.

\begin{prop}\label{prop:nu.dist.inf}[Proposition 1 in \cite{SabZen}]
Let $\G=(V,E)$ be an infinite locally finite graph. There exists a unique probability distribution $\nu_V^W$ on $\calD_V^W$ such that under $\nu_V^W(d\beta)$, for all finite subset $U\subset V$, $\beta_U\sim\nu_U^{W_{U,U},\eta}$ where $\eta=W_{U,U^c}\ind_{U^c}$. Its Laplace transform is
\[\int e^{-\langle\lambda,\beta\rangle}\nu_V^W(d\beta)=e^{-\sum_{i\sim j}W_{i,j}(\sqrt{1+\lambda_i}\sqrt{1+\lambda_j}-1)}\prod_{i\in V}\frac{1}{\sqrt{1+\lambda_i}}\]
for $\lambda\in\R_+^V$ with finite support.
\end{prop}

The wired boundary condition is not only useful to define $\nu_V^W$ on infinite graph, but also to link this distribution to representations of the VRJP, by applying Theorem \ref{thm:beta.mix} to the graph $\G^\n$. Indeed from Proposition \ref{prop:nu.dist.inf}, for any $n\in\N$, under $\nu_V^W(d\beta)$ we have $\beta_{V_n}\sim\nu_{V_n}^{W^\n,\eta^\n}$. Hence, from Proposition \ref{prop:beta.restr.sotr}, we can extend $\beta_{V_n}$ into a potential $\beta^\n\sim\nu_{\tilde{V}^\n}^{\Wt^\n}$ such that $\beta^\n_{V_n}=\beta_{V_n}$. We denote $H_\beta^\n=2\beta^\n-\Wt^\n$ and $G^\n=(H_\beta^\n)^{-1}$. From Theorem \ref{thm:beta.mix}, we know that $G^\n$ gives a representation of the VRJP on $\G^\n$.

\begin{defi}\label{defi:Gh.psi}
\begin{itemize}
\item[(i)]
For $\beta\in\calD_V^W$, let us define $\Gh^\n:V\times V\to\R_+$ by  $(\Gh^\n)_{V_n,V_n}=((H_\beta)_{V_n,V_n})^{-1}$, and $\Gh^\n(i,j)=0$ if $i\notin V_n$ or $j\notin V_n$.
\item[(ii)]
For $\beta\in\calD_V^W$, let $\psi^\n\in\R_+^{V_n}$ be defined by 
\[\begin{cases} (H_\beta \psi^\n)_{V_n}=0 \\ \psi^\n_{V_n^c}=1.\end{cases}\]
Note that $\psi^\n_{V_n}=(\Gh^\n_{V_n,V_n})\eta^\n$.
\end{itemize}
\end{defi}

It is possible, using a decomposition of the Green function as a sum over paths (see \cite{SabZen}, or Proposition \ref{prop:g.somme}), to write
\[G^\n(i,j)=\Gh^\n(i,j)+\psi^\n(i)G^\n(\delta_n,\delta_n)\psi^\n(j)\]
for $i,j\in V_n$. Under $\nu_V^W(d\beta)$, $G^\n(\delta_n,\delta_n)$ is independent of $\beta_{V_n}$, and is always distributed according to a $Gamma(1/2,1)$ distribution (see Proposition \ref{prop:beta.u} (ii)). The following theorem describes how taking $n\to\infty$ in this previous expression gives a representation of the VRJP on infinite graphs.

\begin{thm}\label{thm:Gh.psi.mix}[Theorem 1 in \cite{SabZen}]
\begin{itemize}
\item[(i)] Under $\nu_V^W(d\beta)$, for $i,j\in V$, the increasing sequence $\Gh^\n(i,j)$ converges almost surely to a finite random variable $\Gh(i,j)$.
\item[(ii)] Let $\F_n$ be the $\sigma$-field generated by $\beta_{V_n}$. Then under $\nu_V^W(d\beta)$, for all $i\in V$, $\psi^\n(i)$ is a nonnegative $(\F_n)$-martingale which converges almost surely to an integrable random variable $\psi(i)$. Moreover, $\psi$ is $H_\beta$-harmonic on $V$, \textit{i.e.} $H_\beta \psi(i)=2\beta_i\psi(i)-\sum_{j\sim i}W_{i,j}\psi(j)=0$ for $i\in V$.
\item[(iii)] From now on, we will denote $\nu_V^W(d\beta,d\gamma)=\nu_V^W(d\beta)\otimes\frac{\ind_{\{\gamma>0\}}}{\sqrt{\pi\gamma}}e^{-\gamma}d\gamma$, where $\frac{\ind_{\{\gamma>0\}}}{\sqrt{\pi\gamma}}e^{-\gamma}d\gamma$ is a $Gamma(1/2,1)$ distribution. 

Let now $i_0\in V$ be fixed. For $\beta\in\calD_V^W$ and $\gamma>0$, we define
\[G(i,j)=\Gh(i,j)+\frac{1}{2\gamma}\psi(i)\psi(j),\]
and denote by $P_x^{\beta,\gamma,i_0}$ the distribution of the Markov jump process started at $x\in V$, where the jump rate from $i$ to $j$ is $\frac{1}{2}W_{i,j}\frac{G(i_0,j)}{G(i_0,i)}$.

Then the law $P^{VRJP(i_0)}$ of the time-changed VRJP on $V$, with respect to $W$ and started at $i_0$, is a mixture of these Markov jump processes under $\nu_V^W(d\beta,d\gamma)$, \textit{i.e.}
\[P^{VRJP(i_0)}[\cdot]=\int P^{\beta,\gamma,i_0}_{i_0}[\cdot] \nu_V^W(d\beta,d\gamma).\]
\item[(iv)] For $\nu_V^W$-almost all $\beta$ and all $i_0\in V$, we have:
\begin{itemize}
\item The Markov process $P_x^{\beta,\gamma,i_0}$ is recurrent if and only if $\psi(i)=0$ for all $i\in V$.
\item The Markov process $P_x^{\beta,\gamma,i_0}$ is transient if and only if $\psi(i)>0$ for all $i\in V$.
\end{itemize}
\end{itemize}
\end{thm}

Note that for $i_0\in V$ fixed, in this representation of the VRJP started at $i_0$, the $\beta$ field cannot be expressed as a function of the random jump rates $\frac{W_{i,j}}{2}\frac{G(i_0,j)}{G(i_0,i)}$ that define the environment. However, we can define the $\tilde{\beta}$ field rooted at $i_0$, where $\tilde\beta_i$ is the rate of the exponential holding time at $i$ for the associated Markov process.

\begin{prop}\label{prop:beta.tilde.dist}
For all $i\in V$, $\beta\in\calD_V^W$ and $\gamma>0$, we define
\[\tilde{\beta}_{i}=\sum_{j\sim i}\frac{W_{i,j}}{2}\frac{G(i_0,j)}{G(i_0,i)}=\beta_i-\ind_{\{i=i_0\}}\frac{1}{2G(i_0,i_0)}.\]
Then under $\nu_V^W(d\beta,d\gamma)$, $1/2G(i_0,i_0)$ has distribution $\Gamma(1/2,1)$ and is independent from $\tilde\beta$. Moreover the Laplace transform of $\tilde{\beta}$ is
\[\int e^{-\langle\lambda,\tilde\beta\rangle}\nu_V^W(d\beta,d\gamma)=e^{-\sum_{i\sim j}W_{i,j}(\sqrt{1+\lambda_i}\sqrt{1+\lambda_j}-1)}\prod_{i\neq i_0}\frac{1}{\sqrt{1+\lambda_i}}\]
for $\lambda\in\R_+^V$ with finite support.
\end{prop}

\subsection{A common form for all representations}

We still consider $\G=(V,E)$ to be an infinite connected graph, locally finite and endowed with conductances $(W_{i,j})_{i,j\in V}$. Thanks to Theorem \ref{thm:Gh.psi.mix}, we already know that the law $P^{VRJP(i_0)}$ of the time-changed VRJP can be written as a mixture of Markov jump processes, using the distribution $\nu_V^W$. We will refer to this as the standard representation. We are now interested in other possible random environments, that would represent the VRJP in the same sense, and whether they can be expressed in a form similar to the standard representation.

We will denote by $\mathcal{J}_V^E=\{(r_{i,j})_{i\sim j}\in\R_+^E\}$ the set of jump rates on $\G$.

\begin{defi}\label{def:rep}
Let $\calR(dr)$ be a probability distribution on $\calJ_V^E$. For $i_0\in V$ fixed, we will say that $\calR(dr)$ is the distribution of a random environment representing $P^{VRJP(i_0)}$ if
\[P^{VRJP(i_0)}[\cdot]=\int P^{r}_{i_0}[\cdot] \calR(dr),\]
where for $r\in\calJ_V^E$, $P^{r}$ is the distribution of the Markov jump process with jump rate from $i$ to $j$ given by $r_{i,j}$.
\end{defi}

The following result tells us that in fact, any representation of the VRJP can be expressed in a similar form as the standard representation, using a $\beta$ field as well as $H_\beta$-harmonic functions.

For $i\in V$ and $r\in\calJ_V^E$, we define $r_i=\sum_{j\sim i} r_{i,j}$.

\begin{thm}\label{thm:beta.env}
Let $i_0\in V$ be fixed, and let $\calR(dr)$ be the distribution of a random environment representing $P^{VRJP(i_0)}$. We write $\calR(dr,d\gamma)=\calR(dr)\otimes\frac{\ind_{\{\gamma>0\}}}{\sqrt{\pi\gamma}}e^{-\gamma}d\gamma$.

For $r\in\calJ_V^E$ and $\gamma>0$, we define $\beta\in(\R_+)^V$ by $\beta_i=r_i+\ind_{\{i=i_0\}}\gamma$ for $i\in V$. Then under $\calR(dr,d\gamma)$, $\beta\sim\nu_V^W$, and there exists a random $H_\beta$-harmonic function $h:V\to\R_+$, such that for all $i\sim j$,
\[r_{i,j}=\frac{W_{i,j}}{2}\frac{G(i_0,j)}{G(i_0,i)},\]
where  $G(i_0,i)=\Gh(i_0,i)+h(i)$ for $i\in V$, and $\Gh$ is the function of $\beta$ defined in Theorem \ref{thm:Gh.psi.mix}.
\end{thm}

In order to try and classify all representations of the VRJP, we now need to identify $H_\beta$-harmonic functions, and to determine which ones can appear in the expression of a representation, as in Theorem \ref{thm:beta.env}. Two interesting cases arise, depending on $(\G,W)$: when the VRJP is almost surely recurrent, or almost surely transient.

In the first case, we can use the law of large numbers to show that the representation of $P^{VRJP(i_0)}$ as a mixture of Markov processes is unique.

\begin{prop}\label{prop:repr.rec}
If $(\G,W)$ is such that the VRJP is almost surely recurrent, then the representation of $P^{VRJP(i_0)}$ as a mixture of Markov processes is unique, \textit{i.e.} if $\calR(dr)$ and $\calR'(dr)$ are the distributions of random environments representing $P^{VRJP(i_0)}$, then $\calR(dr)=\calR'(dr)$.
\end{prop}

Note that in this case, according to Theorem \ref{thm:Gh.psi.mix} (iv), under $\nu_V^W(d\beta)$, we have a.s. $\psi(i)=0$ for all $i\in V$, and the jump rates in the standard representation are given by $\frac{W_{i,j}}{2}\frac{\Gh(i_0,j)}{\Gh(i_0,i)}$. Therefore, the $H_\beta$-harmonic function associated with the representation (by Theorem \ref{thm:beta.env}) is $h\equiv 0$.

In the second case, \textit{i.e.} when the VRJP is almost surely transient, we can introduce a random conductance model, associated with $\psi$.

\begin{prop}\label{prop:trans.delta.psi}
If $(\G,W)$ is such that the VRJP is almost surely transient, then under $\nu_V^W(d\beta)$:
\begin{itemize}
\item[(i)] We have a.s. $\psi(i)>0$ for all $i\in V$, where $\psi$ is defined in Theorem \ref{thm:Gh.psi.mix}.
\item[(ii)] We define the random conductances $c^\psi_{i,j}=W_{i,j}\psi(i)\psi(j)$ for all $i,j\in V$. Then the associated reversible random walk is a.s. transient.
\item[(iii)] Let $\Delta^\psi$ be the discrete Laplacian associated with the random conductances $c^\psi_{i,j}$. Then a function $\varphi:V\to\R$ is $\Delta^\psi$-harmonic if and only if $i\mapsto\psi(i)\varphi(i)$ is $H_\beta$-harmonic.
\end{itemize}
\end{prop}

\begin{rem}\label{rem:rep.fonc.harm}
This allows a more convenient expression of representation in the transient case. Indeed, if $\mathcal{R}(dr)$ is the distribution of a random environment representing $P^{VRJP(i_0)}$, Theorem \ref{thm:beta.env} allows us to construct a $\beta$ field distributed according to $\nu_V^W$, and to express the jump rates $r_{i,j}$ using $\beta$ and a $H_\beta$-harmonic function $h$. According to Proposition \ref{prop:trans.delta.psi} (iii), we have $h=\psi\varphi$, where $\varphi$ is a $\Delta^\psi$-harmonic function, \textit{i.e.} harmonic for a transient random walk.
\end{rem}

The notion of Martin boundary is a useful tool to represent harmonic functions with respect to a transient random walk on a graph $\G=(V,E)$. Indeed, $V$ admits a boundary $\M$ so that $V\cup\M$ is compact for a certain topology, and there is a kernel $K:V\times\M$ so that any positive harmonic function $h$ can be written as
\[h(x)=\int_\M K(x,\alpha)\mu^h(d\alpha)\]
for $x\in V$, where $\mu^h$ is a positive measure on $(\M,\mathcal{B}(\M))$. $\M$ is called the Martin boundary of $V$ with respect to the random walk, and $K$ is the Martin kernel, which is defined using the Green kernel associated with the random walk. For more details on Martin boundaries, see Section \ref{3.3}.

In order to study representations of the VRJP in the transient case, we want to describe $\Delta^\psi$-harmonic functions, according to Remark \ref{rem:rep.fonc.harm}. We will therefore need to identify the Martin boundary $\mathcal{M}^\psi$ associated with $\Delta^\psi$. This will be possible when $\G$ is $\Z^d$, or an infinite tree.

\subsection{Main results}

\subsubsection{Representations of the VRJP on $\Z^d$}

Let us consider the case where $\G$ is the lattice $\Z^d$, \textit{i.e.} $\G(V,E)$ with 
\[V=\Z^d \mbox{ and } E=E_d:=\l\{\{x,y\},|x-y|=1\r\}\]
where $|x|$ is the Euclidean norm of $x$. Let us endow $\G$ with constant initial conductances $W$. We can identify several situations in which the representation is unique. For $d=2$, or if $W$ is small enough, the VRJP is almost surely recurrent (see \cite{BHS18}, and Corollary 1 in \cite{SabTar}), so that the representation of $P^{VRJP(i_0)}$ is unique according to Proposition \ref{prop:repr.rec}. For $d\geq 3$ and $W$ large enough, the VRJP is almost surely transient (see Corollary 3 in \cite{SabTar}), hence we can introduce $\Delta^\psi$ defined in Proposition \ref{prop:trans.delta.psi}. Since $(\G,W)$ is vertex transitive, from Proposition 3 of \cite{SabZen}, under $\nu_V^W(d\beta)$, $\psi$ is stationary and ergodic. This allows us to apply a local limit theorem for random walks in random conductances (from \cite{ADS16}), and show that the Martin boundary $\M^\psi$ associated with $\Delta^\psi$ is almost surely trivial for $W$ large enough. These cases are regrouped in the following result.

\begin{thm}\label{thm:zd.rep}
Let $\G$ be the $\Z^d$ lattice, endowed with constant edge weights, \textit{i.e.} $W_{i,j}=W>0$ for all $i\sim j$. We consider representations of $P^{VRJP(0)}$ as a mixture of Markov processes.

Then:
\begin{itemize}
\item If $d\in\{1,2\}$, there is a unique representation of $P^{VRJP(0)}$.
\item If $d\geq 3$, there are constants $\underline{W}$ and $\overline{W}$ such that for $0<W<\underline{W}$ or for $W>\overline{W}$, there is a unique representation of $P^{VRJP(0)}$.
\end{itemize}
\end{thm}

\subsubsection{A family of representations on infinite trees}\label{2.4}

Let us now consider the case where the graph is an infinite tree $\mathcal{T}=(T,E)$, that we assume to be locally finite, and endow with conductances $W$. In \cite{CheZen}, Chen and Zeng described a representation of the time-changed VRJP with a different expression than the standard representation. Indeed, if $(T_n)_{n\in\N}$ is an increasing and exhausting sequence of finite connected subsets of $T$, the subgraphs $\mathcal{T}^\n=(V_n,E_n)$ of $\G$ are finite trees (where $E_n=\l\{\{i,j\}\in E,i,j\in V_n\r\}$). These are called restrictions of $\G$ with free boundary conditions. 

Moreover, on finite trees, Theorem \ref{thm:beta.mix} gives a representation of the VRJP where jump rates are independent. Therefore, a representation of the VRJP on $\mathcal{T}$ can be obtained from representations on $\mathcal{T}^\n$, using independent jump rates.

\begin{thm}\label{thm:mix.tree}[Theorem 3 in \cite{CheZen}]
Let $\phi$ be an arbitrary root for $\calT$. For all $i\in T\backslash\{\phi\}$, we denote by $\p{i}$ the parent of $i$. Let also $(A_i)_{i\in T\backslash\{\phi\}}$ be independent random variables where $A_i$ is an inverse Gaussian random variable with parameter $(W_{\p{i},i},1)$, \textit{i.e.}
\[\P[A_i\in ds]=\ind_{s\geq 0}\sqrt{\frac{W_{\p{i},i}}{2\pi s^3}}e^{-W_{\p{i},i}\frac{(s-1)^2}{2s}}ds.\]
Then the law $P^{VRJP(\phi)}$ on $\calT$ is a mixture of Markov jump processes, in which the jump rate from $\p{i}$ to $i$ is $\frac{1}{2}W_{\p{i},i}A_i$, and the jump rate from $i$ to $\p{i}$ is $\frac{1}{2}\frac{W_{\p{i},i}}{A_i}$, for all $i\in T\backslash\{\phi\}$.
\end{thm}

In some cases, this representation is different from the standard representation.

\begin{prop}\label{prop:2.rep.reg.tree}
Let $\calT=(T,E)$ be an infinite $d$-regular tree with $d\geq 3$, \textit{i.e.} such that each vertex in $T$ has exactly $d$ neighbours. We endow $\calT$ with constant conductances $W$. Then for $W$ large enough, the distribution of the random environment described in Theorem \ref{thm:mix.tree} is different from the distribution of the standard representation.
\end{prop}

We now know two ways of constructing a representation of the VRJP on $\calT$, that are associated with different boundary conditions on restrictions to finite graphs, and can have distinct distributions. This leads us to introduce new boundary conditions in order to construct a family of different representations of the VRJP, following the same method as for the standard representation. 

Let us start by giving a few notations on trees. For all $x,y\in T$, we denote by $d(x,y)$ the graph distance between $x$ and $y$, and by $[x,y]$ the unique shortest path between $x$ and $y$:
\[[x,y]=\l(x=[x,y]_0,[x,y]_1,...,[x,y]_{d(x,y)-1},y=[x,y]_{d(x,y)}\r).\]
Note that any path $\sigma$ from $x$ to $y$ necessarily crosses all vertices $[x,y]_k$ for $0\leq k\leq d(x,y)$.

Let us fix an arbitrary root $\phi$ in $\mathcal{T}$. Then, for all $x\in T$, we denote by $|x|=d(\phi,x)$ the depth of the vertex $x\in T$. If $x\neq\phi$, we also denote by $\p{x}=[\phi,x]_{|x|-1}$ the parent of $x$. Finally, for any $x\in T$, we define the set $S(x)=\{y\in T, x=\p{y}\}$ of $x$'s children, and the set $T_x=\{y\in T,\exists k\geq 0, [\phi,y]_k=x\}$ of its descendants.

For $x,y\in T$, we denote by $x\wedge y$ the "closest common ancestor" of $x$ and $y$, \textit{i.e.} $x\wedge y=[\phi,x]_{K_{x,y}}$ where $K_{x,y}=\max\{k\geq 0, [\phi,x]_k=[\phi,y]_k\}$. Note that we also have $x\wedge y=[x,y]_{k_0}$, where $k_0$ is such that $|[x,y]_{k_0}|=\min\{|[x,y]_k|,0\leq k\leq d(x,y)\}$.

For $n\in\N$, we denote by $D^\n=\{x\in T, |x|=n\}$ the tree's $n$th generation. Let us then define $T^\n=\bigcup_{0\leq k\leq n}D^{(k)}$, as well as $E^\n=\l\{\{i,j\}\in E, i,j\in T^\n\r\}$. We also denote $T_x^\n=T_x\cap T^\n$ for $x\in T$ and $n\geq |x|$. The restriction of the tree to the first $n$ generations is the graph $(T^\n, E^\n)$, that we endow with the induced conductances $W^\n=W_{T^\n,T^\n}$.

Finally, we define the set $\Omega$ of ends of $\mathcal{T}$, \textit{i.e.} the set of infinite self-avoiding paths (or rays) in $\mathcal{T}$ starting at $\phi$. For $x\in T$, we denote by $\Omega_x$ the subset of $\Omega$ corresponding to the branch $T_x$, \textit{i.e.} the set of rays in $\mathcal{T}$ that cross $x$. Note that the Martin boundary associated with a transient walk on a tree is always $\Omega$, which depends only on the geometry of the tree. This will be convenient to express $\Delta^\psi$-harmonic functions, where $\Delta^\psi$ is the random Laplace operator introduced in Proposition \ref{prop:trans.delta.psi}.

In the construction of the standard representation, the wired boundary condition was defined by adding a single boundary point $\delta$ to a finite graph, where $\delta$ could be interpreted as a point at infinity for the graph. We will now introduce a variant of this boundary condition, by adding multiple boundary points, each being a point at infinity for a different branch of the tree.

Let us first fix a generation $m\geq 0$, and to each vertex $x\in D^{(m)}$, we associate a boundary point $\delta_x$, that will be the point at infinity for $T_x$. We denote by $B_m=\{\delta_x, x\in D^{(m)}\}$ the boundary set associated to this generation. For all $n\geq m$, let us then define the graph
\begin{align*}
\G_m^\n=\l(\Tt_m^\n, \tilde{E}_m^\n\r) \mbox{, where } &\Tt_m^\n=T^\n\cup B_m\\
\mbox{and } &\tilde{E}_m^\n=E^\n\cup\bigcup_{x\in D^{(m)}}\l\{\{y,\delta_x\},y\in T_x\cap D^\n\r\}.
\end{align*}

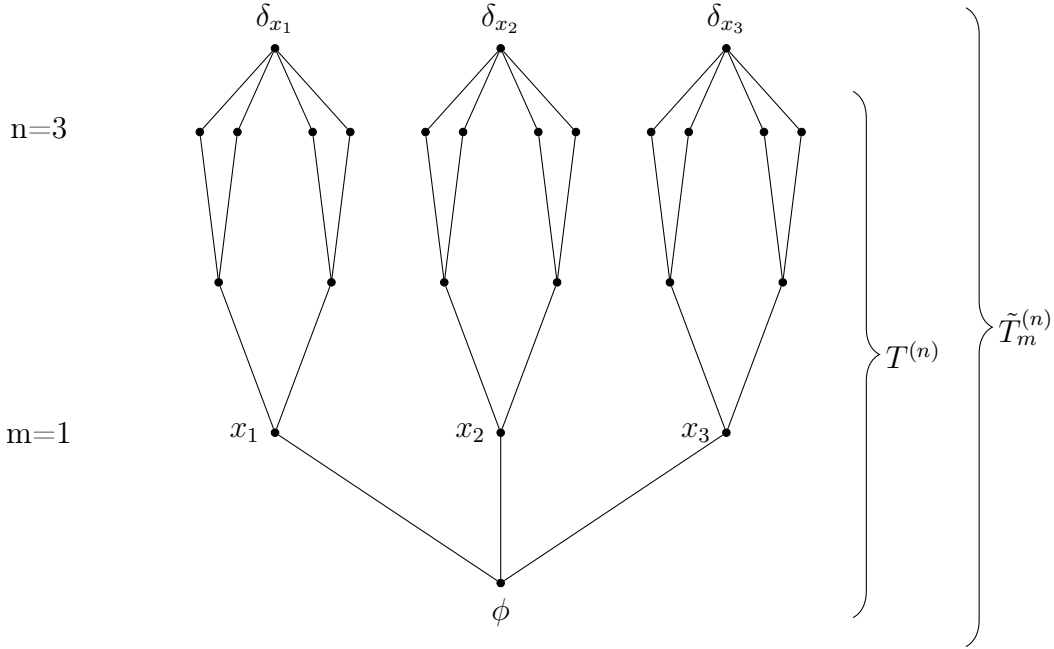
\begin{center}
\begin{tikzpicture}[level distance=2cm,
level 1/.style={sibling distance=3cm},
level 2/.style={sibling distance=1.5cm},
level 3/.style={sibling distance=0.5cm},
snode/.style={circle,draw,inner sep=1,fill=black},
every child node/.style={snode}
]

\node[snode][label=below:{$\phi$}] (Root) {} [grow'=up]
    child { node[label=left:{$x_1$}](x) {} 
    	child { node {} 
    		child{ node(x1) {} }
    		child{ node(x2) {} }
    	}
	        	child { node {} 
    		child{node(x3){}}
    		child{node(x4){}}
    	}
	}
	child{ node[label=left:{$x_2$}](y){}
		child{ node{}
			child{node(y1){}}
			child{node(y2){}}
		}
		child{ node{}
			child{node(y3){}}
			child{node(y4){}}
		}
	}
	child { node(z)[label=left:{$x_3$}] {}
    	child { node {} 
    		child{node(z1){}}
    		child{node(z2){}}
    	}
    	child { node {} 
    		child{node(z3){}}
    		child{node(z4){}}
    	}
    };
\node[snode,above = 5 cm of x,label=above:{$\delta_{x_1}$}](dx){};
	\draw (x1) -- (dx);
	\draw (x2) -- (dx);
	\draw (x3) -- (dx);
	\draw (x4) -- (dx);
\node[snode][label=above:{$\delta_{x_2}$}][above = 5 cm of y](dy){};
	\draw (y1) -- (dy);
	\draw (y2) -- (dy);	
	\draw (y3) -- (dy);
	\draw (y4) -- (dy);
\node[snode][label=above:{$\delta_{x_3}$}][above = 5 cm of z](dz){};
	\draw (z1) -- (dz);
	\draw (z2) -- (dz);
	\draw (z3) -- (dz);
	\draw (z4) -- (dz);
\node[left= 2.5cm of x](m){m=1};
\node[above = 3.5 cm of m]{n=3};
\node[above right= 0.5cm and 0.5cm of z4](a){};
\node[below= 7cm of a](b){};
\draw [decorate,decoration={brace,amplitude=10pt}]
(a) -- (b) node [black,midway,xshift=0.8cm] {
$T^{(n)}$};
\node[above right= 0.5cm and 3cm of dz](c){};
\node[below= 8.5cm of c](d){};
\draw [decorate,decoration={brace,amplitude=10pt},xshift=-4pt,yshift=0pt]
(c) -- (d) node [black,midway,xshift=0.8cm] {
$\tilde{T}_m^{(n)}$};
\end{tikzpicture}
\end{center}

This graph is the restriction of $T$ to $T^\n$ with a variant of the wired boundary condition. Note that we get the standard wired boundary condition by taking $m=0$. We endow $\G_m^\n$ with the conductances $\Wt_m^\n$, defined for $e\in\tilde{E}^\n$ by
\[\l(\Wt_m^\n\r)_e = \begin{cases} W^\n_e =W_e &\mbox{if } e\in E^\n \\ \sum_{j\in S(i)} W_{i,j} &\mbox{if } e=\{i,\delta_x\} \mbox{, where } i\in T_x\cap D^\n.\end{cases}\]
As with the wired boundary condition, these weights are defined so that for $n\geq m$, the weights coming out of $T^\n$ are given by $W_{T^\n,(T^\n)^c}\ind_{(T^\n)^c}=\eta^\n$. This will allow for the compatibility of $\beta_m^\n$ fields defined on $\G_m^\n$ for $n\geq m$. Note that these weights do not depend on $m$, \textit{i.e.} do not depend on the choice of the boundary condition.

From Proposition \ref{prop:nu.dist.inf}, under $\nu_V^W(d\beta)$, for all $n\geq m$ we have $\beta_{T^\n}\sim\nu_{T^\n}^{W^\n,\eta^\n}$. Hence, from Proposition \ref{prop:beta.restr.sotr}, we can extend $\beta_{T^\n}$ into a potential $\beta_m^\n\sim\nu_{\Tt_m^\n}^{\Wt_m^\n}$ on $\G_m^\n$ such that $(\beta_m^\n)_{T^\n}=\beta_{T^\n}$. Let us then denote by $G_m^\n=(2\beta_m^\n-\Wt_m^\n)^{-1}$ the Green function associated with $\beta_m^\n$. From Theorem \ref{thm:beta.mix}, we know how to represent the time-changed VRJP on $\G_m^\n$ using $G_m^\n$. In order to obtain a result on the infinite tree $\mathcal{T}$, we will see that $G_m^\n$ converges when $n\to\infty$.

For $\beta\in\calD_V^W$, we still define $H_\beta=2\beta-W$ and take $V_n=T^\n$ for all $n\in\N$ in Definition \ref{defi:Gh.psi}, we get $\Gh^\n=((H_\beta)_{T^\n,T^\n})^{-1}$ and $\psi^\n=\Gh^\n\eta^\n$, which converge $\nu_V^W$-a.s. to $\Gh$ and $\psi$ respectively, according to Theorem \ref{thm:Gh.psi.mix}.

\begin{defi}\label{defi:chi}
For $n\geq m\geq 0$, let $\chi_m^\n\in\R_+^{T^\n}\times B_m$ be defined by
\[\begin{cases} (H_\beta\chi_m^\n(\cdot,\delta_x))_{T^\n}=0,\\
\chi_m^\n(i,\delta_x)=1 &\mbox{ if } i\in T_x\backslash T_x^\n,\\
\chi_m^\n(i,\delta_x)=0 &\mbox{ if } i\in T\backslash (T_x\cup T^\n).\end{cases}\]
for $x\in D^\m$. Note that $(\chi_m^\n(\cdot,\delta_x))_{T^\n}=(\Gh^\n)_{T^\n,T^\n} W_{T^\n,(T^\n)^c}\ind_{T_x\backslash T_x^\n}$.
\end{defi}

\begin{rem}\label{rem:decoup.chi}
For $n\geq m$, $\chi_m^\n$ is $\beta_{T^\n}$-measurable, and for $x\in D^\m$ and $y\in T^\n$,
\begin{align*}
\sum_{b\in B_m} \chi_m^\n(y,b) &= \sum_{x\in D^\m}\chi_m^\n(y,\delta_x)=\Gh^\n(y,\cdot) W_{T^\n,(T^\n)^c} \ind_{(T^\n)^c}\\
&= \Gh^\n(y,\cdot)\eta^\n= \psi^\n(y).
\end{align*}
\end{rem}

It is possible to decompose $G_m^\n$ as a sum over paths in $\G_m^\n$, which gives, for $i,j\in T^\n$,
\[G_m^\n(i,j)=\Gh^\n(i,j)+\sum_{x,x'\in D^\m}\chi_m^\n(i,\delta_x)G_m^\n(\delta_x,\delta_{x'})\chi_m^\n(j,\delta_{x'}).\]
Once again, we will study the limit of this expression when $n\to\infty$, to obtain a representation of the VRJP on $T$. However under $\nu_V^W(d\beta)$, contrary to $\psi^\n$, $\chi_m^\n(\cdot,\delta_x)$ is not a martingale when $m\neq 0$. Moreover, the term  $(G_m^\n)_{B_m,B_m}$ is not independent of $\beta_{T^\n}$ for $m\neq 0$. Therefore, we cannot use the same argument as in the proof of Theorem \ref{thm:Gh.psi.mix}.

As with $\psi$, we expect $\chi_m^\n(\cdot,\delta_x)$ to converge to a $H_\beta$-harmonic function on $T$, for all $x\in D^\m$ and $\nu_V^W$-almost all $\beta$. When $\psi>0$, we can once again introduce the operator $\Delta^\psi$ in order to study $H_\beta$-harmonic functions (see Proposition \ref{prop:trans.delta.psi}). We can characterise $\Delta^\psi$-harmonic functions with the corresponding Martin boundary $\M^\psi$ and the Martin kernel $K^\psi$. Since the graph is a tree, the Martin boundary is equal to the set $\Omega$ of ends of $\mathcal{T}$, which is deterministic. Note that the boundary condition used to define $\G_m^\n$ corresponds to the identification of $\Omega_x$ to a single point $\delta_x$, for all $x\in D^\m$. We will see that the limit of $\chi_m^\n(\cdot,\delta_x)$ can be expressed with the family of harmonic measures associated $\Delta^\psi$, defined as follows.

For a transient random walk $(X_k)_{k\in\N}$ on a graph $\G=(V,E)$, it is possible to define the limit $X_\infty$ of the trajectory as a $\M$-valued random variable, where $\M$ is the Martin boundary associated with the random walk. Then the family of harmonic measures is defined as $(\mu_x)_{x\in V}$, where $\mu_x$ is the distribution of $X_\infty$ when the walk starts at $x$. For all $x\in V$, $\mu_x$ is a probability measure on $\M$, and for $A\in\mathcal{B}(\M)$, $x\mapsto\mu_x(A)$ is harmonic for the random walk.

We denote by $(\mu^\psi_x)_{x\in T}$ the harmonic measures associated with $\Delta^\psi$. For $\beta\in\calD_V^W$ such that $\psi\equiv 0$, we adopt the convention that $\mu^\psi_x$ is the null measure on $\Omega$ for all $x\in T$. We will now see how under $\nu_V^W(d\beta)$, for all $x\in D^\m$, $\chi_m^\n(\cdot,\delta_x)$ converges to a $H_\beta$-harmonic function related to the harmonic measures $(\mu^\psi_y)_{y\in T}$, and how this gives us a representation of the VRJP on $\mathcal{T}$ for each $m\geq 0$.

\begin{thm}\label{thm:chi.mix}
\begin{itemize}
\item[(i)] For all $m\geq 0$, for $y\in T$ and $x\in D^\m$, we have $\nu_T^W$-almost surely $\chi_m^\n(y,\delta_x)\to\psi(y)\mu_y^\psi(\Omega_x)$. For all $y\in T$, we define the measure $\chi(y,\cdot)=\psi(y)\mu_y^\psi(\cdot)$ on $\Omega$.
\item[(ii)] Let $m\geq 0$ be fixed. For $\nu_T^W$-almost all $\beta$, we define the $|B_m|\times|B_m|$ matrix $\Cc_m$ by
\[(\Cc_m)_{\delta_x,\delta_{x'}}=\begin{cases} 0 &\mbox{ if }x=x', \\ \frac{\chi(x\wedge x',\Omega_x)\chi(x\wedge x',\Omega_{x'})}{\Gh(x\wedge x',x\wedge x')} &\mbox{ otherwise.}\end{cases}\]
From now on, let us write: $\nu_{T,B_m}^W(d\beta,d\rho_m)=\nu_T^W(d\beta)\nu_{B_m}^{\Cc_m}(d\rho_m)$.

For $\nu_T^W$-almost all $\beta$ and for $\rho_m\in\calD_{B_m}^{\Cc_m}$, we define ${\Gc_m=(2\rho_m-\Cc_m)^{-1}}$, as well as $\gc_m:\Omega^2\to\R_+$ a locally constant function, such that for $x,x'\in D^\m$ and $\omega\in\Omega_x$, $\tau\in\Omega_{x'}$, we have $\gc_m(\omega,\tau)=\Gc_m(\delta_x,\delta_{x'})$. Finally, for $\nu_T^W$-almost all $\beta$ and for $\rho_m\in\calD_{B_m}^{\Cc_m}$, for $i,j\in T$, we define
\[G_m(i,j)=\Gh(i,j)+\int_{\Omega^2}\chi(i,d\omega)\chi(j,d\tau)\gc_m(\omega,\tau),\]
and denote by $P^{\beta,\rho_m,i_0}_x$ the distribution  of the Markov jump process started at $x\in V$, where the jump rate from $i$ to $j$ is $\frac{1}{2}W_{i,j}\frac{G_m(i_0,j)}{G_m(i_0,i)}$.

Then the law $P^{VRJP(i_0)}$ is a mixture of these Markov jump processes under $\nu_{T,B_m}^W(d\beta,d\rho_m)$, \textit{i.e.}
\[P^{VRJP(i_0)}[\cdot]=\int P^{\beta,\rho_m,i_0}_{i_0}[\cdot]\nu_{T,B_m}^W(d\beta,d\rho_m).\]
\item[(iii)] The distribution under $\nu_{T,B_m}^W(d\beta,d\rho_m)$ of the jump rates $(\frac{1}{2}W_{i,j}\frac{G_m(i_0,j)}{G_m(i_0,i)})_{i\sim j}$ converges weakly to the distribution of jump rates in the representation described in Theorem \ref{thm:mix.tree}.
\end{itemize}
\end{thm}

Let us now consider the case where $\mathcal{T}$ is a $d$-regular tree, with $d\geq 3$, endowed with constant conductances, \textit{i.e.} $W_e=W>0$ for  all $e\in E$. Then $(\mathcal{T},W)$ is vertex transitive, and from Proposition 3 of \cite{SabZen}, we know that under $\nu_T^W(d\beta)$, $\psi$ is stationary and ergodic. Therefore, depending on $d$ and $W$, we either have $\P[\forall i\in T, \psi(i)=0]=1$, or $\P[\forall i\in T, \psi(i)>0]=1$.

In the first case, from Theorem \ref{thm:Gh.psi.mix} (iv), this means that the VRJP is a.s. recurrent, and therefore admits a unique representation (see Proposition \ref{prop:repr.rec}). Note that in Theorem \ref{thm:chi.mix}, we have a.s. $G_m=\Gh$ for all $m\in\N$, so that all the corresponding representations are indeed equal. The following proposition describes the second case, \textit{i.e.} when the VRJP is a.s. transient. 

\begin{prop}\label{prop:diff.rep.reg.tree}
Let $\mathcal{T}$ be a $d$-regular tree, with $d\geq 3$, endowed with constant conductances $W$ such that $\P[\forall i\in T, \psi(i)>0]=1$. Then the representations of the VRJP given in Theorem \ref{thm:chi.mix} are different for distinct values of $m$, \textit{i.e.} if $m\neq m'$, the distributions of jump rates $(\frac{1}{2}W_{i,j}\frac{G_m(i_0,j)}{G_m(i_0,i)})_{i\sim j}$ under $\nu_{T,B_m}^W(d\beta,d\rho_m)$ and $(\frac{1}{2}W_{i,j}\frac{G_{m'}(i_0,j)}{G_{m'}(i_0,i)})_{i\sim j}$ under $\nu_{T,B_{m'}}^W(d\beta,d\rho_{m'})$ are different for all $i_0\in T$.
\end{prop}

\subsection{Open questions}

A first question concerns the case of $\Z^d$ with constant conductances $W$: is it possible to show that the Martin boundary associated with $\Delta^\psi$ is a.s. trivial for any $W$ such that the VRJP is transient ? In this case, it would prove the unicity of the representation of the VRJP on $\Z^d$ for any constant initial conductances $W$.

Another questions concerns a possible classification of all representations on trees using partitions of the Martin boundary. We have constructed a family of representations from different boundary conditions on the tree, corresponding to some finite partitions of the Martin boundary $\Omega$, more precisely the partition $\Omega=\bigcup_{x\in D^\m}\Omega_x$ for $m\in\N$. It should be possible to define more representations using the same method, with boundary conditions associated with other finite partitions of $\Omega$, where each set in the partition can be written as a finite union of sets $\Omega_x$. To generalise this, we can ask if it is possible to determine which partitions give us a valid representation, and whether all representations can be written in this form, or as a limit of such representations, as in Theorem \ref{thm:chi.mix} (iii).

\subsection{Organisation of the paper}

In Section 3, we expose some useful results concerning the $\beta$ field, as well as basic definitions and properties of the Martin boundary. In Section 4, we prove how all representations of the VRJP have a common form, \textit{i.e.} Theorem \ref{thm:beta.env}. We use these results in Section 5 to study the case of the graph $\Z^d$, and show Theorem \ref{thm:zd.rep} using a local limit theorem in random environment. In Sections 6 and 7, we study the $\beta$ field on trees with our new boundary condition, and show the convergence of the associated Green function. We use this in Section 8, to show that this provides representations of the VRJP (Theorem \ref{thm:chi.mix}), and that they are different in the case of a regular tree (Proposition \ref{prop:diff.rep.reg.tree}).

\section{Technical prerequisite}

\subsection{The random potential $\beta$ on finite graphs}

Let $\G=(V,E)$ be a finite connected nondirected graph, endowed with conductances $(W_e)_{e\in E}$. Let us give some useful properties on the distribution $\nu_V^W$.


\begin{prop}\label{prop:beta.u}[Proposition 2, Theorem 3 in \cite{SabTarZen}]
For $\beta\in\calD_V^W$, let $G=(H_\beta)^{-1}$ be the Green function associated with $\beta$. We define $F:V\times V\to\R$ by
\[F(i,j)=\frac{G(i,j)}{G(j,j)}.\]
Then under $\nu_V^W(d\beta)$, for all $i_0\in V$, we have the following properties:
\begin{itemize}
\item[(i)] $(F(i,i_0))_{i\in V}$ is $(\beta_i)_{i\in V\backslash\{i_0\}}$-measurable.
\item[(ii)] If we denote $\gamma=\frac{1}{2G(i_0,i_0)}$, then $\gamma$ is a Gamma random variable with parameter $(1/2,1)$. Moreover, $\gamma$ is independent of $(\beta_i)_{i\neq i_0}$, and therefore independent of $(F(i,i_0))_{i\in V}$.
\end{itemize}
\end{prop}

This proposition explains the presence of $\gamma$ in the expression on $G$ in Theorem \ref{thm:Gh.psi.mix}. Moreover, it allows us to prove Proposition \ref{prop:beta.tilde.dist}, describing the distribution of the $\tilde\beta$ field.

\begin{proof}[Proof of Proposition \ref{prop:beta.tilde.dist}]
Let $\G=(V,E)$ be an infinite connected nondirected graph, and $(V_n)_{n\in\N}$ an increasing exhausting sequence of finite connected subsets of $V$. For $n\in\N$, let $\G^\n=(\tilde V^\n,\tilde E^\n)$ be the restriction of $\G$ to $V_n$ with wired boundary condition, endowed with conductances $\tilde W^\n$, defined as in section \ref{2.1}. Moreover, for $n\in\N$, we still define $\Gh^\n$ and $\psi^\n$ as in Definition \ref{defi:Gh.psi}.

The proof of Theorem \ref{thm:Gh.psi.mix} (iii) uses the fact that under $\nu_V^W(d\beta,d\gamma)$, there exists a coupling of random fields $(\beta^\n)_{n\in\N}$, such that for all $n\in\N$: $\beta^\n\sim\nu_{\tilde V^\n}^{\tilde W^\n}$ ; $\beta^\n_{V_n}=\beta_{V_n}$ ; and $G^\n=(2\beta^\n-{\tilde W^\n})^{-1}$, where for $i,j\in\tilde V^\n$,
\[G^\n(i,j)=\Gh^\n(i,j)+\frac{1}{2\gamma}\psi^\n(i)\psi^\n(j).\]
We can then apply Proposition \ref{prop:beta.u} to $\beta^\n$ at $i_0$: since $G^\n=(2\beta^\n-{\tilde W^\n})^{-1}$, we have 
\[\beta^\n_i=\tilde\beta^\n_i+\ind_{\{i=i_0\}}\frac{1}{2G^\n(i_0,i_0)}\]
for all $i\in \tilde V^\n$, where
\[\tilde\beta^\n_i=\sum_{j\sim i}\frac{\Wt^\n_{i,j}}{2}\frac{G^\n(i_0,j)}{G^\n(i_0,i)}.\]
According to Proposition \ref{prop:beta.u}, $1/2G^\n(i_0,i_0)$ is a random $\Gamma(1/2,1)$ variable, which is independent of $(\beta^\n_i)_{i\in\tilde V^\n\backslash\{i_0\}}$. Moreover, for $i\neq i_0$, $\tilde\beta^\n_i=\beta^\n_i$, and 
\[\tilde\beta^\n_{i_0}=\sum_{j\sim i_0}\frac{\Wt^\n_{i,j}}{2}\frac{G^\n(i_0,j)}{G^\n(i_0,i_0)}=\sum_{j\sim i_0}\frac{\Wt^\n_{i,j}}{2}F^\n(j,i_0),\]
so that $\tilde\beta^\n$ is $(\beta^\n_i)_{i\in\tilde V^\n\backslash\{i_0\}}$-measurable, and therefore independent of $G^\n(i_0,i_0)$.

Taking the limit when $n\to\infty$, we deduce that $1/2G(i_0,i_0)$ is a random $\Gamma(1/2,1)$ variable, independent from the $\tilde\beta$ field, where for $i\in V$,
\[\tilde\beta_i=\sum_{j\sim i}\frac{W_{i,j}}{2}\frac{G(i,j)}{G(i,i)}=\beta_i-\ind_{\{i=i_0\}}\frac{1}{2G(i_0,i_0)}.\]
Since the Laplace transform of a $\Gamma(1/2,1)$ variable is, for $t\geq 0$,
\[\int e^{-t\gamma}\frac{\ind_{\{\gamma>0\}}}{\sqrt{\pi\gamma}}e^{-\gamma}d\gamma=\frac{1}{\sqrt{1+t}},\]
and given the Laplce transform of $\nu_V^W$ in Proposition \ref{prop:nu.dist.inf}, we now know that the Laplace transform of $\tilde\beta$ is, for $\lambda\in\R_+^V$ with finite support,
\[\E[e^{-\langle\lambda,\tilde\beta\rangle}]=\frac{1}{\sqrt{1+\lambda_{i_0}}} e^{-\sum_{i\sim j}W_{i,j}(\sqrt{1+l_i}\sqrt{1+l_j}-1)}\prod_{i\neq i_0}\frac{1}{\sqrt{1+l_i}}.\]
\end{proof}

On finite graphs, the distribution $\nu_V^W$, and more generally $\nu_V^{W,\eta}$ for $\eta\in\R_+^V$, behaves well with respect to restriction, as shown in the next proposition, which is a generalization of Proposition \ref{prop:beta.restr.sotr}.
\begin{prop}\label{prop:restr.beta}[Lemma 4 in \cite{SabZen}]
Let us fix $U\subset V$ and $\eta\in(\R_+)^V$. Then, under $\nu_V^{W,\eta}(d\beta)$, we have:
\begin{itemize}
\item[(i)] $\beta_U$ is distributed according to $\nu_U^{W_{U,U},\hat{\eta}}$, where
\[\hat{\eta}=\eta_U+W_{U,U^c}\ind_{U^c}.\]
\item[(ii)] Conditionally on $\beta_U$, $\beta_{U^c}$ is distributed according to $\nu_{U^c}^{\Wc,\check{\eta}}$, where
\[\Wc=W_{U^c,U^c}+W_{U^c,U}((H_\beta)_{U,U})^{-1}W_{U,U^c} \mbox{ and } \check{\eta}=\eta_{U^c}+W_{U^c,U}((H_\beta)_{U,U})^{-1}\eta_U.\]
\end{itemize}
\end{prop}
Proposition \ref{prop:beta.restr.sotr} is a direct consequence of (i), in the case where $\eta=0$. Moreover, (ii) is useful to extend a potential $\beta_U\sim\nu_U^{W_{U,U},\hat\eta}$ where $\hat\eta=W_{U,U^c}\ind_{U^c}$ into a potential $\beta_V\sim\nu_V^W$, using the distribution of $\beta_{U^c}$ conditionally on $\beta_U$.

\subsection{Green function and sums over paths}\label{3.2}

Let us still consider a finite connected nondirected graph $\G=(V,E)$ endowed with conductances $W$. For $\beta\in\calD_V^W$, it will be useful to express the Green function $G=H_\beta^{-1}$ as a sum over paths in $\G$. We first introduce some notations for sets of paths.

\begin{defi}\label{defi:sets.paths}
\begin{itemize}
\item[(i)] For $i,j\in V$, we denote by $\calP^V_{i,j}$ the set of paths $\sigma$ from $i$ to $j$ in $V$, \textit{i.e.} the set of finite sequences $\sigma=(\sigma_0,...\sigma_l)$ in $V$, where $\sigma_0=i$, $\sigma_l=j$ and $\sigma_k\sim\sigma_{k+1}$ for $0\leq k\leq l-1$. We denote by $|\sigma|=l$ the length of the path $\sigma$.
\item[(ii)] For $U\subset V$, $i\in U$ and $j\notin U$, we denote by $\o{\calP}^U_{i,j}$ the set of paths $\sigma\in\calP^V_{i,j}$ such that $\sigma_k\in U$ for $0\leq k\leq |\sigma|-1$.
\item[(iii)] For $U\subset V$ and $i,j\in V$, we denote by $\calP^V_{i,U,j}$ the set of paths $\sigma\in\calP^V_{i,j}$ such that $\sigma_k\in U$ for some $k\in\llb 0,|\sigma|\rrb$.
\item[(iv)] For $i,j\in V$ and $\sigma\in\calP^V_{i,j}$, we define the following notations:
\[W_\sigma=\prod_{k=0}^{|\sigma|-1} W_{\sigma_k,\sigma_{k+1}}, (2\beta)_\sigma=\prod_{k=0}^{|\sigma|}2\beta_{\sigma_k} \mbox{ and } (2\beta)_\sigma^-=\prod_{k=0}^{|\sigma|-1}2\beta_{\sigma_k}.\]
\end{itemize}
\end{defi}

We get the following expressions, in terms of sums over paths, for $G$ and related quantities.

\begin{prop}\label{prop:g.somme}[Proposition 6 in \cite{SabTarZen}]
Let $\beta\in\calD_V^W$. Then:
\begin{itemize}
\item[(i)] For $i,j\in V$,
\[G(i,j)=\sum_{\sigma\in\calP^V_{i,j}} \frac{W_\sigma}{(2\beta)_\sigma}.\]
In particular, for $U\subset V$ we denote $\Gh^U=((H_\beta)_{U,U})^{-1}$, then for $i,j\in U$, we obtain
\[\Gh^U(i,j)=\sum_{\sigma\in\calP^U_{i,j}} \frac{W_\sigma}{(2\beta)_\sigma}.\]
\item[(ii)] For $i,j\in V$,
\[F(i,j)=\sum_{\sigma\in\o\calP^{V\backslash\{j\}}_{i,j}}\frac{W_\sigma}{(2\beta)_\sigma^-}=\sum_{z\sim j} \Gh^{V\backslash\{j\}}(i,z)W_{z,j}.\]
\item[(iii)] For $U\subset V$ and $i,j\in U^c$,
\[\sum_{\sigma\in\calP^V_{i,U,j}}\frac{W_\sigma}{(2\beta)_\sigma^-}=\sum_{z_1,z_2\in U}\l(\sum_{\sigma\in\o\calP^{U^c}_{i,z_1}}\frac{W_\sigma}{(2\beta)_\sigma^-}\r)G(z_1,z_2)\l(\sum_{\sigma\in\o\calP^{U^c}_{j,z_2}}\frac{W_\sigma}{(2\beta)_\sigma^-}\r).\]
In particular, if $U=\{z\}$, this becomes
\[\sum_{\sigma\in\calP^V_{i,\{z\},j}}\frac{W_\sigma}{(2\beta)_\sigma}=F(i,z)G(z,z)F(j,z)=F(i,z)G(z,j).\]
\end{itemize}
\end{prop}

\begin{rem}\label{rem:gh.fh.sum}
If $\G=(V,E)$ is now an infinite graph, let $(V_n)_{n\in\N}$ be an increasing sequence of finite connected subsets of $V$ such that $V=\cup_{n\in\N}V_n$. For $\beta\in\calD_V^W$ and $n\geq 0$, we define $\Gh^\n=\Gh^{V_n}=((H_\beta)_{V_n,V_n})^{-1}$ as in Definition \ref{defi:Gh.psi}. Then from Proposition \ref{prop:g.somme} (i), we get
\[\Gh^\n(i,j)=\sum_{\sigma\in\calP^{V_n}_{i,j}}\frac{W_\sigma}{(2\beta)_\sigma}\]
for $n\geq 0$ and $i,j\in V_n$. From Theorem \ref{thm:Gh.psi.mix} (i), under $\nu_V^W(d\beta)$ the increasing sequence $(\Gh^\n(i,j))_{n\in\N}$ converges almost surely to $\Gh(i,j)$. Hence, we get
\[\Gh(i,j)=\sum_{\sigma\in\calP^{V}_{i,j}}\frac{W_\sigma}{(2\beta)_\sigma}\]
for $i,j\in V$.

Let us also define $\Fh^\n(i,j)=\frac{\Gh^\n(i,j)}{\Gh^\n(j,j)}$ and $\Fh(i,j)=\frac{\Gh(i,j)}{\Gh(j,j)}$, for all $i,j\in V$ and $n\geq\max(|i|,|j|)$. Then, from Proposition \ref{prop:g.somme} (ii) we have
\[\Fh^\n(i,j)=\sum_{\sigma\in\o\calP^{V_n\backslash\{j\}}_{i,j}}\frac{W_\sigma}{(2\beta)_\sigma^-}\xrightarrow[n\to\infty]{}\Fh(i,j)=\sum_{\sigma\in\o\calP^{V\backslash\{j\}}_{i,j}}\frac{W_\sigma}{(2\beta)_\sigma^-},\]
where the convergence is true $\nu_V^W$-almost surely.
\end{rem}

\subsection{Martin boundary and harmonic functions}\label{3.3}

Let us give more details about the theory of Martin boundaries. The following results can be found in \cite{Woe}.

Let $\G=(V,E)$ be an infinite graph, we consider an irreducible random walk $(X_n)_{n\in \N}$ on $\G$, whose transition matrix is $P$, where $P_{i,j}=0$ if $\{i,j\}\notin E$ (\textit{i.e.} we assume that $(X_n)$ is a nearest-neighbour random walk). Moreover, we assume that $(X_n)$ is transient.

Let us denote by $\P_x$ the distribution of the random walk started at $x\in V$, and by $g$ the associated Green function, \textit{i.e.}
\[g(x,y)=\sum_{n\in\N}\P_x[X_n=y].\]
We also denote
\[f(x,y)=\P_x[\exists n\in\N, X_n=y]=\frac{g(x,y)}{g(y,y)}.\]
For all $y\in V$, $g(\cdot,y)$ is harmonic at any $x\in V\backslash\{y\}$, \textit{i.e.} for $x\neq y$, $g(x,y)=\sum_{z\sim x}P_{x,z} g(z,y)$. This is still true for $f(\cdot,y)$. The Martin Kernel, defined below using $f$, as well as the Martin boundary, will allow us to represent all positive harmonic functions for the random walk.

\begin{defi}\label{defi:Martin}
Let us fix a reference point $\phi\in V$.
\begin{itemize}
\item[(i)] The Martin kernel is the function $K:V^2\to\R_+$ defined by
\[K(x,y)=\frac{f(x,y)}{f(\phi,y)}=\frac{g(x,y)}{g(\phi,y)}.\]
\item[(ii)] The Martin compactification is the smallest compactification $\hat{V}$ of $V$, which is unique up to a homeomorphism, so that $K(\cdot,\cdot)$ extends continuously to $V\times\hat{V}$. The Martin boundary is defined as $\M=\hat{V}\backslash V$.
\end{itemize}
\end{defi}

\begin{thm}\label{thm:repr.Martin}
Let us denote by $\H^+$ the set of positive harmonic functions on $V$. Then for all $h\in\H^+$, there is a Borel measure $\mu^h$ on $\M$ such that for all $x\in V$,
\[h(x)=\int_\M K(x,\alpha) \mu^h(d\alpha).\]
\end{thm}

\begin{rem}\label{rem:Mart.triv}
If, for all $x\in V$ and for all sequences $(y_n)_{n\geq 1}$ going to infinity, we have $K(x,y_n)\to 1$, then the Martin boundary is trivial, \textit{i.e.} reduced to a single point. According to Theorem \ref{thm:repr.Martin}, in this case, all positive harmonic functions are constant.
\end{rem}

Since $(X_n)$ is transient, we almost surely have $X_n\to\infty$, in the sense that for all finite subset $U\subset V$, $\{n\in\N,X_n\in U\}$ is almost surely finite. Thanks to the Martin boundary, we can now describe this convergence more precisely.

\begin{thm}\label{thm:mes.harm}
For all $x\in V$, $(X_n)$ converges $\P_x$-a.s. to a $\M$-valued random variable $X_\infty$. The distribution of $X_\infty$ under $\P_x$, denoted by $\mu_x$, verifies
\[\mu_x(B)=\int_B K(x,\alpha)\mu_\phi(d\alpha)\]
for all $B\subset\M$ and $x\in V$. 
\end{thm}

The space $(\M,\mathcal{B}(\M),(\mu_x)_{x\in V})$ is called Poisson boundary. Moreover, we call harmonic measures, or exiting measures, the family $(\mu_x)_{x\in V}$.

In the case where $\mathcal{T}=(T,E)$ is an infinite tree, the Martin compactification will coincide with another, which does not depend on the random walk defined by $P$, but simply on the geometry of the tree $\mathcal{T}$.

\begin{defi}\label{defi:extremite}
Let us fix an arbitrary root $\phi$ for $\mathcal{T}$.
\begin{itemize}
\item[(i)]
We call infinite ray in $T$ an infinite self-avoiding path starting at $\phi$, \textit{i.e.} a sequence $\omega=(\omega_k)_{k\in\N}$ of distinct vertices in $T$, such that $\omega_k\sim \omega_{k+1}$ for $k\in\N$ and $\omega_0=\phi$. The set of infinite rays, also called the set of ends of $T$, is denoted by $\Omega$.
\item[(ii)]
If $\omega,\xi\in\Omega$, we denote $K_{\omega,\xi}=\max\{k\in\N,\omega_k=\xi_k\}$. We can also define, if $x\in T$, $K_{\omega,x}=\max\{k\leq|x|,\omega_k=[\phi,x]_k\}$. We then set $O_\omega^k=\{\xi\in\Omega,K_{\omega,\xi}\geq k\}\cup\{x\in T,K_{\omega,x}\geq k\}$.
\item[(iii)]
We define the end topology on $T\cup\Omega$, which is discrete on $T$, and such that $(O_\omega^k)_{k\in\N}$ is a basis of neighbourhoods at $\omega\in\Omega$ .
\end{itemize}
\end{defi}

\begin{prop}\label{prop:comp.ext}
The end topology on $T\cup\Omega$ does not depend on the choice of $\phi$, and is induced by the following metric:
\[d(x,y)=\begin{cases}0 &\mbox{if } x=y\\ e^{-N_{x,y}} &\mbox{otherwise},\end{cases}\]
for $x,y\in T\cup\Omega$. Moreover $T\cup\Omega$ is compact, and called the end compactification.
\end{prop}

\begin{thm}\label{thm:Martin.arbre}
\begin{itemize}
\item[(i)]
Let $(X_n)$ be a nearest-neighbour random walk on $\mathcal{T}$, that we assume to be transient. Then the Martin compactification coincides with the end compactification, and we can identify $\M$ to $\Omega$, and set $\hat{T}=T\cup\Omega$.
\item[(ii)]
The Martin kernel on $T\times\hat{T}$ is locally constant, with
\[K(x,\omega)=K(x,x\wedge\omega), \mbox{ where } x\wedge\omega=\omega_{N_{\omega,x}}\]
for $x\in T, \omega\in\Omega$.
\end{itemize}
\end{thm}

We also have an expression of harmonic measures $\mu_x$ on the tree. For $x\in T$, we denote by $\Omega_x$ the set of ends for the subtree $T_x$, \textit{i.e.} $\Omega_x=\{\omega\in\Omega,\exists k\in\N,\omega_k=x\}$. Moreover, we denote by $U_x=T_x\backslash\{x\}$. Then:

\begin{prop}\label{prop:mes.harm.arbre}
For $x\neq\phi$ and $i\in T$,
\[\mu_i(\Omega_x)=\ind_{\{i\in U_x\}}(1-f(i,x))) + f(i,x)\frac{1-f(x,\p{x})}{1-f(\p{x},x)f(x,\p{x})}.\]
\end{prop}

\begin{rem}
From Carathéodory's extension theorem, this entirely describes the expression of $\mu_\phi$. From Theorem \ref{thm:mes.harm}, we can then describe all harmonic measures using $f$.
\end{rem}

\section{Distributions of arbitrary representations}

\subsection{A common expression for jump rates: Proof of Theorem \ref{thm:beta.env}}

Let $\G=(V,E)$ be a locally finite connected graph, endowed with conductances $(W_{i,j})_{i,j\in V}$ such that $W_{i,j}=W_{j,i}>0$ if $\{i,j\}\in E$, and $W_{i,j}=0$ otherwise. We still denote by $P^{VRJP(i_0)}$ the law of the time-changed VRJP on $(\G,W)$, started at $i_0\in V$. Let us first show that the distribution of the $\tilde\beta$ field (see Proposition \ref{prop:beta.tilde.dist}) appears in all representations of the VRJP.

Recall that for all $r\in\calJ_V^E$ and $i\in V$, we define $r_i=\sum_{j\sim i} r_{i,j}$.

\begin{prop}\label{prop:beta.tilde.env}
Let $i_0\in V$ be fixed, and let $\calR(dr)$ be the distribution of a random environment representing $P^{VRJP(i_0)}$, in the sense of Definition \ref{def:rep}.

Then under $\calR(dr)$, $(r_i)_{i\in V}$ has the same distribution as the field $\tilde{\beta}$ rooted at $i_0$, \textit{i.e.} its Laplace transform is
\[\int e^{-\langle\lambda,r\rangle}\mathcal{R}(dr)=e^{-\sum_{i\sim j}W_{i,j}(\sqrt{1+\lambda_i}\sqrt{1+\lambda_j}-1)}\prod_{i\neq i_0}\frac{1}{\sqrt{1+\lambda_i}},\]
for $\lambda\in\R_+^V$ with finite support.
\end{prop}

\begin{proof}
Let $i_0\in V$ be fixed, let $\calR(dr)$ be the distribution of a random environment representing $P^{VRJP(i_0)}$, \textit{i.e.}
\[P^{VRJP(i_0)}[\cdot]=\int P^{r}_{i_0}[\cdot] \mathcal{R}(dr),\]
where $P^{r}$ is the
distribution of the Markov jump process with jump rate from $i$ to $j$ given by $r_{i,j}$. 

Let us prove that under $\calR(dr)$, $(r_i)_{i\in V}$ has the same distribution as the $\tilde\beta$ field from the standard representation.

\begin{lem}\label{lem:r=exp(u)}
There exists a random field $(u_i)_{i\in V}\in\R^V$ such that $\calR$-almost surely, $r_{i,j}=\frac{W_{i,j}}{2}e^{u_j-u_i}$ for $i\sim j$.
\end{lem}

\begin{rem}
Since the random field $(u_i)_{i\in V}$ is defined up to an additive constant, we can set $u_{i_0}=0$ a.s. without loss of generality.
\end{rem}

\begin{proof}[Proof of Lemma \ref{lem:r=exp(u)}]
For $r\in\calJ_V^E$, let us define $t_{i,j}=\frac{2}{W_{i,j}}r_{i,j}$ for all $i\sim j$. Then to prove this lemma, it is enough to show that for any cycle $\sigma=(\sigma_0,...,\sigma_n)$, we have $\calR$-a.s. $t_\sigma:=\prod_{k=0}^{n-1} t_{\sigma_k,\sigma_{k+1}}=1$. Since $\G$ is connected, we only need to prove this for cycles $\sigma$ such that $\sigma_0=i_0$.

Recall that we denote by $(Z_t)_{t\geq 0}$ the canonical process on $D([0,\infty),V)$. Let $P^{MJP}$ be the distribution of the Markov jump process with jump rates $\frac{1}{2}W_{i,j}$. Then, according to Theorem 3 from \cite{ST16}, for all $T\geq 0$ the law of $(Z_t)_{t\leq T}$ under $P^{VRJP(i_0)}$ is absolutely continuous with respect to its law under $P^{MJP}_{i_0}$, and its Radon-Nykodim derivative is 
\[\frac{e^{-\sum_{i\sim j}W_{i,j}(\sqrt{1+l_i}\sqrt{1+l_j}-1)}}{e^{-\sum_{i\in V}\frac{1}{2}W_i l_i}}\prod_{i\neq i_0}\frac{1}{\sqrt{1+l_i}},\]
where $W_i=\sum_{j\sim i}W_{i,j}$, and $l_i=\int_0^T \ind_{\{Z_t=i\}}dt$ is the local time at $i$.

Let $\sigma$ be a cycle such that $\sigma_0=\sigma_{|\sigma|}=i_0$. We denote by $\sigma^n$ the $n$-th concatenation of $\sigma$, and for $T\geq 0$ and $\tau$ a path in $\G$, by $\l\{(Z_t)_{t\leq T}\sim\tau\r\}$ the event where the discrete trajectory of $(Z_t)_{t\geq T}$ follows the path $\tau$. Then we have, for $n\geq 1$ and $T\geq 0$,
\[P^{VRJP(i_0)}[(Z_t)_{t\leq T}\sim\sigma^n]=\int \ind_{\{(z_t)_{t\leq T}\sim\sigma^n\}} \frac{e^{-\sum_{i\sim j}W_{i,j}(\sqrt{1+l_i}\sqrt{1+l_j}-1)}}{e^{-\sum_{i\in V}\frac{1}{2}W_i l_i}}\prod_{i\neq i_0}\frac{1}{\sqrt{1+l_i}}  P^{MJP}_{i_0}(dz).\] 
However, since the random environment $(r_{i,j})_{i\sim j}$ gives a representation of the VRJP as a mixture of Markov processes, we also have
\begin{align*}
P^{VRJP(i_0)}&[(Z_t)_{t\leq T}\sim\sigma^n] = \int P^r_{i_0}[(Z_t)_{t\leq T}\sim\sigma^n] \mathcal{R}(dr) \\
&= \int \l(\int\ind_{\{(z_t)_{t\leq T}\sim\sigma^n\}} \frac{e^{-\sum_{i\in V}r_i l_i}\l(\prod_{k=0}^{|\sigma|-1}r_{\sigma_k,\sigma_{k+1}}\r)^n}{e^{-\sum_{i\in V}\frac{1}{2}W_i l_i}\l(\prod_{k=0}^{|\sigma|-1}\frac{1}{2}W_{\sigma_k,\sigma_{k+1}}\r)^n}P^{MJP}_{i_0}(dz)\r)\mathcal{R}(dr) \\
&= \int\ind_{\{(z_t)_{t\leq T}\sim\sigma^n\}} \int \frac{e^{-\sum_{i\in V}r_i l_i}}{e^{-\sum_{i\in V}\frac{1}{2}W_i l_i}}(t_\sigma)^n\mathcal{R}(dr) P^{MJP}_{i_0}(dz)
\end{align*}

Let us fix $\epsilon>0$, and define the event $A_{\sigma,\epsilon}=\l\{t_\sigma\geq 1+\epsilon\r\}$. Then we get
\[P^{VRJP(i_0)}[(Z_t)_{t\leq T}\sim\sigma^n] \geq \int\ind_{\{(z_t)_{t\leq T}\sim\sigma^n\}} \int \ind_{A_{\sigma,\epsilon}}e^{-\sum_{i\in V}r_i l_i}(1+\epsilon)^n\mathcal{R}(dr) P^{MJP}_{i_0}(dz).\]
Let $M>0$ be such that under $\calR(dr)$, ${\P[A_{\sigma,\epsilon}\cap B_M]\geq \P[A_{\sigma,\epsilon}]/2}$, where $B_M=\{\forall i\in V, r_i\leq M\}$. Note that $T=\sum_{i\in V} l_i$, so that
\begin{align*}
P^{VRJP(i_0)}[(Z_t)_{t\leq T}\sim\sigma^n] &\geq \int\ind_{\{(z_t)_{t\leq T}\sim\sigma^n\}} \int \ind_{A_{\sigma,\epsilon}\cap B_M} e^{-MT}(1+\epsilon)^n\mathcal{R}(dr) P^{MJP}_{i_0}(dz) \\
&\geq \frac{e^{-MT}}{2}(1+\epsilon)^n \P[A_{\sigma,\epsilon}] P^{MJP}_{i_0}[(Z_t)_{t\leq T}\sim\sigma^n].
\end{align*}
On the other hand, we also have
\begin{align*}
P^{VRJP(i_0)}[(Z_t)_{t\leq T}\sim\sigma^n] &= \int\ind_{\{(z_t)_{t\leq T}\sim\sigma^n\}} \frac{e^{-\sum_{i\sim j}W_{i,j}(\sqrt{1+l_i}\sqrt{1+l_j}-1)}}{e^{-\sum_{i\in V}\frac{1}{2}W_i l_i}}\prod_{i\neq i_0}\frac{1}{\sqrt{1+l_i}}  P^{MJP}_{i_0}(dz)\\
&\leq  e^{M'T}P^{MJP}_{i_0}[(Z_t)_{t\leq T}\sim\sigma^n],
\end{align*}
where $M'=\max\{\frac{1}{2}W_{\sigma_k},0\leq k<|\sigma|\}$. Since $P^{MJP}_{i_0}[(Z_t)_{t\leq T}\sim\sigma^n]>0$ for all $T>0$ and $n\in\N$, we get
\[\frac{e^{-MT}}{2}(1+\epsilon)^n \P[A_{\sigma,\epsilon}]\leq e^{M'T}.\]
Taking $n\to\infty$ for fixed $T>0$ shows that $\P[A_{\sigma,\epsilon}]=0$. As a result, we have almost surely $t_\sigma\leq 1$.

For $\epsilon>0$, we now set $A'_{\sigma,\epsilon}=\{t_\sigma\leq 1-\epsilon\}$. Using the same notations as before, and the fact that a.s. $t_\sigma\leq 1$, we get
\begin{align*}
P^{VRJP(i_0)}[(Z_t)_{t\leq T}\sim\sigma^n] &\leq \int\ind_{\{(z_t)_{t\leq T}\sim\sigma^n\}} \int e^{M'T}\l(\ind_{{A'_{\sigma,\epsilon}}^c}+ \ind_{A'_{\sigma,\epsilon}}(1-\epsilon)^n\r)\mathcal{R}(dr) P^{MJP}_{i_0}(dz) \\
&\leq e^{M'T}\l(\P[{A'_{\sigma,\epsilon}}^c]+\P[A'_{\sigma,\epsilon}](1-\epsilon)^n\r)P^{MJP}_{i_0}[(Z_t)_{t\leq T}\sim\sigma^n].
\end{align*}
On the other hand, on the event $\{(Z_t)_{t\leq T}\sim\sigma^n\}$, we have $l_i\leq T$ for all $i\in\{\sigma_k,0\leq k<|\sigma|\}$ and $l_i=0$ for all other $i\in V$. As a result, for such trajectories,
\[\frac{e^{-\sum_{i\sim j}W_{i,j}(\sqrt{1+l_i}\sqrt{1+l_j}-1)}}{e^{-\sum_{i\in V}\frac{1}{2}W_i l_i}}\prod_{i\neq i_0}\frac{1}{\sqrt{1+l_i}} \geq \frac{e^{-M''T}}{(1+T)^{\frac{|\sigma|-1}{2}}},\]
where $M''=\sum_{i,j\in\{\sigma_k\}}W_{i,j}$, so that,
\[P^{VRJP(i_0)}[(Z_t)_{t\leq T}\sim\sigma^n]\geq \frac{e^{-M''T}}{(1+T)^{\frac{|\sigma|-1}{2}}}P^{MJP}_{i_0}[(Z_t)_{t\leq T}\sim\sigma^n].\]
As before, this yields
\[e^{M'T}\l(\P[{A'_{\sigma,\epsilon}}^c]+\P[A'_{\sigma,\epsilon}](1-\epsilon)^n\r)\geq \frac{e^{-M''T}}{(1+T)^{\frac{|\sigma|-1}{2}}}\]
for all $T>0$ and $n\in\N$. Taking first $n\to\infty$, then $T\to 0$, we get that uner $\calR(dr)$, $\P[{A'_{\sigma,\epsilon}}^c]=1$. Therefore, we can conclude that $t_\sigma=1$ $\calR$-almost surely.
\end{proof}

In order to identify the distribution of $(r_i)_{i\in V}$ under $\calR(dr)$, we obtain their Laplace transform as the density of cyclic trajectories of $(Z_t)_{t\geq 0}$ under $P^{VRJP(i_0)}$ with respect to $P^{MJP}_{i_0}$. Indeed, given a cyclic trajectory $(z_t)_{t\geq 0}$ in $\G$, started at $i_0$, we denote by $\sigma$ the associated cyclic path in $\G$, and $(l_i)_{i\in V}$ the local times, so that $T=\sum_{i\in V} l_i$, and $l_i>0$ if and only if $i\in\{\sigma_k,0\leq k<|\sigma|\}$. Then the Radon-Nykodim derivative at $(z_t)_{t\geq 0}$ of $P^{VRJP(i_0)}$ with respect to $P^{MJP}_{i_0}$ is almost surely
\[\frac{e^{-\sum_{i\sim j}W_{i,j}(\sqrt{1+l_i}\sqrt{1+l_j}-1)}}{e^{-\sum_{i\in V}\frac{1}{2}W_i l_i}}\prod_{i\neq i_0}\frac{1}{\sqrt{1+l_i}},\]
but also
\[\int \frac{e^{-\sum_{i\in V}r_i l_i}}{e^{-\sum_{i\in V}\frac{1}{2}W_i l_i}}\mathcal{R}(dr)\]
since $t_\sigma=1$ $\calR$-almost surely. Therefore, for all finite connected subset $U$ of $V$, and almost all $(l_i)_{i\in V}\in (\R_+^*)^U\times\{0\}^{V\backslash U}$, we have
\[e^{-\sum_{i\sim j}W_{i,j}(\sqrt{1+l_i}\sqrt{1+l_j}-1)}\prod_{i\neq i_0}\frac{1}{\sqrt{1+l_i}}=\int e^{-\sum_{i\in V}r_i l_i}\mathcal{R}(dr)=\E[e^{-\sum_{i\in V}r_i l_i}].\]
Since these are continuous functions of $(l_i)_{i\in V}$, this equality is true for all $(l_i)_{i\in V}\in\R_+^V$ with finite support. As a result, under $\calR(dr)$, $(r_i)_{i\in V}$ has the same Laplace transform as the field $\tilde\beta$, associated with the standard representation of the VRJP started at $i_0$ (see Proposition \ref{prop:beta.tilde.dist}), and therefore the same distribution.
\end{proof}

\begin{proof}[Proof of Theorem \ref{thm:beta.env}]
Let $i_0\in V$ be fixed, and $\calR(dr)$ be the distribution of a random environment representing $P^{VRJP(i_0)}$. Thanks to Proposition \ref{prop:beta.tilde.env}, we know the distribution of $(r_i)_{i\in V}$ under $\calR(dr)$. Note that the distribution of a $\Gamma(1/2,1)$ variable is $\frac{\ind_{\{\gamma>0\}}}{\sqrt{\pi\gamma}}e^{-\gamma}d\gamma$, and that its Laplace transfrom is given by
\[\int e^{-t\gamma}\frac{\ind_{\{\gamma>0\}}}{\sqrt{\pi\gamma}}e^{-\gamma}d\gamma=\frac{1}{\sqrt{1+t}}\]
for $t\geq 0$. From now on, we denote $\calR(dr,d\gamma)=\calR(dr)\otimes \frac{\ind_{\{\gamma>0\}}}{\sqrt{\pi\gamma}}e^{-\gamma}d\gamma$, which is a distribution on $\calJ_V^E\times\R$.

For $r\in\calJ_V^E$ and $\gamma>0$, we now define $(\beta_i)_{i\in V}$ by $\beta_i=r_i + \ind_{\{i=i_0\}}\gamma$ for $i\in V$. Then under $\calR(dr,d\gamma)$, the Laplace transform of $\beta$ is, for $\lambda\in\R_+^V$ with finite support,
\begin{align*}
\int e^{-\langle\lambda,\beta\rangle} \mathcal{R}(dr)\frac{\ind_{\{\gamma>0\}}}{\sqrt{\pi\gamma}}e^{-\gamma}d\gamma &=\int e^{-\sum_{i\in V}\lambda_i r_i} \mathcal{R}(dr)\int e^{-\lambda_{i_0}\gamma}\frac{\ind_{\{\gamma>0\}}}{\sqrt{\pi\gamma}}e^{-\gamma}d\gamma  \\
& =\frac{1}{\sqrt{1+\lambda_{i_0}}} e^{-\sum_{i\sim j}W_{i,j}(\sqrt{1+l_i}\sqrt{1+l_j}-1)}\prod_{i\neq i_0}\frac{1}{\sqrt{1+l_i}},
\end{align*}
\textit{i.e.} $\beta$ is distributed according to $\nu_V^W$ (see Proposition \ref{prop:nu.dist.inf}). We can then define $\calR(dr,d\gamma)$-a.s. $\Gh:V\times V\to\R_+$ and $\psi:V\to\R_+$ thanks to Theorem \ref{thm:Gh.psi.mix}. Moreover, by analogy with the standard representation, let $G(i_0,\cdot):V\to\R_+$ be defined by:
\[G(i_0,i)=\frac{1}{2\gamma}e^{u_i},\]
where $(u_i)_{i\in V}$ was introduced in Lemma \ref{lem:r=exp(u)}. This way, under $\calR(dr,d\gamma)$, for all $i\neq j\in V$ the jump rate $r_{i,j}$ can be written as 
\[r_{i,j}=\frac{W_{i,j}}{2}\frac{G(i_0,j)}{G(i_0,i)}.\]

Let us set $h(i)=G(i_0,i)-\Gh(i_0,i)$ for all $i\in V$. Then $H_\beta h=0$. Indeed, for $i\neq i_0$, we have
\begin{align*}
2\beta_i G(i_0,i)-\sum_{j\sim i} W_{i,j}G(i_0,j)&=2r_i G(i_0,i)-\sum_{j\sim i}2r_{i,j}G(i_0,i)\\
&=0=2\beta_i\Gh(i_0,i)-\sum_{j\sim i}W_{i,j}\Gh(i_0,j),
\end{align*}
and for $i=i_0$,
\begin{align*}
2\beta_{i_0} G(i_0,i_0)-\sum_{j\sim i_0} W_{i_0,j}G(i_0,j)&=\frac{r_{i_0}+\gamma}{\gamma} -\sum_{j\sim i_0}\frac{r_{i_0,j}}{\gamma}\\
&=1=2\beta_{i_0}\Gh(i_0,i_0)-\sum_{j\sim i_0}W_{i_0,j}\Gh(i_0,j).
\end{align*}
As a result, $G(i_0,\cdot)$ can be written, for all $i\in V$, as 
\[G(i_0,i)=\Gh(i_0,i)+h(i),\]
where $h:V\to\R_+$ is a non-negative $H_\beta$-harmonic function.
\end{proof}

\subsection{The recurrent and transient cases: Proofs of Propositions \ref{prop:repr.rec} and \ref{prop:trans.delta.psi}}\label{4.2}

\begin{proof}[Proof of Proposition \ref{prop:repr.rec}]
We assume that $(\G,W)$ is such that the VRJP is almost surely recurrent.

Let $(r_{i,j})_{i\sim j}$ be fixed jump rates on $V$, such that the associated Markov chain is recurrent. We denote by $P^r_{i_0}$ its distribution when started at $i_0$. Note that under $P^r_{i_0}$, the time spent at a vertex $i$ before jumping is an exponential variable with parameter $r_i$, and the probability to then jump to a specific neighbour $j$ is $\frac{r_{i,j}}{r_i}$.

Let us then define the following functions of the trajectory $(Z_t)$: for $i\in V$ and $n\geq 1$, we define $\delta t^\n_i$ as the time spent by $(Z_t)$ at the vertex $i$ during its $n$th visit to $i$, and $v^\n_i$ the neighbour of $i$ towards which the process jumps after its $n$th visit to $i$. Under $P^r_{i_0}$, since the process is recurrent, these random variables are well-defined for all $i\in V$ and $n\geq 1$. Moreover, the sequences $(\delta t^\n_i)_{n\geq 1}$ and $(v^\n_i)_{n\geq 1}$ are independent, so thanks to the law of large numbers, we have almost surely
\[\o{\delta t}_i:=\lim_{n\to\infty}\frac{\sum_{k=1}^n \delta t^{(k)}_i}{n}=\frac{1}{r_i} \mbox{  and  } \o{p}_{i,j}:=\lim_{n\to\infty}\frac{\sum_{k=1}^n \ind_{\{ v^{(k)}_i=j\}}}{n}=\frac{r_{i,j}}{r_i},\]
for all $i,j\in V$.

Let now $\mathcal{R}(dr)$ be the distribution of a random environment representing the VRJP on $(\G,W)$, \textit{i.e.} $P^{VRJP(i_0)}[\cdot]=\int P^r_{i_0}[\cdot]\mathcal{R}(dr)$. Since the VRJP is a.s. recurrent, then under $\mathcal{R}(dr)$, $P^r_{i_0}$ is a.s. the distribution of a recurrent Markov chain. Moreover, under $P^{VRJP(i_0)}$, $\o{\delta t}_i$ and $\o{p}_{i,j}$ are a.s. well-defined for all $i,j\in V$, and $\l(\frac{\o{p}_{i,j}}{\o{\delta t}_i}\r)_{i\sim j}$ is distributed according to $\mathcal{R}$. Since these functions of the trajectory do not depend on the chosen representation, the distribution $\mathcal{R}$ is uniquely determined.
\end{proof}

\begin{proof}[Proof of Proposition \ref{prop:trans.delta.psi}]
Let $(\G,W)$ be such that the VRJP is almost surely recurrent. Since, according to Theorem \ref{thm:Gh.psi.mix} (iii), we have $P^{VRJP(i_0)}[\cdot]=\int P^{\beta,\gamma,i_0}_{i_0}[\cdot]\nu_V^W(d\beta,d\gamma)$, then under $\nu_V^W(d\beta,d\gamma)$, the Markov process with distribution $P^{\beta,\gamma,i_0}_{i_0}$ is a.s. transient. From Theorem \ref{thm:Gh.psi.mix} (iv), this means that under $\nu_V^W(d\beta)$, we have a.s. $\psi(i)>0$ for all $i\in V$, which proves \textit{(i)}.

Let us now consider the random conductance model with conductances $c^\psi_{i,j}=W_{i,j}\psi(i)\psi(j)$. We denote by $\pi^\psi_i=\sum_{j\sim i}c^\psi_{i,j}$ the corresponding invariant measure, where $\pi^\psi_i=\psi(i)\sum_{j\sim i}W_{i,j}\psi(j)=2\beta_i\psi(i)^2$ since $\psi$ is $H_\beta$-harmonic. Let $P^\psi$ be the distribution of the associated random walk, whose transition probability from $i$ to $j$ is
\[p^\psi_{i,j}=\frac{c_{i,j}^\psi}{\pi^\psi_i}=\frac{W_{i,j}\psi(j)}{2\beta_i\psi(i)}\]
for $i,j\in V$. Moreover, let us denote by $g^\psi$ the Green kernel associated with $P^\psi$, defined for $i,j\in V$ as $g^\psi(i,j)=\sum_{k\in\N}P^\psi_i[X_k=j]$, where $(X_k)_{k\in\N}$ denotes the canonical process on $V^\N$. Then we have
\begin{align*}
g^\psi(i,j)&=\sum_{k\in\N}\sum_{\sigma\in\calP^T_{i,j} , |\sigma|=k}P^\psi_i[(X_0,...,X_k)=\sigma] = \sum_{\sigma\in\calP^T_{i,j}} \prod_{k=0}^{|\sigma|-1} \frac{W_{\sigma_k,\sigma_{k+1}}\psi(\sigma_{k+1})}{2\beta_{\sigma_k}\psi(\sigma_k)}\\
&= \frac{\psi(j)}{\psi(i)}\sum_{\sigma\in\calP^T_{i,j}}\frac{W_\sigma}{(2\beta)_\sigma^-} = \frac{\psi(j)}{\psi(i)}2\beta_j\Gh(i,j),
\end{align*}
where under $\nu_V^W(d\beta)$, $\Gh(i,j)$ is a.s. finite for all $i,j\in V$, from Theorem \ref{thm:Gh.psi.mix} (i). As a result, we have almost surely $g^\psi(i,j)<\infty$, therefore the random walk $P^\psi$ is transient almost surely, proving \textit{(ii)}.

Let $\Delta^\psi=(p^\psi_{i,j}-\ind_{\{i=j\}})_{i,j\in V}$ be the discrete Laplacian associated with $P^\psi$. We will say that a function $\varphi:V\to\R$ is $\Delta^\psi$-harmonic if $(\Delta^\psi \varphi)(i)=\l(\sum_{j\sim i} p^\psi_{i,j} \varphi(j)\r) - \varphi(i)=0$ for all $i\in V$. Therefore, a function $\varphi$ is $\Delta^\psi$-harmonic if and only if for any $i\in V$, 
\[2\beta_i\psi(i)\varphi(i)-\sum_{j\sim i}W_{i,j}\psi(j)\varphi(j)=0,\]
\textit{i.e.} if and only if $\psi \varphi$ is $H_\beta$-harmonic, which concludes the proof of \textit{(iii)}.
\end{proof}

\section{Representations of the VRJP on $\Z^d$: Proof of Theorem \ref{thm:zd.rep}}

Let us now consider the case where $\G=(V,E)$ is the $\Z^d$ lattice, endowed with constant edge weights, \textit{i.e.} $W_{i,j}=W>0$ for all $i\sim j$. For $x\in\R^d$, we will denote by $|x|$ its Euclidean norm. We fix $i_0=0$. 

\subsection{Recurrence and transience of the VRJP on $\Z^d$}

For $d=2$, the VRJP on $(\G,W)$ is a.s. recurrent for all $W>0$, according to Theorem 1.1 in \cite{BHS18}. Therefore, the representation of $P^{VRJP(0)}$ as a mixture of Markov jump processes is unique (see Proposition \ref{prop:repr.rec}). If $d\geq 3$, Corollary 1 in \cite{SabTar} tells us that for small enough $W$, the VRJP is a.s. recurrent, in which case the representation of $P^{VRJP(0)}$ is once again unique. Let us now show that for large enough $W$, even though the VRJP is almost surely transient, the representation is still unique.

From Corollary 3 in \cite{SabTar}, we know that for $W$ large enough, the VRJP is a.s. transient. From now on, we consider such $W$. Then thanks to Proposition \ref{prop:trans.delta.psi}, under $\nu_V^W(d\beta)$ we have a.s. $\psi(i)>0$ for all $i\in V$. Moreover, we can define the Markov operator $\Delta^\psi$ and, for $h:V\to\R_+$, $h$ is $H_\beta$-harmonic if and only if $h/\psi$ is $\Delta^\psi$-harmonic. In light of Remark \ref{rem:rep.fonc.harm}, in order to show that the representation of the VRJP is unique, we need to show that the only positive $\Delta^\psi$-harmonic functions are constants, \textit{i.e.} that the Martin boundary $\M^\psi$ associated with $\Delta^\psi$ is almost surely trivial. To do this, we will need a local limit theorem in random environment, found in \cite{ADS16}.

\subsection{Local limit theorem for random walk in random conductances}

Let us consider the random conductances model on $\G=(\Z^d,E_d)$, with $d\geq 2$. Let $\P$ be a distribution on the set of conductances $(\R_+^*)^{E_d}$, such that under $\P(d\omega)$, we have a.s. $0<\omega_{i,j}<\infty$ for all $i\sim j$. For $\omega\in(\R_+^*)^{E_d}$, let $P^\omega$ be the distribution of the continuous-time constant speed random walk associated with $\omega$. This is the Markov jump process with jump rate from $i$ to $j$ given by $\frac{\omega_{i,j}}{\pi^\omega_i}$, where $\pi^\omega_i=\sum_{j\sim i}\omega_{i,j}$. This way under $P^\omega$, the holding time of $(Z_t)_{t\geq 0}$ at each point is an exponential variable of parameter $1$, which justifies the term "constant speed". Finally, we denote by $q^\omega$ the heat kernel, \textit{i.e.} the transition density of the walk with respect to $\pi^\omega$: for $x,y\in \Z^d$ and $t\geq 0$,
\[q^\omega(t,x,y)=\frac{P^\omega_x[Z_t=y]}{\pi^\omega_y}.\]
The following theorem from \cite{ADS16} is a local limit theorem for $q^\omega$, under ergodicity and integrability assumptions.

\begin{thm}\label{thm:llthk}[Theorem 1.11 in \cite{ADS16}]
Let us assume that $\P(d\omega)$ is stationary and ergodic with respect to translations of $\Z^d$, and that there exist $p,q\in(1,\infty]$ satisfying $1/p+1/q<2/d$ such that $\E[\omega_{i,j}^p]<\infty$ and $\E[\omega_{i,j}^{-q}]<\infty$ for all $i\sim j$.

Then for $0<T_1<T_2$ and $K>0$, we have $\P$-a.s.
\[\lim_{n\to\infty}\sup_{|x|\leq K}\sup_{t\in[T_1,T_2]}\l| n^d q^\omega(n^2 t,0,\lf nx\rf) - ak_t(x) \r|=0,\]
where $\lf nx\rf=(\lf nx_1\rf,...,\lf nx_d\rf)$, $a=1/\E[\pi^\omega_0]$ and $k_t$ is the Gaussian heat kernel with some deterministic covariance matrix $\Sigma^2$, \textit{i.e.}
\[k_t(x)=\frac{1}{\sqrt{(2\pi t)^d\det(\Sigma^2)}}e^{-\frac{x^t(\Sigma^2)^{-1}x}{2t}}.\]
\end{thm}

\begin{rem}\label{rem:isom.covar}
If $\P(d\omega)$ is also stationary with respect to all isometries of $\Z^d$, then the limiting Brownian motion must be as well, therefore its deterministic covariance matrix has the form $\Sigma^2=\sigma^2I_d$, where $\sigma^2>0$.
\end{rem}

This also provides a local limit theorem for the Green kernel $g^\omega$, defined for $\omega\in(\R_+^*)^{E_d}$ and $x,y\in Z^d$ by
\[g^\omega(x,y)=\int_0^\infty q^\omega(t,x,y)dt.\]
This result was also mentioned in \cite{ADS16}, we give here the details for the proof of a slightly stronger result\footnote{I would like to thank Sebastian Andres for his help regarding the details of this proof.}, that insures the uniform convergence for $x$ in an annulus.

\begin{thm}[Variant of Theorem 1.14 in \cite{ADS16}]\label{thm:lltgk}
For $d\geq 3$, under the assumptions of Theorem \ref{thm:llthk}, we have $\P$-a.s.
\[\lim_{n\to\infty}\sup_{1\leq|x|\leq 2}|n^{d-2}g^\omega(0,\lf nx\rf) - a g_{BM}(0,x)|=0,\]
where $g_{BM}$ is the Green kernel associated with the Brownian motion with covariance matrix $\Sigma^2$, \textit{i.e.}
\[g_{BM}(0,x)=\int_0^\infty k_t(x)dt=\frac{\Gamma(d/2-1)}{2\pi^{d/2}\det(\Sigma^2)^{1/2}}(x^t(\Sigma^2)^{-1}x)^{1-d/2}.\]
\end{thm}

\begin{proof}
This result is obtained by integrating in Theorem \ref{thm:llthk}. Moreover, we will need the following bounds on $q^\omega$, which are true almost surely.

Firstly, Theorem 1.6 in \cite{ADS16-EJP} gives a short-range bound, which also applies to $k_t$: $\P$-a.s. there are constants $C, c_1,c_2>0$ such that for $t\geq Cn|x|$,
\[q^\omega(t,0,\lf nx\rf)\leq c_1 t^{-d/2}e^{-\frac{c_2 (n|x|)^2}{t}},\]
and for all $t\geq 0$,
\[k_t(x)\leq c_1 t^{-d/2}e^{-\frac{c_2 |x|^2}{t}}.\]

Now, for a long-range bound: using Corollaries 11 and 12 from \cite{Dav93}, there exists $\P$-a.s. a constant $c_3>0$ such that for all $t\geq 0$, we have
\[q^\omega(t,0,\lf nx\rf)\leq \frac{1}{\sqrt{\pi^\omega_0\pi^\omega_{\lf nx\rf}}}e^{-c_3 n|x|} .\]
Note that the integrability assumption implies that $\E[\rho^\omega_0]<\infty$, where $\rho^\omega_i=\sum_{l\sim i}\frac{1}{\omega_{i,l}}$. Therefore, for $|x|\leq 2$,
\[\frac{1}{\pi^\omega_{\lf nx\rf}}\leq\rho^\omega_{\lf nx\rf}\leq\sum_{|y|\leq 2n}\rho^\omega_y \mbox{, as well as } \frac{1}{\pi^\omega_0}\leq\sum_{|y|\leq 2n}\rho^\omega_y,\]
and thanks to the ergodic theorem, $\P$-a.s. there exist $c_4>0$ and $N_0\leq 1$ such that for $n\geq N_0$,
$\sum_{|y|\leq 2n}\rho^\omega_y\leq c_4(2n)^d\E[\rho^\omega_0]$.  For such $n$, we get
\[q^\omega(t,0,\lf nx\rf)\leq c_5 n^d e^{-c_3 n|x|} .\]

Using these bounds, we now know that for $n\geq N_0$ and $1\leq|x|\leq 2$, we have $\P$-a.s.
\begin{align*}
|n^{d-2}&g^\omega(0,\lf nx\rf) - a g_{BM}(0,x)|= \l| n^d\int_0^\infty q^\omega(n^2t,0,\lf nx\rf)dt - a\int_0^\infty k_t(x)dt \r|\\
&\leq n^d\int_0^{2C/n}q^\omega(n^2t,0,\lf nx\rf)dt + n^d\int_{2C/n}^{T_1}q^\omega(n^2 t,0,\lf nx\rf)dt + a\int_{0}^{T_1}k_t(x)dt\\
&\quad+ \int_{T_1}^{T_2}|n^d q^\omega(n^2 t,0,\lf nx\rf) - ak_t(x)|dt +n^d\int_{T_2}^\infty q^\omega(n^2 t,0,\lf nx\rf)dt + a\int_{T_2}^\infty k_t(x)dt\\
&\leq C'n^{2d-1}e^{-c_3 n} + (1+a)\int_0^{T_1}c_1t^{-d/2}e^{-c_2/t}dt  + (1+a)\int_{T_2}^\infty c_1t^{-d/2}e^{-c_2/t}dt \\
&\quad+ (T_2-T_1)\sup_{|x|\leq 2}\sup_{t\in[T_1,T_2]}|n^d q^\omega(n^2 t,0,\lf nx\rf) - ak_t(x)|.
\end{align*}
Let $\epsilon>0$. Since $t\mapsto c_1t^{-d/2}e^{-c_2/t}$ is integrable on $(0,\infty)$, we can fix $T_1,T_2>0$ independently of $x$ such that
\[\int_0^{T_1}c_1t^{-d/2}e^{-c_2/t}dt +\int_{T_2}^\infty c_1t^{-d/2}e^{-c_2/t}dt <\frac{\epsilon}{2(1+a)}.\]
Then
\begin{align*}
\sup_{1\leq|x|\leq 2}|n^{d-2}g^\omega(0,\lf nx\rf) - a g_{BM}(0,x)|&\leq (T_2-T_1)\sup_{|x|\leq 2}\sup_{t\in[T_1,T_2]}|n^d q^\omega(n^2 t,0,\lf nx\rf) - ak_t(x)| \\
&\quad + C'n^{2d-1}e^{-c_3 n} + \frac{\epsilon}{2},
\end{align*}
so that from Theorem \ref{thm:llthk}, there exists $N\geq N_0$ independent of $x$ such that for $n\geq N$,
\[\sup_{1\leq|x|\leq 2}|n^{d-2}g^\omega(0,\lf nx\rf) - a g_{BM}(0,x)|\leq \epsilon,\]
which is true $\P$-almost surely.
\end{proof}

\begin{rem}\label{rem:green.dis.con}
Let us fix conductances $\omega\in(\R_+^*)^{E_d}$. We denote by $(\tilde Z_n)_{n\in\N}$ the discrete version of $(Z_t)_{t\geq 0}$. Then, for $x,y\in\Z^d$,
\begin{align*}
g^\omega(x,y)&=\int_0^\infty \frac{P^\omega_x[Z_t=y]}{\pi^\omega_y}dt=\frac{1}{\pi^\omega_y} E^\omega_x\l[\int_0^\infty \ind_{\{Z_t=y\}}dt\r]\\
&=\frac{1}{\pi^\omega_y} E^\omega_x\l[\sum_{n=0}^\infty \ind_{\{\tilde Z_n=y\}}\r]=\frac{1}{\pi_y^\omega}\sum_{n=0}^\infty P^\omega_x[\tilde Z_n=y],
\end{align*}
where $\sum_{n=0}^\infty P^\omega_x[\tilde Z_n=\cdot]$ is the Green kernel associated with $(\tilde Z_n)_{n\in\N}$ under $P^\omega_x$. Indeed, since under $P^\omega_x$ the holding time of $Z$ at each point is an exponential variable of parameter $1$, the expected time spent by $(Z_t)_{t\geq 0}$ at $y$ is exactly the expected number of visits of $y$ by $(\tilde{Z}_n)_{n\in\N}$.
\end{rem}

\subsection{Martin boundary associated with $\Delta^\psi$}

We return to the VRJP on $\Z^d$, $d\geq 3$, with constant initial conductances $W$ large enough so that the VRJP is almost surely transient. From Proposition \ref{prop:trans.delta.psi}, under $\nu_V^W(d\beta)$, we then have a.s. $\psi(i)>0$ for all $i\in V$. Moreover, the random conductance model associated with conductances $c^\psi_{i,j}=W_{i,j}\psi(i)\psi(j)$ defines almost surely a transient random walk. We still denote by $\Delta^\psi$ the discrete Laplacian, and define $\pi^\psi_i=\sum_{j\sim i}c^\psi_{i,j}=2\beta_i\psi(i)^2$, as well as $g^\psi$ the corresponding Green kernel:
\begin{align*}
g^\psi(x,y)=\sum_{k=0}^\infty P^\psi_x[X_k=y]=\frac{\psi(j)}{\psi(i)}2\beta_j\Gh(i,j).
\end{align*}
We want to identify the Martin boundary $\M^\psi$ associated with $\Delta^\psi$, by studying the behaviour at infinity of the Martin kernel $K^\psi$, defined by
\[K^\psi(x,y)=\frac{g^\psi(x,y)}{g^\psi(0,y)}\]
for all $x,y\in\Z^d$. In order to do this, we will use Theorem \ref{thm:lltgk}.

\begin{prop}\label{prop:mart.triv}
There exists $\o{W}>0$ such that for $W>\o{W}$, under $\nu_V^W(d\beta)$, the Martin boundary $\M^\psi$ is almost surely trivial.
\end{prop}

\begin{proof}
Note that under $\nu_V^W(d\beta)$, the distribution of the random conductances $c^\psi_{i,j}$ is stationary and ergodic with respect to all isometries of $\Z^d$, thanks to Proposition 3 of \cite{SabZen}. Moreover, for $W$ large enough, the integrability assumption of Theorem \ref{thm:lltgk} will be verified.

\begin{lem}\label{lem:intgb.cond}
Consider the graph $\G=(V=\Z^d,E=E_d)$, with $d\geq 3$, with constant initial conductances $W$. Then for all $p\geq 1$, there exists $W_p>0$ such that for $W>W_p$, for all $i\sim j$, under $\nu_{V}^W(d\beta)$ we have
\[\E[(\psi(i)\psi(j))^p]<\infty \mbox{ and } \E[(\psi(i)\psi(j))^{-p}]<\infty.\]
\end{lem}

\begin{proof}
The proof is the same as for Lemma 9 (i) in \cite{SabZen}, and uses Theorem 1 of \cite{DisSpeZir}, which gives a control on moments of arbitrarily large order.
\end{proof}

Let $W^t>0$ be such that for $W>W^t$, the VRJP on $(\Z^d,W)$ is almost surely transient. Moreover, we define $\o{W}=\max(W^t,W_{d+1})$. From now on, we assume that $W>\o{W}$, so that thanks to Lemma \ref{lem:intgb.cond}, under $\nu_{V}^W(d\beta)$, for all $i\sim j$ we have
\[\E[(c_{i,j}^\psi)^{d+1}]<\infty \mbox{ and } \E[(c_{i,j}^\psi)^{-(d+1)}]<\infty,\]
Therefore, according to Theorem \ref{thm:lltgk} and Remarks \ref{rem:isom.covar} and \ref{rem:green.dis.con}, there exists $\sigma^2>0$ such that $\nu_{V}^W$-almost surely,
\[\sup_{1\leq|z|\leq 2}|n^{d-2}(\pi^\psi_{\lf nz\rf})^{-1}g^\psi(0,\lf n z\rf) - a g_{BM}(0,z)|\xrightarrow[n\to\infty]{} 0,\]
where $a=1/\E[\pi^\psi_0]$ and $g_{BM}$ is the Green kernel associated with a Brownian motion with covariance matrix $\sigma^2 I_d$, \textit{i.e.}
\[g_{BM}(0,z)=\frac{\Gamma(d/2-1)}{2\pi^{d/2}\sigma^2}|z|^{2-d}.\]

Using this result, we have $\nu_{V}^W$-almost surely, for any sequence $(y_n)_{n\geq 1}$ such that $|y_n|\to\infty$, $g^\psi(0,y_n)\sim_{n\to\infty}a \pi^\psi_{y_n}g_{BM}(0,y_n)$ . Indeed, for such a sequence $(y_n)$, let us define $m_n=\lf|y_n|\rf$ and $z_n=y_n/m_n$. Then, since $1\leq|z_n|\leq 2$ for all $n\geq 1$, we have
\begin{align*}
\l|\frac{g^\psi(0,y_n)}{a \pi^\psi_{y_n} g_{BM}(0,y_n)}-1\r|&=\l|\frac{(\pi^\psi_{\lf m_n z_n\rf})^{-1}g^\psi(0,m_n z_n)}{ m_n^{2-d}ag_{BM}(0,z_n)}-1\r|\\
&=\l|\frac{m_n^{d-2}(\pi^\psi_{\lf m_n z_n\rf})^{-1}g^\psi(0,\lf m_n z_n\rf)-ag_{BM}(0,z_n)}{ag_{BM}(0,z_n)}\r|\\
&\leq \frac{\sup_{1\leq|z|\leq 2}|m_n^{d-2}(\pi^\psi_{\lf m_n z\rf})^{-1}g^\psi(0,\lf m_n z\rf) - a g_{BM}(0,z)|}{a\inf_{1\leq|z|\leq 2}g_{BM}(0,z)}\xrightarrow[n\to\infty]{}0
\end{align*}
$\nu_{V}^W$-almost surely.

Moreover, for $x\in\Z^d$ fixed, let $\psi^x$ be the translated function defined by $\psi^x(y)=\psi(y-x)$. Then $\psi^x$ and $\psi$ have the same distribution under $\nu_V^W(d\beta)$, therefore we have $\nu_{V}^W$-a.s., for all $(y_n)_{n\geq 1}$ such that $|y_n|\to\infty$,
\[g^\psi(x,y_n)=g^{\psi^x}(0,y_n-x)\sim_{n\to\infty}a\pi^\psi_{y_n}g_{BM}(0,y_n-x),\]
since $|y_n-x|\to\infty$ and $\pi^{\psi^x}_{y_n-x}=\pi^\psi_{y_n}$. Let us denote by $A_x$ the $\nu_{V}^W$-almost sure event where this is true. Since $\Z^d$ is denumerable, $\bigcap_{x\in \Z^d}A_x$ is still $\nu_{V}^W$-almost sure. Therefore, we have $\nu_{V}^W$-a.s. that for all $x\in\Z^d$, for all $(y_n)_{n\geq 1}$ such that $|y_n|\to\infty$,
\[K^\psi(x,y_n)\sim_{n\to\infty}\frac{a\pi^\psi_{y_n}g_{BM}(0,y_n-x)}{a\pi^\psi_{y_n}g_{BM}(0,y_n)}=\frac{|y_n-x|^{2-d}}{|y_n|^{2-d}}\xrightarrow[n\to\infty]{} 1.\]
As a result, from Remark \ref{rem:Mart.triv}, the Martin boundary associated with $\Delta^\psi$ is $\nu_{V}^W$-a.s. trivial.
\end{proof}

Let $\calR(dr)$ be the distribution of an environment representing $P^{VRJP(0)}$ on $\Z^d$ endowed with constant initial conductances $W>\o{W}$.

For $r\in\calJ_{V}^E$ and $\gamma>0$, we define $\beta$ by $\beta_i=r_i+\ind_{\{i=0\}}\gamma$. According to Theorem \ref{thm:beta.env}, under $\calR(dr,d\gamma)$ we then have $\beta\sim\nu_V^W$. We define $\Gh$ and $\psi$ as functions of $\beta$, as in  Theorem \ref{thm:Gh.psi.mix}, and we can write
\[r_{i,j}=\frac{W_{i,j}}{2}\frac{G(0,j)}{G(0,i)},\]
where $G(0,i)=\Gh(0,i)+h(i)$ for all $i\in\Z^d$, with $h$ a $H_\beta$-harmonic function. Since $W$ is large enough so that under $\nu_V^W(d\beta)$, $\psi(i)>0$ for all $i\in\Z^d$ almost surely, the operator $\Delta^\psi$ is well-defined, and $h/\psi$ is $\Delta^\psi$-harmonic. However, according to Proposition \ref{prop:mart.triv} the Martin boundary associated with $\Delta^\psi$ is $\nu_V^W$-a.s. trivial, therefore positive $\Delta^\psi$-harmonic functions are almost surely constant. As a result, there is a nonnegative random variable $g$ such that for all $i\in\Z^d$, we have $\calR$-a.s.
\[G(0,i)=\Gh(0,i)+g\psi(i).\]

In particular, $g=(G(0,0)-\Gh(0,0))/\psi(0)$, so $g$ can be written as a function of $(\beta,\frac{1}{2G(0,0)})$, and therefore has a function of $((r_i)_{i\in\Z^d},\gamma)$. Since according to Proposition \ref{prop:beta.tilde.env}, under $\calR(dr,d\gamma)$ the distribution of $((r_i)_{i\in\Z^d},\gamma)$ does not depend on the chosen representation $\calR$, this shows that the distribution of the jump rates $r_{i,j}=\frac{W_{i,j}}{2}\frac{G(0,j)}{G(0,i)}$ is uniquely determined, \textit{i.e.} that the representation is unique.

\begin{rem}
Note that we can identify the distribution of $g$ using the standard representation. This shows that under $\calR(dr),d\gamma$, we have $g=\psi(0)/2\gamma'$, where $\gamma'$ is a $\Gamma(1/2,1)$ random variable independent from $(\beta_i)_{i\in\Z^d}$.
\end{rem}

\section{The potential $\beta_m^\n$ on trees}

Let $T$ be an infinite tree, that is locally finite. We still fix an arbitrary root $\phi$. In this section, we define a $\beta$ field on the restrictions of $T$ with the boundary conditions introduced in Section \ref{2.4}. This way, we can apply Theorem \ref{thm:beta.mix} to these finite restrictions.

\subsection{Construction of $\beta_m^\n$ on $\G_m^\n$}

For all $n\in\N$, under $\nu_T^W(d\beta)$, $\beta_{T^\n}\sim\nu_{T^\n}^{W^\n,\eta^\n}$ from Proposition \ref{prop:nu.dist.inf}, where $W^\n=W_{T^\n,T^\n}$ and $\eta^\n=W_{T^\n,(T^\n)^c}\ind_{(T^\n)^c}$. As usual, for $\beta\in\calD_T^W$, we define $H_\beta=2\beta-W$. For $n\in\N$, let us take $V_n=T^\n$ in Definition \ref{defi:Gh.psi}, so that we get ${\Gh^\n=\Gh^{T^\n}=((H_\beta)_{T^\n,T^\n})^{-1}}$ and $\psi^\n=\Gh^\n\eta^\n$.

Let us fix $n\geq m\geq 0$. To represent the VRJP on $\G_m^\n$ using Theorem \ref{thm:beta.mix}, we need to introduce a potential $\beta_m^\n$ on $\Tt_m^\n=T^\n\cup B_m$, distributed according to $\nu_{\Tt_m^\n}^{\Wt_m^\n}$. According to Proposition \ref{prop:restr.beta}, this is true if and only if: $(\beta_m^\n)_{T^\n}\sim\nu_{T^\n}^{\Wh_m^\n,\hat{\eta}_m^\n}$, and conditionally on $(\beta_m^\n)_{T^\n}$, $(\beta_m^\n)_{B_m}\sim\nu_{B_m}^{\Wc_m^\n}$, where
\begin{align*}
&\Wh_m^\n=(\Wt_m^\n)_{T^\n,T^\n} \mbox{, } \hat{\eta}_m^\n=(\Wt_m^\n)_{T^\n,B_m}\ind_{B_m} \\
\mbox{ and } &\Wc_m^\n=(\Wt_m^\n)_{B_m,B_m}+(\Wt_m^\n)_{B_m,T^\n}((H_\beta^\n)_{T^\n,T^\n})^{-1}(\Wt_m^\n)_{T^\n,B_m}
\end{align*}
and where we denote by $H_\beta^\n=2\beta_m^\n-\Wt_m^\n$ the Schrödinger operator associated with $\beta_m^\n$. Note that thanks to the way $\Wt_m^\n$ was defined, we have $\Wh_m^\n=(\Wt_m^\n)_{T^\n,T^\n}=W^\n$. Moreover, for $i\in T^\n$,
\[(\hat\eta_m^\n)_i=\sum_{b\in B}(\Wt_m^\n)_{i,b}=\ind_{i\in D^\n}\sum_{j\in S(i)}W_{i,j}=W_{\{i\},D^{(n+1)}}\ind_{D^{(n+1)}}=\eta^\n_i.\]

Therefore, let us define $(\beta_m^\n)_{T^\n}=\beta_{T^\n}$. Hence, under $\nu_T^W(d\beta)$ we have $(\beta_m^\n)_{T^\n}\sim\nu_{T^\n}^{W^\n,\eta^\n}$, where $\nu_{T^\n}^{W^\n,\eta^\n}=\nu_{T^\n}^{\Wh_m^\n,\hat\eta_m^\n}$. It remains to extend $(\beta_m^\n)_{T^\n}$ into a field $\beta_m^\n$ over all of $\Tt_m^\n$. To do this, we define $(\beta_m^\n)_{B_m}$ through its distribution conditionally on $(\beta_m^\n)_{T^\n}$, which we set to be $\nu_{B_m}^{\Wc_m^\n}$. Note that $(\Wt_m^\n)_{B_m,B_m}=0$ and $(H_\beta^\n)_{T^\n,T^\n}=2(\beta_m^\n)_{T^\n}-(\Wt_m^\n)_{T^\n,T^\n}=2\beta_{T^\n}-W^\n=(H_\beta)_{T^\n,T^\n}$, therefore
\begin{align*}
\Wc_m^\n&=(\Wt_m^\n)_{B_m,B_m}+(\Wt_m^\n)_{B_m,T^\n}((H_\beta^\n)_{T^\n,T^\n})^{-1}(\Wt_m^\n)_{T^\n,B_m}\\
&=(\Wt_m^\n)_{B_m,T^\n}\Gh^\n(\Wt_m^\n)_{T^\n,B_m}.
\end{align*}
From Proposition \ref{prop:restr.beta}, the field $\beta_m^\n$ constructed this way on $\Tt_m^\n$ is distributed according to $\nu_{\Tt_m^\n}^{\Wt_m^\n}$ under $\nu_T^W(d\beta)\nu_{B_m}^{\Wc_m^\n}(d(\beta_m^\n)_{B_m})$.

\begin{rem}\label{rem:aj.couplage}
Since $(\beta_m^\n)_{B_m}$ has only been defined through its distribution conditionally on $\beta_{T^\n}$, we will later have some freedom to choose a coupling of the sequence $((\beta_m^\n)_{B_m})_{n\geq m}$ such that the matrix sequence $\Hc_\beta^\n=2(\beta_m^\n)_{B_m}-\Wc_m^\n$ converges almost surely.
\end{rem}

For $n\geq m$, let us denote by $G_m^\n=(H_\beta^\n)^{-1}=(2\beta_m^\n-\Wt_m^\n)^{-1}$ the Green function associated with $\beta_m^\n$ on $\G_m^\n$. Then, from Theorem \ref{thm:beta.mix}, we know that the law of the time-changed VRJP on $\G_m^\n$ with respect to $\Wt_m^\n$, started at $i_0\in T^\n$, is a mixture of Markov jump processes under $\nu_T^W(d\beta)\nu_{B_m}^{\Wc_m^\n}(d(\beta_m^\n)_{B_m})$, where the jump rate from $i$ to $j$ is $\frac{1}{2}(\Wt_m^\n)_{i,j}\frac{G_m^\n(i_0,j)}{G_m^\n(i_0,i)}$. In order to obtain a result on the infinite tree $\mathcal{T}$, we will show that for all $i,j\in T$, $G_m^\n(i,j)$ converges almost surely when $n\to\infty$.

\subsection{Expression of $G_m^\n(i,j)$ as a sum over paths}

To show the convergence of $G_m^\n$, it will be useful to express it as a sum over paths in $\calT$. We will use notations and results presented in Section \ref{3.2}.

Let $\beta\in\calD_T^W$ and $(\beta_m^\n)_{B_m}\in\calD_{B_m}^{\Wc_m^\n}$. For $i,j\in T$, let us fix some $n_0>m$ such that $i,j\in T^{(n_0)}$. For $n\geq n_0$, by applying Proposition \ref{prop:g.somme} (i) to $\beta_m^\n$ in the graph $\G_m^\n$, we get
\[G_m^\n(i,j)=\sum_{\sigma\in\calP^{\Tt_m^\n}_{i,j}} \frac{(\Wt_m^\n)_\sigma}{(2\beta_m^\n)_\sigma}.\]
This sum over paths can be decomposed as follows: a path $\sigma\in\calP^{\Tt_m^\n}_{i,j}$ can either cross some vertex in $B_m$, in which case $\sigma\in\calP^{\Tt_m^\n}_{i,B_m,j}$, or never cross any vertex in $B_m$, in which case $\sigma\in\calP^{T^\n}_{i,j}$. As a result, we have $\calP^{\Tt_m^\n}_{i,j}=\calP^{T^\n}_{i,j}\cup\calP^{\Tt_m^\n}_{i,B_m,j}$, where $\calP^{T^\n}_{i,j}\cap\calP^{\Tt^\n}_{i,B,j}=\emptyset$, so
\begin{align*}
G_m^\n(i,j)&= \sum_{\sigma\in\calP^{T^\n}_{i,j}} \frac{(\Wt_m^\n)_\sigma}{(2\beta_m^\n)_\sigma}+\sum_{\sigma\in\calP^{\Tt_m^\n}_{i,B_m,j}} \frac{(\Wt_m^\n)_\sigma}{(2\beta_m^\n)_\sigma} \\
&=\sum_{\sigma\in\calP^{T^\n}_{i,j}} \frac{(\Wt_m^\n)_\sigma}{(2\beta_m^\n)_\sigma} + \sum_{b,b'\in B_m} \l(\sum_{\sigma\in\o{\calP}^{T^\n}_{i,b}}\frac{(\Wt_m^\n)_\sigma}{(2\beta_m^\n)_\sigma^-}\r) G_m^\n(b,b') \l(\sum_{\sigma\in\o{\calP}^{T^\n}_{j,b'}}\frac{(\Wt_m^\n)_\sigma}{(2\beta_m^\n)_\sigma^-}\r),
\end{align*}
from Proposition \ref{prop:g.somme}. Note that for any $\sigma\in\calP^{T^\n}_{i,j}$, $\frac{(\Wt_m^\n)_\sigma}{(2\beta_m^\n)_\sigma}=\frac{W_\sigma}{(2\beta)_\sigma}$, since $(\Wt_m^\n)_{T^\n,T^\n}=W_{T^\n,T^\n}$ and $(\beta_m^\n)_{T^\n}=\beta_{T^\n}$. As a result,
\[\sum_{\sigma\in\calP^{T^\n}_{i,j}} \frac{(\Wt_m^\n)_\sigma}{(2\beta_m^\n)_\sigma}=\sum_{\sigma\in\calP^{T^\n}_{i,j}} \frac{W_\sigma}{(2\beta)_\sigma}=\Gh^\n(i,j).\]

Moreover, note that for $x\in D^\m$ and $y\in T^\n$, 
\begin{align*}
\sum_{\sigma\in\o{\calP}^{T^\n}_{y,\delta_x}}\frac{(\Wt_m^\n)_\sigma}{(2\beta_m^\n)_\sigma^-}&=\sum_{z\sim\delta_x}\l(\sum_{\sigma_1\in\calP^{T^\n}_{y,z}}\frac{(\Wt_m^\n)_{\sigma_1}}{(2\beta_m^\n)_{\sigma_1}}\r)(\Wt_m^\n)_{z,\delta_x}
=\sum_{z\in T_x\cap D^\n}\Gh^\n(y,z)\sum_{z'\in S(z)}W_{z,z'}\\
&=\Gh^\n(y,\cdot)W_{T^\n,(T^\n)^c}\ind_{T_x\backslash T_x^\n} = \chi_m^\n(y,\delta_x),
\end{align*}
from Definition \ref{defi:chi}. As a result, we have
\[G_m^\n(i,j)= \Gh^\n(i,j)+\sum_{b,b'\in B_m}\chi_m^\n(i,b)G_m^\n(b,b')\chi_m^\n(j,b').\]
We will show that under $\nu_T^W(d\beta)$, this expression converges almost surely when $n\to\infty$. From Theorem \ref{thm:Gh.psi.mix} (i), we already know that $\Gh^\n(i,j)$ converges to $\Gh(i,j)$, let us now study the respective limits of $\chi_m^\n$ and $(G_m^\n)_{B_m,B_m}$.

\section{Convergence of $G_m^\n$}

The goal of this section is to prove the convergence of $\chi_m^\n$ and $(G_m^\n)_{B_m,B_m}$. We will first describe the Martin boundary associated with $\Delta^\psi$, and the harmonic measures $(\mu_i^\psi)_{i\in T}$.

\subsection{Process associated with $\Delta^\psi$ and Martin boundary}

First, note that for $\beta\in\calD_T^W$, we have either $\psi(i)>0$ for all $i\in T$, or $\psi\equiv 0$. In the first case, we can do as in Proposition \ref{prop:trans.delta.psi}: the random walk associated with the conductances $(c^\psi_{i,j})_{i\sim j}$ is transient for $\nu_T^W$-almost all $\beta$, since the associated Green function $g=g^\psi$ is given by
\[g^\psi(i,j)=\frac{\psi(j)}{\psi(i)}2\beta_j\Gh(i,j),\]
for $i,j\in T$. Moreover, we define the Markov operator $\Delta^\psi$ by
\[(\Delta^\psi h)(i)=\l(\sum_{j\sim i} \frac{W_{i,j}\psi(j)}{2\beta_i\psi(i)} h(j)\r) -h(i)=0\] for $h:T\to\R$ and $i\in T$, and a function $h$ is $\Delta^\psi$-harmonic if and only if $\psi h$ is $H_\beta$-harmonic. 

Let us fix $\beta\in\calD_T^W$ such that $\psi(i)>0$ for all $i\in T$, and such that the random walk associated with $(c^\psi_{i,j})_{i\sim j}$ is transient. This allows us to apply results regarding the Martin boundary of a tree. From Theorem \ref{thm:Martin.arbre}, the Martin boundary $\M^\psi$ associated with $\Delta^\psi$ is the set $\Omega$ of ends of $T$. Note that it does not depend on $\beta$. We also get the Martin kernel $K=K^\psi$: for $x\in T$ and $\omega\in\Omega$,
\[K(x,\omega)=K(x,x\wedge\omega)=\frac{f(x,x\wedge\omega)}{f(\phi,x\wedge\omega)}=\frac{\psi(\phi)\Fh(x,x\wedge\omega)}{\psi(x)\Fh(\phi,x\wedge\omega)},\]
where $f(i,j)=\frac{g(i,j)}{g(j,j)}=\frac{\psi(j)}{\psi(i)}\Fh(i,j)$ for $i,j\in T$. Moreover, we denote by $(\mu^\psi_i)_{i\in T}$ the associated family of harmonic measures on $\Omega$. From Proposition \ref{prop:mes.harm.arbre}, we have, for $i,x\in T$,
\[\mu^\psi_i(\Omega_x)=\ind_{\{i\in U_x\}}(1-f(i,x)) + f(i,x)\frac{1-f(x,\p{x})}{1-f(x,\p{x})f(\p{x},x)}.\]

Note that we have only defined $(\mu^\psi_y)_{y\in T}$  for $\beta\in\calD_T^W$ such that $\psi>0$ and the walk in conductances $c^\psi$ is transient. In remaining cases, we adopt the convention that $\mu^\psi_y$ is the null measure on $\Omega$ for all $y\in T$.

\subsection{Convergence of $\chi_m^\n$: Proof of Theorem \ref{thm:chi.mix} (i)}

From Theorem \ref{thm:Gh.psi.mix}, we know that $\nu_T^W$-almost surely, for all $i,j\in T$, $\Gh^\n(i,j)$ converges to $\Gh(i,j)$ and $\psi^\n(i)$ converges to $\psi(i)$. Let $\beta\in\calD_T^W$ be such that these convergences hold. Let us show that for such $\beta$, for all $m\in\N$, $x\in D^\m$ and for all $i\in T$, $\chi_m^\n(i,\delta_x)$ converges to $\psi(i)\mu^\psi_i(\Omega_x)$, and we will have shown that this convergence holds $\nu_T^W$-almost surely.

If $\beta$ is such that $\psi\equiv 0$, we know that for all $i\in T$, $x\in D^\m$, $0\leq\chi_m^\n(i,\delta_x)\leq\psi^\n(i)$ from Remark \ref{rem:decoup.chi}, so $\chi_m^\n(i,\delta_x)\xrightarrow[n\to\infty]{}0=\psi(i)\mu^\psi_i(\Omega_x)$. We now suppose that $\beta$ is such that $\psi(i)>0$ for all $i\in T$.

Let us fix $i\in T$ and $x\in D^\m$. Recall that for $n\geq \max(|i|,m)$,
\begin{align*}
\chi_m^\n(i,\delta_x)&=\sum_{y\in T_x\cap D^\n}\Gh^\n(i,y)\eta^\n_y=\sum_{y\in T_x\cap D^\n}\l(\sum_{\sigma\in\calP_{i,y}^{T^\n}}\frac{W_\sigma}{(2\beta)_\sigma}\r)\eta^\n_y.
\end{align*}
Let us decompose the paths $\sigma\in\calP_{i,y}^{T^\n}$, in order to write $\chi_m^\n(i,\delta_x)$ as a function of $\Fh^\n$ and $\psi^\n$. We will distinguish two cases.

First, if $i\notin U_x=T_x\backslash\{x\}$, then for all $y\in T_x\cap D^\n$, any path from $i$ to $y$ in $T^\n$ necessarily visits $x$, \textit{i.e.} $\calP_{i,y}^{T^\n}=\calP_{i,\{x\},y}^{T^\n}$. Therefore, from Proposition \ref{prop:g.somme} (iii), $\Gh^\n(i,y)=\Fh^\n(i,x)\Gh^\n(x,y)$. In order to decompose $\Gh^\n(x,y)$, let us introduce the quantity $c_x(\sigma)$, defined as the number of times the path $\sigma$ crosses the directed edge $(x,\p{x})$, \textit{i.e.}
\[c_x(\sigma)=\#\{k\in\llb 0,|\sigma|-1\rrb,(\sigma_k,\sigma_{k+1})=(x,\p{x})\}.\]
Then we have
\[\Gh^\n(x,y)=\sum_{C\in\N}\sum_{\substack{\sigma\in\calP_{x,y}^{T^\n}\\c_x(\sigma)=C}}\frac{W_\sigma}{(2\beta)_\sigma}.\]
If $\sigma\in\calP_{x,y}^{T^\n}$ is such that $c_x(\sigma)=C\geq 1$, then it has to visit $\p{x}$ at least once. As a result, $\sigma$ can be written as the concatenation of a path $\sigma_1\in\o\calP^{T^\n\backslash\{\p{x}\}}_{x,\p{x}}$ with a path $\sigma'_1\in\calP^{T^\n}_{\p{x},y}$ such that $c_x(\sigma'_1)=C-1$. Since $\p{x}\notin U_x$, the path $\sigma'_1$ has to visit $x$, so it can be written as the concatenation of a path $\sigma_2\in\o\calP^{T^\n\backslash\{x\}}_{\p{x},x}$ with a path $\sigma_3\in\calP^{T^\n}_{x,y}$ such that $c_x(\sigma_3)=C-1$. Therefore, for all $C\geq 1$,
\begin{align*}
\sum_{\substack{\sigma\in\calP_{x,y}^{T^\n}\\c_x(\sigma)=C}}\frac{W_\sigma}{(2\beta)_\sigma} &= \l(\sum_{\sigma_1\in\o\calP^{T^\n\backslash\{\p{x}\}}_{x,\p{x}}}\frac{W_{\sigma_1}}{(2\beta)_{\sigma_1}}\r)\l(\sum_{\sigma_2\in\o\calP^{T^\n\backslash\{x\}}_{\p{x},x}}\frac{W_{\sigma_2}}{(2\beta)_{\sigma_2}}\r)\l(\sum_{\substack{\sigma_3\in\calP_{x,y}^{T^\n}\\c_x(\sigma_3)=C-1}}\frac{W_{\sigma_3}}{(2\beta)_{\sigma_3}}\r)\\
&=\Fh^\n(x,\p{x})\Fh^\n(\p{x},x)\sum_{\substack{\sigma\in\calP_{x,y}^{T^\n}\\c_x(\sigma)=C-1}}\frac{W_\sigma}{(2\beta)_\sigma}.
\end{align*}
Moreover, note that the paths $\sigma\in\calP^{T^\n}_{x,y}$ such that $c_x(\sigma)=0$ are those that stay in the subtree $T^\n_x$, \textit{i.e.} the set $\calP^{T^\n_x}_{x,y}$. By induction, we get:
\[\Gh^\n(x,y)=\sum_{C\in\N}\l(\Fh^\n(x,\p{x})\Fh^\n(\p{x},x)\r)^C \sum_{\sigma\in\calP^{T^\n_x}_{x,y}}\frac{W_\sigma}{(2\beta)_\sigma}.\]
Since $\Gh^\n(x,y)<\infty$, we have $\Fh^\n(x,\p{x})\Fh^\n(\p{x},x)<1$, which gives
\begin{align*}
\chi_m^\n(i,\delta_x)&= \sum_{y\in T_x\cap D^\n}\Fh^\n(i,x)\Gh^\n(x,y)\eta^\n_y\\
&=\Fh^\n(i,x)\frac{1}{1-\Fh^\n(x,\p{x})\Fh^\n(\p{x},x)}\sum_{y\in T_x\cap D^\n}\l(\sum_{\sigma\in\calP^{T^\n_x}_{x,y}}\frac{W_\sigma}{(2\beta)_\sigma}\r)\eta^\n_y.
\end{align*}
In order to express this last sum, recall that
\begin{align*}
\psi^\n(x)&=\sum_{y\in D^\n}\Gh^\n(x,y)\eta^\n_y\\
&= \sum_{y\in D^\n}\l(\sum_{\sigma\in\calP^{T^\n}_{x,\{\p{x}\},y}}\frac{W_\sigma}{(2\beta)_\sigma} + \sum_{\sigma\in\calP^{T^\n_x}_{x,y}}\frac{W_\sigma}{(2\beta)_\sigma}\r)\eta^\n_y,
\end{align*}
where we have separated the paths that go from $x$ to $y$ by visiting $\p{x}$, and those that stay in $T_x^\n$, since $\calP_{x,y}^{T^\n}=\calP_{x,\{\p{x}\},y}^{T^\n}\cup\calP_{x,y}^{T_x^\n}$. From Proposition \ref{prop:g.somme} (iii), we have
\[\sum_{\sigma\in\calP^{T^\n}_{x,\{\p{x}\},y}}\frac{W_\sigma}{(2\beta)_\sigma}=\Fh^\n(x,\p{x})\Gh^\n(\p{x},y).\]
Moreover, if $y\in D^\n\backslash T_x$, then $\calP_{x,y}^{T_x^\n}$ is empty. As a result, we get
\[\sum_{y\in T_x\cap D^\n}\l(\sum_{\sigma\in\calP^{T^\n_x}_{x,y}}\frac{W_\sigma}{(2\beta)_\sigma}\r)\eta^\n_y=\psi^\n(x)-\Fh^\n(x,\p{x})\psi^\n(\p{x}),\]
which finally gives
\[\chi_m^\n(i,\delta_x)=\Fh^\n(i,x)\frac{\psi^\n(x)-\Fh^\n(x,\p{x})\psi^\n(\p{x})}{1-\Fh^\n(x,\p{x})\Fh^\n(\p{x},x)}.\]

In the second case, \textit{i.e.} if $i\in U_x$, then for $y\in T_x\cap D^\n$, there are paths from $i$ to $y$ in $T^\n$ that do not visit $x$. More precisely, we have the following partition: $\calP_{i,y}^{T^\n}=\calP_{i,\{x\},y}^{T^\n}\cup\calP_{i,y}^{U_x^\n}$, where $U_x^\n=U_x\cap T^\n$. As a result, we have
\[\chi_m^\n(i,\delta_x)=\Fh^\n(i,x)\frac{\psi^\n(x)-\Fh^\n(x,\p{x})\psi^\n(\p{x})}{1-\Fh^\n(x,\p{x})\Fh^\n(\p{x},x)}+\sum_{y\in T_x\cap D^\n}\l(\sum_{\sigma\in\cup\calP_{i,y}^{U_x^\n}}\frac{W_\sigma}{(2\beta)_\sigma}\r)\eta^\n_y.\]
In the same way we did above, we can show that
\[\sum_{y\in T_x\cap D^\n}\l(\sum_{\sigma\in\cup\calP_{i,y}^{U_x^\n}}\frac{W_\sigma}{(2\beta)_\sigma}\r)\eta^\n_y=\psi^\n(i)-\Fh^\n(i,x)\psi^\n(x).\]

In conclusion, we have established the following:
\begin{align*}
\chi_m^\n(i,\delta_x) &=\ind_{\{i\in U_x\}}(\psi^\n(i)-\Fh^\n(i,x)\psi^\n(x))+\Fh^\n(i,x)\frac{\psi^\n(x)-\Fh^\n(x,\p{x})\psi^\n(\p{x})}{1-\Fh^\n(x,\p{x})\Fh^\n(\p{x},x)}\\
&=\psi^\n(i)\l(\ind_{\{i\in U_x\}}(1-f^\n(i,x))+f^\n(i,x)\frac{1-f^\n(x,\p{x})}{1-f^\n(x,\p{x})f^\n(\p{x},x)}\r),
\end{align*}
where
\[f^\n(i,j)=\frac{\psi^\n(j)}{\psi^\n(i)}\Fh^\n(i,j)\xrightarrow[n\to\infty]{}\frac{\psi(j)}{\psi(i)}\Fh(i,j)=f(i,j)\]
for all $i,j\in T$. As a result, we finally have
\[\chi_m^\n(i,\delta_x)\xrightarrow[n\to\infty]{}\psi(i)\l(\ind_{\{i\in U_x\}}(1-f(i,x))+f(i,x)\frac{1-f(x,\p{x})}{1-f(x,\p{x})f(\p{x},x)}\r)=\psi(i)\mu^\psi_i(\Omega_x).\]

We can now define, for all $i\in T$, the measure $\chi(i,\cdot)=\psi(i)\mu^\psi_i$. Note that $\chi(i,\cdot)$ is absolutely continuous with respect to $\chi(\phi,\cdot)$, and its Radon-Nikodym derivative is $\omega\mapsto\frac{\Fh(i,i\wedge\omega)}{\Fh(\phi,i\wedge\omega)}$. Moreover, for all $A\in\mathcal{B}(\Omega)$, $i\mapsto\mu_i^\psi(A)$ is $\Delta^\psi$-harmonic, so $\chi(\cdot,A):i\mapsto\psi(i)\mu_i^\psi(A)$ is $H_\beta$-harmonic.

\subsection{Convergence of $(G_m^\n)_{B_m,B_m}$}

Recall that for $n\geq m$, we have defined $(\beta_m^\n)_{B_m}$ only by its distribution conditionally on $\beta$, which is $\nu_{B_m}^{\Wc_m^\n}$. Let us show that there is a coupling of these distributions such that the matrix $(G_m^\n)_{B_m,B_m}$ converges $\nu_T^W$-almost surely.

We can write $(G_m^\n)_{B_m,B_m}$ as the inverse of a Schur complement. Indeed,
\begin{align*}
(G_m^\n)_{B_m,B_m}&=(H_\beta^\n)^{-1}_{B_m,B_m}=\l((H_\beta^\n)_{B_m,B_m}-(\Wt_m^\n)_{B_m,T^\n}((H_\beta^\n)_{T^\n,T^\n})^{-1}(\Wt_m^\n)_{T^\n,B_m}\r)^{-1}\\
&= \l((2\beta_m^\n)_{B_m}-\Wc_m^\n\r)^{-1}=(\Hc_\beta^\n)^{-1},
\end{align*}
where $\Hc_\beta^\n=2(\beta_m^\n)_{B_m}-\Wc_m^\n$. We apply the following change of variables: for $\beta\in\calD_T^W$ and $b\in B_m$, let us define $(\rho_m^\n)_b=(\beta_m^\n)_b-\frac{1}{2}(\Wc_m^\n)_{b,b}$. Then $\Hc_\beta^\n=2\rho_m^\n-\Cc_m^\n$, where if $b,b'\in B$,
\[(\Cc_m^\n)_{b,b'}=\begin{cases}(\Wc_m^\n)_{b,b'} &\mbox{if } b\neq b' \\0 &\mbox{if } b=b'. \end{cases}\]
The vector $\rho_m^\n$ is then distributed according to $\nu_{B_m}^{\Cc_m^\n}$ conditionally on $\beta_T$. Let us show that the matrix $\Cc_m^\n$ converges $\nu_T^W$-almost surely, to prove that $\rho_m^\n$ converges in distribution.

Let us fix $\beta\in\calD_T^W$, as well as $x\neq y\in D^{(m)}$, and $i\sim\delta_x$, $j\sim\delta_y$. A path from $i$ to $j$ in $T^\n$ necessarily crosses $x\wedge y$, since $i\in T_x$ and $j\in T_y$. Therefore, $\calP^{T^\n}_{i,j}=\calP^{T^\n}_{i,\{x\wedge y\},j}$, so
\begin{align*}
(\Cc_m^\n)_{\delta_x,\delta_y} &=\sum_{i\sim\delta_x,j\sim\delta_y} (\Wt_m^\n)_{\delta_x,i} \Gh^\n(i,j) (\Wt_m^\n)_{j,\delta_y}\\
&= \sum_{i\sim\delta_x,j\sim\delta_y} (\Wt_m^\n)_{\delta_x,i}\l(\sum_{\sigma\in\calP^{T^\n}_{i,\{x\wedge y\},j}}\frac{W_\sigma}{(2\beta)_\sigma}\r)(\Wt_m^\n)_{j,\delta_y} \\
&= \sum_{i\sim\delta_x,j\sim\delta_y} (\Wt_m^\n)_{\delta_x,i}\Fh^\n(i,x\wedge y) (\Wt_m^\n)_{\delta_y,j}\Fh^\n(j,x\wedge y) \Gh^\n(x\wedge y,x\wedge y)\\
&= \frac{1}{\Gh^\n(x\wedge y,x\wedge y)} \l(\sum_{i\sim\delta_x} (\Wt_m^\n)_{\delta_x,i}\Gh^\n(i,x\wedge y)\r) \l(\sum_{j\sim\delta_y} (\Wt_m^\n)_{\delta_y,j}\Gh^\n(j,x\wedge y)\r) \\
&= \frac{\chi^\n(x\wedge y,\delta_x)\chi^\n(x\wedge y,\delta_y)}{\Gh^\n(x\wedge y,x\wedge y)},
\end{align*}
and $(\Cc_m^\n)_{\delta_x,\delta_x}=0$. Since $\chi^\n$ converges $\nu_T^W$-almost surely, the matrix $\Cc_m^\n$ also converges to a matrix $\Cc_m$, where for $x,x'\in D^\m$,
\[(\Cc_m)_{\delta_x,\delta_{x'}}=\begin{cases}\frac{\chi(x\wedge y,\Omega_x)\chi(x\wedge y,\Omega_{x'})}{\Gh(x\wedge y,x\wedge y)} &\mbox{if } x\neq x' \\0 &\mbox{if } x=x'. \end{cases}\]
Conditionally on $\beta_T$, $\rho_m^\n$ is distributed according to $\nu_{B_m}^{\Cc_m^\n}$, which almost surely converges weakly to $\nu_{B_m}^{\Cc_m}$  by Lévy's theorem. Since the random variables $(\beta_m^\n)_{B_m}$ (and consequently $\rho_m^\n$) have only been defined through their distribution conditionally on $\beta_T$ (see Remark \ref{rem:aj.couplage}), we can define a sequence $(\rho_m^\n)_{n\geq m}$ which converges almost surely to a potential $\rho_m$ on $B_m$, whose distribution conditionally on $\beta_T$ is $\nu_{B_m}^{\Cc_m}$. The matrix $\Hc_\beta^\n=2\rho_m^\n-\Cc_m^\n$ then converges almost surely to $\Hc_\beta=2\rho_m-\Cc_m$, which is inversible given the distribution $\nu_{B_m}^{\Cc_m}$. Therefore, $(G_m^\n)_{B_m,B_m}$ converges $\nu_T^W$-almost surely to $\Gc_m:=(\Hc_\beta)^{-1}$.

\section{Representations of the VRJP on infinite trees}

\subsection{Representation of the VRJP using $G_m$: Proof of Theorem \ref{thm:chi.mix} (ii)}

Recall that for $0\leq m\leq n$ and $i,j\in T^\n$,
\[G_m^\n(i,j)= \Gh^\n(i,j)+\sum_{b,b'\in B_m}\chi_m^\n(i,b)G_m^\n(b,b')\chi_m^\n(j,b'),\]
where under $\nu_T^W$, $G_m^\n=(2\beta_m^\n-\Wt_m^\n)^{-1}$ gives a representation on the VRJP on $\G_m^\n$, from Theorem \ref{thm:beta.mix}. We have shown that when $n\to\infty$, $\chi_m^\n$ converges almost surely, as well as $(\Gc_m^\n)_{B_m,B_m}$. As a result for all $i,j\in T$, $G_m^\n(i,j)$ converges almost surely to a limit $G_m(i,j)$, where
\[G_m(i,j)=\Gh(i,j)+\sum_{x,x'\in D^\m}\chi(i,\Omega_x)\chi(j,\Omega_y)\Gc_m(\delta_x,\delta_y).\]

The second term in $G_m(i,j)$ can be rewritten as an integral on $\Omega^2$. Indeed, let us define $\gc_m:\Omega^2\to\R$ in the following way: for $\omega,\tau\in\Omega$, if $x,y\in D^{(m)}$ are such that $\omega\in\Omega_x$ and $\tau\in\Omega_y$, we set $\gc_m(\omega,\tau)=\Gc_m(\delta_x,\delta_y)$. With these new notations, we can now write, for $i,j\in T$,
\begin{align*}
G_m(i,j)&=\Gh(i,j)+\chi_m(i,\cdot)\Gc_m {^t}\l(\chi_m(j,\cdot)\r)\\
&=\Gh(i,j)+\int_{\Omega^2} \chi(i,d\omega)\chi(j,d\tau)\gc_m(\omega,\tau).
\end{align*}

For $\beta\in\calD_T^W$, $\rho_m\in\calD_{B_m}^{\Cc_m}$, and all $i_0,i,j\in T$, we denote $r^{\beta,\rho_m,i_0}_{i,j}=\frac{W_{i,j}}{2}\frac{G_m(i_0,j)}{G_m(i_0,i)}$. To prove Theorem \ref{thm:chi.mix} (ii), we have to see that that $P^{VRJP(i_0)}$ is a mixture of Markov processes $P^{\beta,\rho_m,i_0}_{i_0}$ under $\nu_{T,B_m}^{W}(d\beta,d\rho_m)$. The proof is the same than that of Theorem \ref{thm:Gh.psi.mix} (iii) (see \cite{SabZen}). It consists in studying trajectories of the time-changed VRJP, stopped when they leave a finite subgraph included in $T^\n$. They can be considered as trajectories of the time-changed VRJP on $\G_m^\n$, and represented using $G_m^\n$ thanks to Theorem \ref{thm:beta.mix}. Taking $n\to\infty$ then gives the result. Note that the proof needs an argument of uniform integrability on the family $\l(\frac{G_m^\n(i_0,j)}{G_m^\n(i_0,i)}\r)_{n\geq m}$ for all $i,j\in T$, which is given by Proposition 7 and Corollary 2 from \cite{SabZen}.

\subsection{Convergence to another representation: Proof of Theorem \ref{thm:chi.mix} (iii)}

Let us show that the representations of the VRJP built with $G_m$ converge in distribution when $m\to\infty$ to the representation described in Theorem \ref{thm:mix.tree}. To show this, we use a tightness argument, based on the following lemma regarding the distribution $\nu_V^W$.

\begin{lem}\label{lem:dist.xGx}
Let $\G=(V,E)$ be a finite graph, endowed with conductances $W$. We denote $G=(H_\beta)^{-1}$ for $\beta\in\calD_V^W$. Then for all $\eta\in\R_+^V$, under $\nu_V^W(d\beta)$, $\langle\eta,G\eta\rangle$ has the same distribution as $\frac{\langle \eta,\ind\rangle^2}{2\gamma}$, where $\gamma$ is a Gamma random variable with parameter $(1/2,1)$.
\end{lem}

\begin{proof}
Let $\eta\in\R_+^V$ be fixed. We will compute the Laplace transform of $\langle\eta,G\eta\rangle$: for $\lambda\in\R_+$,
\begin{align*}
\E\l[e^{-\lambda \langle\eta,G\eta\rangle}\r]&=\int e^{-\lambda\langle\eta,(H_\beta)^{-1}\eta\rangle}\nu_V^W(d\beta)\\
&=e^{-\sqrt{2\lambda}\langle\eta,\ind\rangle}\int e^{-\frac{1}{2}\langle\sqrt{2\lambda}\eta,(H_\beta)^{-1}\sqrt{2\lambda}\eta\rangle}e^{\langle\sqrt{2\lambda}\eta,\ind\rangle}\nu_V^W(d\beta)\\
&=e^{-\sqrt{2\lambda}\langle\eta,\ind\rangle}\int\nu_V^{W,\sqrt{2\lambda}\eta}(d\beta)=e^{-\sqrt{2\lambda}\langle\eta,\ind\rangle},
\end{align*}
since $\nu_V^{W,\sqrt{2\lambda}\eta}$ is a probability measure. Let us now compute, for $\gamma\sim\Gamma(1/2,1)$, the Laplace transform of $\frac{1}{2\gamma}$: for $\lambda\in\R_+$,
\begin{align*}
\E\l[e^{-\frac{\lambda}{2\gamma}}\r]&=\int\ind_{u>0}\frac{e^{-u}}{\sqrt{\pi u}}e^{-\frac{\lambda}{2u}}du=\int_{v>0}\frac{1}{\sqrt{2\pi v^3}}e^{-\frac{1}{2}(\frac{1}{v}+2\lambda v)}dv=e^{-\sqrt{2\lambda}}\int_{v>0}\frac{1}{\sqrt{2\pi v^3}}e^{-\frac{2\lambda}{2v}(v-\frac{1}{\sqrt{2\lambda}})^2}dv,
\end{align*}
by taking $v=1/2u$. Since $\ind_{v>0}\frac{1}{\sqrt{2\pi v^3}}e^{-\frac{2\lambda}{2v}(v-\frac{1}{\sqrt{2\lambda}})^2}dv$ is the density of an Inverse Gaussian distribution with parameter $(1,1/\sqrt{2\lambda})$, we finally get $\E[e^{-\frac{\lambda}{2\gamma}}]=e^{-\sqrt{2\lambda}}$. Therefore, for all $\lambda\geq 0$,
\[\E\l[e^{-\lambda\langle\eta,G\eta\rangle}\r]=e^{-\sqrt{2\lambda\langle\eta,\ind\rangle^2}}=\E\l[e^{-\frac{\lambda\langle\eta,\ind\rangle^2}{2\gamma}}\r],\]
which proves the result.
\end{proof}

For $m\geq 0$ and $\beta\in\calD_T^W$, let us define, for $i\in T$ and $m\geq 0$, the vector $\bar\mu_i^\m\in\R^{B_m}$ by $(\bar\mu_i^\m)(\delta_x)=\mu_i^\psi(\Omega_x)$ for all $x\in D^\m$. Then, for $\rho_m\in\calD_{B_m}^{\Cc_m}$ and $i,j\in T$,
\begin{align*}
\int_{\Omega^2}\chi(i,d\omega)\chi(j,d\tau)\gc_m(\omega,\tau) &=\psi(i)\psi(j)\sum_{x,x'\in D^\m}\int_{\Omega_x\times\Omega_{x'}}\mu_i(d\omega)\mu_j(d\tau)\Gc_m(\delta_x,\delta_{x'})\\
&=\psi(i)\psi(j)\langle\bar\mu_i^\m,\Gc_m\bar\mu_j^\m\rangle.
\end{align*}

We denote, for $m\geq 0$ and $i,j\in T$,
\[a^\m_{i,j}=\frac{1}{4}\langle\bar\mu_i^\m+\bar\mu_j^\m,\Gc_m(\bar\mu_i^\m+\bar\mu_j^\m)\rangle.\]
so that we have
\begin{align*}
G_m(i,j) &= \Gh(i,j)+\psi(i)\psi(j)\langle\bar\mu_i^\m,\Gc_m\bar\mu_j^\m\rangle\\
&= \Gh(i,j)+\frac{\psi(i)\psi(j)}{2}\l(4a^\m_{i,j}-a^\m_i-a^\m_j\r).
\end{align*}
Therefore, we can write $(G_m(i,j))_{i,j\in T}=\Phi\l((\Gh(i,j))_{i,j\in T},(\psi(i))_{i\in T},(a^\m_{i,j})_{i,j\in T}\r)$, where $\Phi$ is a continuous function.

We will denote by $\o\nu_T^W(d\beta,d\rho)$ the distribution of a coupling of distributions $\nu_{T,B_m}^{W}(d\beta,d\rho_m)$ for all $m\geq 0$.

\begin{lem}
Let us set $Z_m=\l((\Gh(i,j))_{i,j\in T},(\psi(i))_{i\in T},(a^\m_{i,j})_{i,j\in T}\r)$ for $m\geq 0$, which takes its values in $\R^{T^2}\times\R^T\times\R^{T^2}$. Then under $\o\nu_T^W(d\beta,d\rho)$, $(Z_m)_{m\geq 0}$ is tight.
\end{lem}

\begin{proof}
For $\epsilon>0$, let $K_\epsilon$ be a compact subset of $\R$ such that $0\in K_\epsilon$ and
\[\P\l[\frac{1}{2\gamma}\in K_\epsilon\r]>1-\epsilon\]
when $\gamma\sim\Gamma(1/2,1)$. Let us now fix $m\geq 0$. Recall that $\Gc_m=(2\rho_m-\Cc_m)^{-1}$, where $\Cc_m$ is a $\beta$-measurable matrix of conductances on $B_m$, and conditionally on $\beta$, $\rho_m$ is distributed according to $\nu_{B_m}^{\Cc_m}$. Therefore for $i,j\in T$, from Lemma \ref{lem:dist.xGx}, ${\langle\bar\mu_i^\m+\bar\mu_j^\m,\Gc_m\l(\mu_i^\m+\bar\mu_j^\m\r)\rangle}$ has conditionally on $\beta$ the same distribution as $\frac{(\mu_i^\psi(\Omega)+\mu_j^\psi(\Omega))^2}{2\gamma}$, where $\gamma\sim\Gamma(1,1/2)$. This implies that conditionally on $\{\psi\not\equiv 0\}$, $a_{i,j}^\m$ has the same distribution as $\frac{1}{2\gamma}$, and conditionnally on $\{\psi\equiv 0\}$, $a_{i,j}^\m=0$. As a result, for all $\epsilon>0$,
\begin{align*}
\P\l[a_{x,x'}^\m\in K_\epsilon\r]&=\P\l[\psi\equiv 0\r] +\P\l[\psi\not\equiv 0\r]\P\l[\frac{1}{2\gamma}\in K_\epsilon\r]>1-\epsilon.
\end{align*}

Let now $\l(\tilde{a}_k^\m\r)_{k\in\N}$ be an enumeration of $\l(a_{i,j}^\m\r)_{i,j\in T}$. Then for $\epsilon>0$,
\[\P\l[\forall k\in\N, \tilde{a}_k^\m\in K_{2^{-n-1}\epsilon}\r]\geq 1-\sum_{k\in\N}2^{-n-1}\epsilon=1-\epsilon,\]
where $\tilde{K}_\epsilon=\prod_{k\in\N}K_{2^{-n-1}\epsilon}$ is a compact subset of $\R^\N$. Moreover, the $\beta$-measurable random variable $\l((\Gh(i,j))_{i,j\in T},(\psi(i))_{i\in T}\r)$ takes its values in $\R^{T^2}\times\R^T$, where $T$ is countable. As a result, for all $\epsilon>0$, there is a compact subset $K'_\epsilon\subset\R^{T^2}\times\R^T$ such that
\[\P\l[\l((\Gh(i,j))_{i,j\in T},(\psi(i))_{i\in T}\r)\in K'_\epsilon\r]>1-\epsilon.\]
We can now conclude that for all $\epsilon>0$,
\[\P\l[\l((\Gh(i,j))_{i,j\in T},(\psi(i))_{i\in T},(a^\m_{i,j})_{i,j\in T}\r)\in \tilde{K}_{\epsilon/2}\times K'_{\epsilon/2}\r]>1-\epsilon,\]
where $\tilde{K}_{\epsilon/2}\times K'_{\epsilon/2}$ is compact, and does not depend on $m$.
\end{proof}

As a result, there is an extraction $(m_k)_{k\in\N}$ such that $(Z_{m_k})_{k\in\N}$ converges in distribution under $\o\nu_T^W(d\beta,d\rho)$. Since $G_{m_k}=\Phi(Z_{m_k})$ where $\Phi$ is continuous, $(G_{m_k}(i,j))_{i,j\in T}$ also converges in distribution under under $\o\nu_T^W(d\beta,d\rho)$, as do the random jump rates $(r_{i,j}^{(m_k),\phi})_{i,j\in T}$. Let us show that the limit distribution of the environment does not depend on the extraction, which will mean that $((r_{i,j}^{(m),\phi})_{i,j\in T})_{m\in\N}$ converges in distribution, since it is tight.

\begin{lem}\label{lem:dist.rates.tree}
For $m\geq 1$ and for all $n\geq m$, under the distribution $\o\nu_T^W(d\beta,d\rho)$, the random variables $\l(\frac{G_m^\n(\phi,i)}{G_m^\n(\phi,\p{i})}\r)_{i\in T^\m\backslash\{\phi\}}$ are independent inverse Gaussian variables, where $\frac{G_m^\n(\phi,i)}{G_m^\n(\phi,\p{i})}$ has parameter $(W_{\p{i},i},1)$ for $i\in T^\m\backslash\{\phi\}$.
\end{lem}

\begin{proof}
Let us fix $1\leq m\leq n$. For $i\in T^\m\backslash\{\phi\}$, we denote $g_i=\frac{G_m^\n(\phi,i)}{G_m^\n(\phi,\p{i})}$. Since $|i|\leq m$, any path in $\G_m^\n$ from $\phi$ to $i$ crosses $\p{i}$, so from Proposition \ref{prop:g.somme} (ii) and (iii),
\[g_i=\frac{G_m^\n(\p{i},i)}{G_m^\n(\p{i},\p{i})}=\sum_{j\sim\p{i}}W_{\p{i},j}\sum_{\sigma\in\calP_{j,i}^{\Tt_m^\n\backslash\{\p{i}\}}}\frac{(\Wt_m^\n)_\sigma}{(2\beta_m^\n)_\sigma}.\]
For $i\in T^\m\backslash\{\phi\}$, let us denote by $\Tt_i$ the connected component of $i$ in $\Tt_m^\n\backslash\{\p{i}\}$, \textit{i.e.} $\Tt_i=T_m^\n\cup\{\delta_x,x\in D^\m\cap T_i\}$, endowed with the restriction of $\Wt_m^\n$. This way, we get
\[g_i=W_{\p{i},i}\sum_{\sigma\in\calP_{i,i}^{\Tt_i}}\frac{(\Wt_m^\n)_\sigma}{(2\beta_m^\n)_\sigma}=W_{\p{i},i}\l((H_\beta^\n)_{\Tt_i,\Tt_i}\r)^{-1}(i,i),\]
so $g_i$ is $(\beta_m^\n)_{\Tt_i}$-measurable.

To prove that $(g_i)_{i\in T^\m\backslash\{\phi\}}$ are independent, it will be enough to see that for $i\in T^\m\backslash\{\phi\}$, $g_i$ is independent of $g_{U_i^\m}$, and that for $x\in T^{(m-1)}$, the restrictions $(g_{\Tt_i})_{i\in S(x)}$ are independent.

Writing $\l((H_\beta^\n)_{\Tt_i,\Tt_i}\r)^{-1}(i,i)$ as a Schur complement, we see that, if we set $\tilde{U}_i=\Tt_i\backslash\{i\}$,
\[g_i=\frac{W_{\p{i},i}}{2(\beta_m^\n)_i-\sum_{j,j'\in S(i)}W_{i,j}W_{i,j'}\l((H_\beta^\n)_{\tilde{U}_i,\tilde{U}_i}\r)^{-1}(j,j')}=\frac{W_{\p{i},i}}{\Hc_\beta^{\{i\}}}.\]
From Proposition \ref{prop:restr.beta}, conditionally on $(\beta_m^\n)_{\tilde{U}_i}$, the distribution of $(\beta_m^\n)_i$ is given by
\[\sqrt{\frac{2}{\pi}}\frac{\ind_{\Hc_\beta^{\{i\}}>0}}{\sqrt{\det(\Hc_\beta^{\{i\}})}}e^{-\frac{1}{2}\l(\Hc_\beta^{\{i\}}+W_{\p{i},i}^2(\Hc_\beta^{\{i\}})^{-1}\r)}e^{W_{\p{i},i}}d(\beta_m^\n)_i,\]
so by a change of variables, the distribution of $g_i$ conditionnally on $(\beta_m^\n)_{\tilde{U}_i}$ is
\[\ind_{g_i>0}\sqrt{\frac{W_{\p{i},i}}{2\pi g_i^3}}e^{-\frac{W_{\p{i},i}}{2g_i}(g_i-1)^2}dg_i,\]
\textit{i.e.} $g_i\sim IG(W_{\p{i},i},1)$. Since this distribution does not depend on $(\beta_m^\n)_{\tilde{U}_i}$, $g_i$ is independent of $(\beta_m^\n)_{\tilde{U}_i}$. For all $j\in U_i^\m$, $\Tt_j\subset \tilde{U}_i$, so $g_j$ is $(\beta_m^\n)_{\tilde{U}_i}$-measurable. Therefore, $g_i$ is independent of $g_{U_i^\m}$.

Moreover, for $x\in T^{(m-1)}$, the sets $(\Tt_i)_{i\in S(x)}$ are all at distance $2$ from one another in $\G_m^\n$. Since $\beta^\n$ is $1$-dependent, the restrictions $(\beta^\n_{\Tt_i})_{i\in S(x)}$ are independent. For $j\in \Tt_i$, we have $\Tt_j\subset\Tt_i$, so $g_j$ is $\beta^\n_{\Tt_i}$-measurable.
Therefore the restrictions $(g_{\Tt_i})_{i\in S(x)}$ are independent, which concludes the proof.
\end{proof}

For $m\geq 1$, since $G_m^\n$ converges a.s. to $G_m$, it also converges in distribution. If we denote $g^\m_i=\frac{G_m(\phi,i)}{G_m(\phi,\p{i})}$ for $i\in T^\m\backslash\{\phi\}$, and take the limit in Lemma \ref{lem:dist.rates.tree}, we get that $(g^\m_i)_{i\in T^\m\backslash\{\phi\}}$ are independent, and that $g^\m_i\sim IG(W_{\p{i},i},1)$ for $i\in T^\m\backslash\{\phi\}$. Recall that the random environment associated with $G_m$ is given by the following jump rates:
\[r^{\beta,\rho_m,\phi}_{\p{i},i}=\frac{W_{\p{i},i}}{2}\frac{G_m(\phi,i)}{G_m(\phi,\p{i})} \mbox{ and } r^{\beta,\rho_m,\phi}_{i,\p{i}}=\frac{W_{\p{i},i}}{2}\frac{G_m(\phi,\p{i})}{G_m(\phi,i)},\]
for all $i\in T\backslash\{\phi\}$, and $r^{\beta,\rho_m,\phi}_{i,j}=0$ if $i\not\sim j$.

Let now $(m_k)_{k\in\N}$ be an extraction such that under $\o\nu_T^W(d\beta,d\rho)$, $(r^{\beta,\rho_{m_k},\phi}_{i,j})_{i,j\in T}$ converges in distribution to a limit environment $(r^{(\infty),\phi}_{i,j})_{i,j\in T}$. Then, we have $r^{(\infty),\phi}_{i,j}=0$ for $i\not\sim j$. Moreover, let us set $g^{(\infty)}_i=\frac{2}{W_{\p{i},i}}r^{(\infty),\phi}_{\p{i},i}$ for $i\in T\backslash\{\phi\}$. Note that for all $m\in\N$, if $k$ is such that $m_k\geq m$, we have $T^\m\subset T^{(m_k)}$ so for all $i\in T^\m\backslash\{\phi\}$,
\[r^{\beta,\rho_{m_k},\phi}_{\p{i},i}=\frac{W_{\p{i},i}}{2}g^{(m_k)}_i \mbox{ and } r^{\beta,\rho_{m_k},\phi}_{i,\p{i}}=\frac{W_{\p{i},i}}{2g^{(m_k)}_i}.\]
Taking the limit when $k\to\infty$, we get that $g^{(m_k)}_i$ converges in distribution to $g^{(\infty)}_i$ for all $i\in T^\m\backslash\{\phi\}$, which implies that  $(g^{(\infty)}_i)_{i\in T^\m\backslash\{\phi\}}$ are independent. Since this is true for all $m\geq 0$, $(g^{(\infty)}_i)_{i\in T\backslash\{\phi\}}$ are independent. Moreover, for all $i\in T\backslash\{\phi\}$, $g^{(\infty)}_i\sim IG(W_{\p{i},i},1)$ and 
\[r^{(\infty),\phi}_{\p{i},i}=\frac{W_{\p{i},i}}{2}g^{(\infty)}_i \mbox{ and } r^{(\infty),\phi}_{i,\p{i}}=\frac{W_{\p{i},i}}{2g^{(\infty)}_i}.\]
The random environment given by these jump rates is in fact the one described in Theorem \ref{thm:mix.tree}, hence its distribution does not depend on the extraction $(m_k)_{k\in\N}$. Since the sequence of jump rates $((r^{\beta,\rho_m,\phi}_{i,j})_{i,j\in T})_{m\geq 1}$ is tight, this implies that it converges in distribution to the random environment given in Theorem \ref{thm:mix.tree}.

\subsection{Distinct representations on a regular tree: Proofs of Propositions \ref{prop:2.rep.reg.tree} and \ref{prop:diff.rep.reg.tree}}

Let us start by proving that on regular trees where the VRJP is transient, the standard representation and the one given in Theorem \ref{thm:mix.tree} are different.

\begin{proof}[Proof of Proposition \ref{prop:2.rep.reg.tree}]
Let $\calT=(T,E)$ be a $d$-regular tree, where $d\geq 3$. It was shown in \cite{DavVol2} that there exists a $\o{W}>0$ such that for $W>\o{W}$, the VRJP on $\calT$ endowed with constant conductances $W$ is almost surely transient. Note that the VRJP is defined in a slightly different manner in \cite{DavVol2}, but it can be related to the definition used here, thanks to a time rescaling described in Appendix B of \cite{SabTarZen}. From now on, we take $W>\o{W}$.

We consider jump rates $(r_{i,j})_{i\sim j}$ on the tree $\calT$. Let $\phi$ be an arbitrary root for $\calT$, and let $(i_k)_{k\geq 0}$ be an infinite self-avoiding path (or ray) in $\calT$, such that for $k\geq 0$, $|i_k|=k$. Let us define $S_n=\prod_{k=1}^n\frac{2}{W}r_{i_{k-1},i_k}$. We will compare the distribution of $S_n$ under two distribution of jump rates.

Let $\mathcal{R}_{ind}(dr)$ be the distribution of jump rates in the representation described in Theorem \ref{thm:mix.tree}. Under $\mathcal{R}_{ind}(dr)$, we know that $S_n$ has the distribution of $\prod_{i=1}^n A_{i_k}$, where $A_{i_k}$ are independent inverse Gaussian variables with parameter $(W,1)$. Note that $\E[A_1]=1$, so by Jensen's inequality, $\E[\log(A_1)]<0$. By the law of large numbers, we then have a.s. that $\sum_{k=1}^n\log(A_i)\xrightarrow[n\to\infty]{}-\infty$,
so that $S_n\xrightarrow[n\to\infty]{a.s.}0$.

Let now $\mathcal{R}_{st}(dr)$ be the distribution of jump rates in the standard representation of the VRJP started at $\phi=i_0$. Under $\mathcal{R}_{st}(dr)$, Theorem \ref{thm:Gh.psi.mix} tells us that $S_n$ has the same distribution as 
\[\prod_{k=1}^n \frac{G(i_0,i_k)}{G(i_0,i_{k-1})}=\frac{G(i_0,i_n)}{G(i_0,i_0)}=\frac{\Gh(i_0,i_n)+\frac{1}{2\gamma}\psi(i_0)\psi(i_n)}{G(i_0,i_0)}\]
under $\nu_V^W(d\beta,d\gamma)$, where according to Proposition \ref{prop:trans.delta.psi}, $\psi(i)>0$ a.s. for all $i\in T$. Moreover, since the distribution of $\psi$ under $\nu_V^W(d\beta)$ is stationary for the group of transformations of $\calT$ (see Proposition 3 in \cite{SabZen}), $\psi(i_n)$ has the same distribution as $\psi(i_0)$ for all $n\in\N$, and cannot tend to $0$ a.s. when $n\to\infty$. Therefore, neither can $S_n$ under $\mathcal{R}_{st}(dr)$, which proves that $\mathcal{R}_{st}$ and $\mathcal{R}_{ind}$ are different.
\end{proof}

We will now prove Proposition \ref{prop:diff.rep.reg.tree}, \textit{i.e.} that on a $d$-regular tree with $d\geq 3$, for constant $W$ large enough so that the VRJP is transient, the representations in the family given by Theorem \ref{thm:chi.mix} are all different. In order to do this, we will compare the distribution of the random harmonic measures for each representation.

The following Proposition gives an expression for the measure of sets $\Omega_x$, for $x\in T$. We will see how this expression behaves differently whether or not $|x|>m$.

\begin{prop}\label{prop:mes.harm.mix}
Let $m\in\N$ be fixed. We also fix $\beta\in\calD_T^W$ and $\rho_m\in\calD_{B_m}^{\Cc_m}$. We denote by $\mu_\phi^{\beta,\rho_m,\phi}$ the exiting measure of the transient Markov process $P^{\beta,\rho_m,\phi}_\phi$ defined in Theorem \ref{thm:chi.mix}. Then for $x\in T$,
\[\mu_\phi^{\beta,\rho_m,\phi}(\Omega_x)=\frac{\int_{\Omega\times\Omega_x}\chi(\phi,d\omega)\gc_m(\omega,\tau)\chi(\phi,d\tau)}
{\int_{\Omega^2}\chi(\phi,d\omega)\gc_m(\omega,\tau)\chi(\phi,d\tau)}.\]
\end{prop}

\begin{proof}
We denote $\mu=\mu_\phi^{\beta,\rho_m,\phi}$. Let $g$ be the Green function associated with the dicrete Markov chain associated with $P^{\beta,\rho_m,\phi}$, \textit{i.e.} with jump rates $r^{\beta,\rho_m,\phi}_{i,j}=\frac{W_{i,j}}{2}\frac{G_m(\phi,j)}{G_m(\phi,i)}$. Let us denote, for $i,j\in T$, $f(i,j)=\frac{g(i,j)}{g(j,j)}$. Then from Proposition \ref{prop:mes.harm.arbre}, we get the following expression for $x\neq\phi$:
\[\mu(\Omega_x)=f(\phi,x)\frac{1-f(x,\p{x})}{1-f(x,\p{x})f(\p{x},x)}.\]

For $i,j\in T$, we have
\[g(i,j)= \sum_{k\in\N}P^{\beta,\rho_m,\phi}_i[X_k=j]=\sum_{\sigma\in\calP_{i,j}^T}\prod_{k=0}^{|\sigma|-1}\frac{r^{\beta,\rho_m,\phi}_{\sigma_k,\sigma_{k+1}}}{\tilde\beta_{\sigma_k}}= \frac{G_m(\phi,j)}{G_m(\phi,i)}\sum_{\sigma\in\calP_{i,j}^T} \frac{W_\sigma}{(2\tilde\beta)_\sigma^-},
\]
where $\tilde\beta_i=\sum_{j\sim i} r^{\beta,\rho_m,\phi}_{i,j}$ are the rates of the corresponding holding times. Note that $\tilde\beta_i=\beta_i-\ind_{\{i=\phi\}}\frac{1}{2G_m(\phi,\phi)}$ for $i\in T$. In particular, if a path $\sigma$ never crosses $\phi$, then $(2\tilde\beta)_\sigma=(2\beta)_\sigma$.

Let us denote $\Gt(i,j)=\sum_{\sigma\in\calP^T_{i,j}}\frac{W_\sigma}{(2\tilde\beta)_\sigma}$ and $\Ft(i,j)=\frac{\Gt(i,j)}{\Gt(j,j)}$, then
\[g(i,j)=\frac{G_m(\phi,j)}{G_m(\phi,i)}\Gt(i,j)2\tilde\beta_j, \mbox{ and } f(i,j)=\frac{G_m(\phi,j)}{G_m(\phi,i)}\frac{\Gt(i,j)}{\Gt(j,j)}=\frac{G_m(\phi,j)}{G_m(\phi,i)}\Ft(i,j).\]
The expression for the measure becomes
\begin{align*}
\mu(\Omega_x) &= \frac{G_m(\phi,x)}{G_m(\phi,\phi)}\frac{\Gt(\phi,x)}{\Gt(x,x)}\l(\frac{1-\frac{G_m(\phi,\p{x})}{G_m(\phi,x)}\Ft(x,\p{x})}{1-\Ft(x,\p{x})\Ft(\p{x},x)}\r)\\
&= \frac{\Gt(\phi,x)}{G_m(\phi,\phi)}\l(\frac{G_m(\phi,x)-G_m(\phi,\p{x})\Ft(x,\p{x})}{\Gt(x,x)-\Ft(x,\p{x})\Gt(\p{x},x)}\r).
\end{align*}

Let us compute the following terms: firstly,
\begin{align*}
\Gt(x,x)-\Ft(x,\p{x})\Gt(\p{x},x) &= \Gt(x,x)-\Ft(x,\p{x})\Gt(\p{x},\p{x})\Ft(x,\p{x})\\
&=\sum_{\sigma\in\calP^T_{x,x}} \frac{W_\sigma}{(2\tilde\beta)_\sigma}-\sum_{\sigma\in\calP^T_{x,\{\p{x}\},x}} \frac{W_\sigma}{(2\tilde\beta)_\sigma}=\sum_{\sigma\in\calP^{T_x}_{x,x}} \frac{W_\sigma}
{(2\tilde\beta)_\sigma}.
\end{align*}
Indeed, paths from $x$ to $x$ that do not cross $\p{x}$ have to stay in the connected component of $x$ in $T\backslash\{\p{x}\}$, which is $T_x$, \textit{i.e.} $\calP_{x,x}^T\backslash\calP_{x,\{\p{x}\},x}^T=\calP_{x,x}^{T_x}$. Moreover, $\phi\notin T_x$, so for $\sigma\in\calP_{x,x}^{T_x}$, $(2\tilde\beta)_\sigma=(2\beta)_\sigma$. As a result,
\[\Gt(x,x)-\Ft(x,\p{x})\Gt(\p{x},x) =\sum_{\sigma\in\calP^{T_x}_{x,x}} \frac{W_\sigma}{(2\beta)_\sigma}=\Gh^{T_x}(x,x).\]

Secondly,
\begin{align*}
G_m(\phi,x)-G_m(\phi,\p{x})\Ft(x,\p{x}) &= \Gh(\phi,x)-\Gh(\phi,\p{x})\Ft(x,\p{x})\\
& \qquad +\int_{\Omega^2} \chi(\phi,d\omega)\gc_m(\omega,\tau)\l(\chi(x,d\tau)-\Ft(x,\p{x})\chi(\p{x},d\tau)\r).
\end{align*}
Note that if $\sigma\in\o\calP^{T\backslash\{\p{x}\}}_{x,\p{x}}$, then $\sigma_k\in T_x$ for $k\leq |\sigma|-1$, so $(2\tilde\beta)_\sigma^-=(2\beta)_\sigma^-$. Therefore, $\Ft(x,\p{x})=\Fh(x,\p{x})$.
Moreover, since $\calP^T_{\phi,x}=\calP^T_{\phi,\{\p{x}\},x}$, we have
\[\Gh(\phi,x) -\Gh(\phi,\p{x})\Fh(x,\p{x}) =0.\]
Recall also that the density of $\chi(x,\cdot)$ with respect to $\chi(\phi,\cdot)$ is $\tau\mapsto \frac{\Fh(x,x\wedge\tau)}{\Fh(\phi,x\wedge\tau)}$. As a result,
\begin{align*}
G_m(\phi,x)-G_m&(\phi,\p{x})\Ft(x,\p{x}) = \int_{\Omega^2} \chi(\phi,d\omega)\gc_m(\omega,\tau)\l(\chi(x,d\tau)-\Fh(x,\p{x})\chi(\p{x},d\tau)\r)\\
&= \int_{\Omega^2} \chi(\phi,d\omega)\gc_m(\omega,\tau)\chi(\phi,d\tau)\l(\frac{\Fh(x,x\wedge\tau)}{\Fh(\phi,x\wedge\tau)}-\Fh(x,\p{x})\frac{\Fh(\p{x},\p{x}\wedge\tau)}{\Fh(\phi,\p{x}\wedge\tau)}\r).
\end{align*}
For $\tau\notin\Omega_x$, $x\wedge\tau=\p{x}\wedge\tau$ and paths from $x$ to $\p{x}\wedge\tau$ cross $\p{x}$. Therefore,
\begin{align*}
\frac{\Fh(x,x\wedge\tau)}{\Fh(\phi,x\wedge\tau)}-\Fh(x,\p{x})\frac{\Fh(\p{x},\p{x}\wedge\tau)}{\Fh(\phi,\p{x}\wedge\tau)} &=\frac{\Fh(x,\p{x}\wedge\tau)-\Fh(x,\p{x})\Fh(\p{x},\p{x}\wedge\tau)}{\Fh(\phi,\p{x}\wedge\tau)}=0.
\end{align*}
For $\tau\in\Omega_x$, $x\wedge\tau=x$ and $\p{x}\wedge\tau=\p{x}$, so
\begin{align*}
\frac{\Fh(x,x\wedge\tau)}{\Fh(\phi,x\wedge\tau)}-\Fh(x,\p{x})\frac{\Fh(\p{x},\p{x}\wedge\tau)}{\Fh(\phi,\p{x}\wedge\tau)}&=\frac{\Gh(x,x)}{\Gh(\phi,x)}-\Fh(x,\p{x})\frac{\Gh(\p{x},\p{x})}{\Gh(\phi,\p{x})}\\
&=\frac{\Gh(x,x)-\Fh(x,\p{x})\Gh(\p{x},\p{x})\Fh(\p{x},x)}{\Gh(\phi,x)}=\frac{\Gh^{T_x}(x,x)}{\Gh(\phi,x)}.
\end{align*}
As a result, we have
\[G_m(\phi,x)-G_m(\phi,\p{x})\Ft(x,\p{x}) = \frac{\Gh^{T_x}(x,x)}{\Gh(\phi,x)}\int_{\Omega\times\Omega_x}\chi(\phi,d\omega)\gc_m(\omega,\tau)\chi(\phi,d\tau).\]

For $x\neq\phi$, we finally get
\begin{align*}
\mu(\Omega_x)&=\frac{\Gt(\phi,x)}{G_m(\phi,\phi)\Gh(\phi,x)}\int_{\Omega\times\Omega_x}\chi(\phi,d\omega)\gc_m(\omega,\tau)\chi(\phi,d\tau)\\
&=\frac{\Gt(\phi,\phi)}{\Gh(\phi,\phi)}\int_{\Omega\times\Omega_x}\chi(\phi,d\omega)\gc_m(\omega,\tau)\chi(\phi,d\tau)
\end{align*}
since $\Ft(x,\phi)=\Fh(x,\phi)$. Moreover, by summing over $x\in S(\phi)$, we have the same expression for $\mu(\Omega_\phi)=\mu(\Omega)$:
\[1=\mu(\Omega)=\frac{\Gt(\phi,\phi)}{G_m(\phi,\phi)\Gh(\phi,\phi)}\int_{\Omega^2}\chi(\phi,d\omega)\gc_m(\omega,\tau)\chi(\phi,d\tau).\]
As a result, for all $x\in T$,
\[\mu(\Omega_x)=\frac{\int_{\Omega\times\Omega_x}\chi(\phi,d\omega)\gc_m(\omega,\tau)\chi(\phi,d\tau)}
{\int_{\Omega^2}\chi(\phi,d\omega)\gc_m(\omega,\tau)\chi(\phi,d\tau)}.\]
\end{proof}

\begin{proof}[Proof of Proposition \ref{prop:diff.rep.reg.tree}]
Let $\mathcal{T}$ be a $d$-regular tree, with $d\geq 3$, endowed with constant conductances $W$ such that $\P[\forall i\in T, \psi(i)>0]=1$. Note that $(\mathcal{T},W)$ is vertex transitive, so it is enough to show the proposition for $i_0=\phi$. The following lemma is a consequence of the symmetries of $(\mathcal{T},W)$, and guarantees that almost surely, the exiting measure gives weight to the whole boundary $\Omega$.

\begin{lem}\label{lem:chi.nonzero}
Almost surely under $\nu_V^W(d\beta)$, for all $x\neq \phi$, $\chi(\phi,\Omega_x)>0$.
\end{lem}

\begin{proof}
For all $x\neq\phi$, we define $\hat\chi_x=\psi(x)-\Fh(x,\p x)\psi(\p x)$. Then $\chi(\phi,\Omega_x)=\Fh(\phi,x)\frac{\hat\chi_x}{1-\Fh(x,\p x)\Fh(\p x, x)}$, and $\P[\chi(\phi,\Omega_x)>0]=\P[\hat\chi_x>0]$.

Note that
\[\psi^\n(x)-\Fh^\n(x,\p x)\psi^\n(\p x)=\sum_{y\in T_x\cap D^\n}\l(\sum_{\sigma\in\calP_{x,y}^{T_x^\n}}\frac{W_\sigma}{(2\beta)_\sigma}\r)\eta_y^\n\]
is $\beta_{T_x}$-measurable. Therefore, taking the limit when $n\to\infty$ shows that $\hat\chi_x$ is also $\beta_{T_x}$-measurable. As a result, given a fixed $m\leq 1$, the random variables $(\hat\chi_x)_{x\in D^\m}$ are independant, since $\nu_V^W$ is $1$-dependent, and have the same distribution, since $\nu_V^W$ is invariant under the group of automorphisms of $\mathcal{T}$.

Moreover, we have $\psi(\phi)=\sum_{y\in D^\m}\Fh(\phi,y)\frac{\hat\chi_y}{1-\Fh(y,\p y)\Fh(\p y, y)}$, so
\[\P[\psi(\phi)=0]=\P[\forall y\in D^\m, \hat\chi_y=0]=\P[\hat\chi_x=0]^{|D^\m|}\]
for any $x\in D^\m$. Since $\P[\psi(\phi)=0]=0$, we get $\P[\hat\chi_x=0]=0$ for all $x\in D^\m$ and all $m\geq 1$, which implies that almost surely, for all $x\neq \phi$, $\chi(\phi,\Omega_x)>0$.

\end{proof}

Let us fix $m>m'$, and denote $\mu^\m=\mu_\phi^{\beta,\rho_m,\phi}$ and $\mu^{(m')}=\mu_\phi^{\beta,\rho_{m'},\phi}$. Let $x\in D^\m$ be fixed, note that $x\neq\phi$. We define the following events: 
\[A_x^\m=\l\{\frac{\mu^\m(\Omega_x)}{\mu^\m(\Omega_{\p{x}})}=\frac{\chi(\phi,\Omega_x)}{\chi(\phi,\Omega_{\p{x}})}\r\} \mbox{ and } A_x^{(m')}=\l\{\frac{\mu^{(m')}(\Omega_x)}{\mu^{(m')}(\Omega_{\p{x}})}=\frac{\chi(\phi,\Omega_x)}{\chi(\phi,\Omega_{\p{x}})}\r\}.\]

Let us first show that the event $A_x^\m$ is $r^{\beta,\rho_m,\phi}$-measurable. Note that 
the exiting measure $\mu^\m$ is measurable with respect to the corresponding environment $r^{\beta,\rho_m,\phi}$. Moreover, for $i\neq\phi$, $\beta_i=\beta'_i=\sum_{j\sim i}r^{\beta,\rho_m,\phi}_{i,j}$ is $r^{\beta,\rho_m,\phi}$-measurable. Therefore, we just have to show that $\frac{\chi(\phi,\Omega_x)}{\chi(\phi,\Omega_{\p{x}})}$ is $\beta_{T\backslash\{\phi\}}$-measurable. Since $\chi(\phi,\Omega_x)=\Gh(\phi,x)\sum_{y\in S(x)}W_{x,y}\hat\chi_y$, we have
\begin{align*}
\frac{\chi(\phi,\Omega_x)}{\chi(\phi,\Omega_{\p{x}})}=\frac{\Gh(\phi,x)\sum_{y\in S(x)}W_{x,y}\hat\chi_y}{\Gh(\phi,\p{x})\sum_{z\in S(\p{x})}W_{\p x,z}\hat\chi_z}=\Fh(x,\p{x})\frac{\sum_{y\in S(x)}W_{x,y}\hat\chi_y}{\sum_{z\in S(\p{x})}W_{\p x,z}\hat\chi_z},
\end{align*}
which is $\beta_{U_{\p{x}}}$-measurable and therefore $\beta_{T\backslash\{\phi\}}$-measurable. We can conclude that $A_x^\m$ is $r^{\beta,\rho_m,\phi}$-measurable, and in the same way, $A_x^{(m')}$ is $r^{\beta,\rho_{m'},\phi}$-measurable. We are now going to show that under $\nu_{T,B_{m'}}^W(d\beta,d\rho_{m'})$ we have $\P[A_x^{(m')}]=1$, while under $\nu_{T,B_m}^W(d\beta,d\rho_m)$ we have $\P[A_x^\m]=0$. This will prove that the distributions of $r^{\beta,\rho_m,\phi}$ under $\nu_{T,B_m}^W(d\beta,d\rho_m)$ and $r^{\beta,\rho_{m'},\phi}$ under $\nu_{T,B_m}^W(d\beta,d\rho_m)$ are different.

Since $|x|=m>m'$, we have $|\p x|\geq m'$, so there exists $z\in D^{(m')}$ such that $|\p x|\in T_z$, \textit{i.e.} $\Omega_{\p x}\subset\Omega_z$. Then for all $\tau\in \Omega_{\p x}$, $\int_\Omega\chi(\phi,d\omega) \gc_{m'}(\omega,\tau)=\sum_{b\in B_{m'}}\chi_{m'}(\phi,b)\Gc_{m'}(b,\delta_z)$. As a result, 
\begin{align*}
\frac{\mu^{(m')}(\Omega_x)}{\mu^{(m')}(\Omega_{\p{x}})}&=\frac{\int_{\Omega_x}\l(\sum_{b\in B_{m'}}\chi_{m'}(\phi,b)\Gc_{m'}(b,\delta_z)\r)\chi(\phi,d\tau)}{\int_{\Omega_{\p{x}}}\l(\sum_{b\in B_{m'}}\chi_{m'}(\phi,b)\Gc_{m'}(b,\delta_z)\r)\chi(\phi,d\tau)}=\frac{\chi(\phi,\Omega_x)}{\chi(\phi,\Omega_{\p{x}})},\\
\end{align*}
so $\P[A_x^{(m')}]=1$ under $\nu_{T,B_{m'}}^W(d\beta,d\rho_{m'})$.

We will finally show that $\nu_{T,B_m}^W$-almost surely, $\P[A_x^\m|\beta]=0$. Since $|x|=m$, we have 
\begin{align*}
\mu^\m(\Omega_x)&=\chi_m(\phi,\delta_x)\sum_{b\in B_m}\Gc_m(\delta_x,b)\chi_m(\phi,b),\\
 \mbox{and } \mu^\m(\Omega_{\p x})&=\sum_{y\in S(\p x)}\chi_m(\phi,\delta_y)\sum_{b\in B_m}\Gc_m(\delta_y,b)\chi_m(\phi,b).
\end{align*}
Let us denote, for $y\in D^\m$, $u_y=\sum_{b\in B_m}\Gc_m(\delta_y,b)\chi_m(\phi,b)$. Then 
\[A_x^\m=\l\{\frac{\mu^\m(\Omega_x)}{\mu^\m(\Omega_{\p{x}})}=\frac{\chi(\phi,\Omega_x)}{\chi(\phi,\Omega_{\p{x}})}\r\} = \l\{\sum_{y\in S(\p x)}\frac{\chi(\phi,\Omega_y)}{\chi(\phi,\Omega_{\p x})}u_y = u_x\r\}=\{u\in\ker(f_\beta)\},\]
where $f_\beta: (v_y)_{y\in D^\m}\mapsto \sum_{y\in S(\p x)}\frac{\chi(\phi,\Omega_y)}{\chi(\phi,\Omega_{\p x})}v_y - v_x$ is a linear form conditionally on $\beta$, which has almost surely rank $1$ according to Lemma \ref{lem:chi.nonzero}, so that $\ker(f_\beta)$ is a hyperplane of $\R^{|D^\m|}$. Let us show that conditionally on $\beta$, the distribution of $(u_y)_{y\in D^\m}$ is absolutely continuous with respect to the Lebesgue measure on $\R^{|D^\m|}$, and therefore $\P[A_x^\m|\beta]=\P[u\in\ker(f_\beta)|\beta]=0$.

Recall that $\Gc_m=(2\rho_m-\Cc_m)^{-1}$, where conditionally on $\beta$, $\rho_m$ is distributed according to $\nu_{B_m}^{\Cc_m}$, which is absolutely continuous with respect to the Lebesgue measure on $\R^{|B_m|}=\R^{|D^\m|}$. Let us define 
\[\Phi:
\begin{array}{rcl}
    \R^{|D^\m|} & \longrightarrow &\R^{|D^\m|} \\
    \rho_m & \longmapsto & (u_y)_{y\in D^\m}=\Gc_m \chi_m(\phi,\cdot)\\
\end{array}.\]
For all $\rho_m$ such that $2\rho_m-\Cc_m>0$, $\Phi$ is differentiable, and its differential is
\[d_{\rho_m}\Phi(v)=-2\Gc_m \diag(v) \Gc_m\chi_m(\phi,\cdot)=-2\Gc_m \diag(v)u,\]
which is invertible, with $(d_{\rho_m}\Phi)^{-1}(w)=\l(-\frac{(\Gc_m^{-1} w)_y}{2u_y}\r)_{y\in D^\m}$. Note that this is well-defined since $u_y>0$ for all $y\in D^\m$, thanks to Lemma \ref{lem:chi.nonzero}. As a result, $\Phi$ is a local diffeomorphism. Therefore, the distribution of $u=\Phi(\rho_m)$, conditionally on $\beta$, admits a density with respect to the Lebesgue measure on $\R^{|D^\m|}$. We deduce that almost surely, $\P[A_x^\m|\beta]=\P[u\in\ker(f_\beta)|\beta]=0$, and therefore $\P[A_x^\m]=0$, which concludes the proof.
\end{proof}

\section*{Acknowledgements}
I would like to thank my advisor Christophe Sabot, for introducing me to this subject, and guiding me through my research and the writing of this paper. I would also like to thank Sebastian Andres, Vadim Kaimanovich and Pierre Mathieu for useful information regarding Martin boundaries for random walks in random environment.

\bibliographystyle{plain}
\bibliography{biblio}

\end{document}